\documentclass[10pt,reqno]{amsart}

\usepackage{a4wide}
\usepackage[T1]{fontenc}
\usepackage{lmodern}
\usepackage{amsmath,amssymb,amsthm,amscd,mathrsfs,mathtools,bm}
\usepackage{graphicx}
\usepackage{indentfirst}
\usepackage{enumitem}
\usepackage[numbers,sort&compress]{natbib}
\usepackage[dvipsnames]{xcolor}
\usepackage[colorlinks=true,linkcolor=blue,citecolor=blue,urlcolor=blue]{hyperref}
\usepackage{aliascnt}
\usepackage[nameinlink,capitalise]{cleveref}

\numberwithin{equation}{section}
\allowdisplaybreaks[2]
\theoremstyle{plain}
\newtheorem{theorem}{Theorem}[section]

\newaliascnt{lemma}{theorem}
\newtheorem{lemma}[lemma]{Lemma}
\aliascntresetthe{lemma}

\newaliascnt{corollary}{theorem}
\newtheorem{corollary}[corollary]{Corollary}
\aliascntresetthe{corollary}

\newaliascnt{proposition}{theorem}
\newtheorem{proposition}[proposition]{Proposition}
\aliascntresetthe{proposition}

\theoremstyle{definition}
\newaliascnt{definition}{theorem}
\newtheorem{definition}[definition]{Definition}
\aliascntresetthe{definition}

\theoremstyle{remark}
\newaliascnt{remark}{theorem}
\newtheorem{remark}[remark]{Remark}
\aliascntresetthe{remark}

\crefname{theorem}{Theorem}{Theorems}
\Crefname{theorem}{Theorem}{Theorems}
\crefname{lemma}{Lemma}{Lemmas}
\Crefname{lemma}{Lemma}{Lemmas}
\crefname{proposition}{Proposition}{Propositions}
\Crefname{proposition}{Proposition}{Propositions}
\crefname{corollary}{Corollary}{Corollaries}
\Crefname{corollary}{Corollary}{Corollaries}
\crefname{definition}{Definition}{Definitions}
\Crefname{definition}{Definition}{Definitions}
\crefname{remark}{Remark}{Remarks}
\Crefname{remark}{Remark}{Remarks}

\newcommand{\R}{\mathbb R}

\newcommand{\dd}{\,\mathrm d}

\newcommand{\dist}{\operatorname{dist}}
\newcommand{\tr}{\operatorname{tr}}
\newcommand{\II}{\mathrm{II}}

\title[Green kernels for inverse spectral powers and Robin nondegeneracy]
{Convexity of Green kernels for inverse spectral powers and nondegeneracy of Robin centers in convex domains}
\author[X. Liang, W. Wang]{Xiuda Liang, Wenjie Wang}

\address[Xiuda Liang]{School of Mathematics and Statistics, Central China Normal University, Wuhan 430079, China}
\email{lxddd@mails.ccnu.edu.cn}

\address[Wenjie Wang]{School of Mathematics and Statistics, Central China Normal University, Wuhan 430079, China}
\email{wjwang3269@mails.ccnu.edu.cn}

\date{}

\begin{document}

\begin{abstract}
Let $\Omega\subset\R^N$ be a bounded convex domain, and let $A=-\Delta_D$
be the Dirichlet Laplacian.  For every $s>0$ with $N>2s$, we prove
that the Green-kernel transform
$(x,y)\mapsto K_{s,\Omega}(x,y)^{-1/(N-2s)}$ extends continuously
by zero to the diagonal and is jointly convex on $\Omega\times\Omega$.  We prove that the regular part extends real analytically across the
diagonal and that the associated Robin function is strictly convex and
diverges to $+\infty$ at the boundary.  Hence it has a unique critical point.  If $\Omega$
is of class $C^{2,\vartheta}$, this point is nondegenerate for every
$s>0$ with $N>2s$.  Moreover, the Hessian of the Robin function is
positive definite throughout $\Omega$ when $0<s<1$ or $N>3s$.
The arguments combine convexity inequalities for Dirichlet heat kernels,
a weighted boundary identity at fractional orders, and a finite-part
transfer of Hessian positivity between spectral powers.  In particular,
the results apply to spectral fractional Dirichlet Laplacians and Navier
polyharmonic operators.
\end{abstract}

\maketitle

{\small
\smallskip

}

\section{Introduction and main results}
\setcounter{equation}{0}

The Robin function is a classical object in potential theory and plays a
central role in the analysis of singularly perturbed elliptic equations.
Let $\Omega\subset\mathbb R^N$ be a bounded domain, where $N\ge2$.
We adopt the convention that the fundamental solution of $-\Delta$ in
$\mathbb R^N$ is
\[
 \Phi_1(z)=
 \begin{cases}
 -\dfrac{1}{2\pi}\log|z|, & N=2,\\[2mm]
 \dfrac{1}{(N-2)|\mathbb S^{N-1}|}|z|^{2-N}, & N>2,
 \end{cases}
 \qquad z\in\mathbb R^N\setminus\{0\},
\]
where $|\mathbb S^{N-1}|$ denotes the surface measure of the unit sphere.
If $G_{1,\Omega}$ is the Dirichlet Green function of $-\Delta$, its regular
part $H_{1,\Omega}$ and the associated Robin function $R_{1,\Omega}$ are
defined by
\[
 G_{1,\Omega}(x,y)=\Phi_1(x-y)-H_{1,\Omega}(x,y),\qquad x\ne y,
 \quad
 R_{1,\Omega}(y)=H_{1,\Omega}(y,y).
\]
The critical points of the Robin function arise in the description of
concentration points for various singularly perturbed elliptic problems;
see, for example,
\cite{Rey1990,Han1991,BahriLiRey1995,BandleFlucher1996,
BartschPistoia2015,MichelettiPistoia2014,BartschMichelettiPistoia2019,
GladialiGrossiLuoYan2025}.  In convex domains, the Robin function
satisfies important convexity properties.  In dimension two,
Caffarelli--Friedman \cite{CaffarelliFriedman1985} proved its strict
convexity.  In dimensions $N\ge3$, Cardaliaguet--Tahraoui
\cite{CardaliaguetTahraoui} proved strict concavity of the harmonic radius,
which in turn yields strict convexity of the Robin function and uniqueness of
the harmonic center.  More recently, Li--Liu--Ma \cite{LiLiuMa} proved
global positive definiteness of the Hessian on smooth convex domains by means
of a translation--curvature identity.  Thus, for smooth bounded convex domains, the classical Robin function
has a unique nondegenerate critical point.

The corresponding results for higher-order operators are less complete.
For $p\in\mathbb N$, the Navier polyharmonic problem is associated with the
iterated Dirichlet Laplacian $(-\Delta_D)^p$ and, for sufficiently regular
solutions, with the boundary conditions
\begin{equation}\label{eq:Navier-problem-intro}
 (-\Delta)^pu=f\quad\text{in }\Omega,
 \qquad
 (-\Delta)^ju=0\quad\text{on }\partial\Omega,
 \quad j=0,1,\ldots,p-1.
\end{equation}
The corresponding Green and Robin functions arise naturally in the blow-up
analysis of nearly critical biharmonic and polyharmonic equations; see
\cite{ChouGeng2000,Geng2005,Takahashi2008,TakahashiNondeg2008,
SatoTakahashi2009}.  For $p\ge2$, the convexity and nondegeneracy properties of the
corresponding Robin functions are less well understood.  Takahashi \cite{Takahashi2013}
proved nondegeneracy at the origin on domains symmetric with respect to
the coordinate hyperplanes and convex in each coordinate direction, in
dimensions $N\ge2p$.  He also explicitly raised the analogous question
for general convex domains.  Our results establish uniqueness and
nondegeneracy on bounded convex domains without symmetry assumptions in
the range $N>2p$; the logarithmic case $N=2p$ is not considered here.

A parallel question concerns fractional powers of the Dirichlet Laplacian.
In the spectral setting, $(-\Delta_D)^s$, $0<s<1$, is defined by functional
calculus from the Dirichlet Laplacian, and its Green kernel is represented by
the Dirichlet heat semigroup.  This realization must be distinguished from
the restricted fractional Laplacian, for which the homogeneous condition is
imposed on $\mathbb R^N\setminus\Omega$: the two operators have different
Green kernels, boundary behavior and Robin functions.  Spectral fractional
Green and Robin functions enter the reduction of nearly critical problems;
see D\'avila--L\'opez R\'ios--Sire \cite{DavilaLopezRiosSire2017}.  More
recently, Ortega \cite{Ortega2024,Ortega2026} obtained nondegeneracy and
directional Hessian information at critical points selected by symmetry.
Simultaneous-translation fields and first-variation formulas also appear
in \cite{Ortega2026}.
These results do not establish global convexity of the spectral fractional
Robin function on arbitrary convex domains.  Related identities
for the restricted fractional Laplacian have been studied in
\cite{DjitteSueur2026}.  The present paper deals exclusively with the
spectral realization.

We investigate whether the Robin function associated with a spectral power
of the Dirichlet Laplacian has a unique critical point on a convex domain and
whether this critical point is nondegenerate.  Our results apply to all real
spectral orders $s>0$ in the range $N>2s$ and cover both the spectral
fractional and Navier polyharmonic settings.

To state the results, let
\[
 A:=-\Delta_D
\]
be the positive self-adjoint Dirichlet Laplacian on $L^2(\Omega)$, and let
$\{(\lambda_k,\phi_k)\}_{k\ge1}$ be an orthonormal basis of eigenfunctions,
with $0<\lambda_1\le\lambda_2\le\cdots$.  For every real $s>0$ we define
\[
 D(A^s):=\left\{u\in L^2(\Omega):
 \sum_{k=1}^\infty\lambda_k^{2s}|\langle u,\phi_k\rangle|^2<\infty\right\},
 \qquad
 A^su:=\sum_{k=1}^\infty
 \lambda_k^s\langle u,\phi_k\rangle\phi_k,
\]
and
\[
 A^{-s}f:=\sum_{k=1}^\infty
 \lambda_k^{-s}\langle f,\phi_k\rangle\phi_k.
\]
We denote by $K_{s,\Omega}$ the Green kernel associated with $A^s$,
equivalently the integral kernel of $A^{-s}$.  It has the representation
\begin{equation}\label{eq:Ks-mellin-definition-intro}
 K_{s,\Omega}(x,y)
 =\frac1{\Gamma(s)}\int_0^\infty
 t^{s-1}p_\Omega(t,x,y)\,\dd t,
 \qquad x\ne y,
\end{equation}
where $p_\Omega$ is the Dirichlet heat kernel.  Thus
\begin{equation}\label{eq:spectral-semigroup-intro}
 A^{-s}A^{-t}=A^{-(s+t)},
 \qquad s,t>0.
\end{equation}
For $0<s<1$, $A^s$ is the spectral fractional Dirichlet Laplacian described
above.  If $s=p\in\mathbb N$ and $u\in D(A^p)$, then
$A^ju\in D(A)\subset H_0^1(\Omega)$ for $j=0,\ldots,p-1$.
These zero traces give the Navier conditions
\eqref{eq:Navier-problem-intro} whenever the functions involved are
sufficiently regular.  For noninteger $s>1$, we use the term
\emph{spectral power}, with the operator defined by the same functional
calculus and without assigning it an additional local boundary value problem.

Assume now $s>0$ and $N>2s$.  In $\mathbb R^N$, the inverse fractional power
$(-\Delta)^{-s}$ is represented by the Riesz potential kernel
\begin{equation}\label{eq:real-fundamental-intro}
 \Phi_s(z):=a_{N,s}|z|^{2s-N},
 \qquad
 a_{N,s}:=
 \frac{\Gamma(\frac N2-s)}{4^s\pi^{N/2}\Gamma(s)}.
\end{equation}
For $x\ne y$ we define the regular part
\[
 H_{s,\Omega}(x,y):=\Phi_s(x-y)-K_{s,\Omega}(x,y).
\]
We prove below that $H_{s,\Omega}$ admits a unique real-analytic extension
across the diagonal.  With this extension understood, we make the following
definition.
\begin{definition}
\label{def:spectral-robin-real}
For $s>0$ with $N>2s$, define
\[
 R_{s,\Omega}(y):=H_{s,\Omega}(y,y).
\]
If $0<s<1$, this is the \emph{spectral fractional Robin function}.  If
$s=p\in\mathbb N$, then $K_{p,\Omega}$ is the Navier polyharmonic Green
kernel and $R_{p,\Omega}$ is the \emph{Navier polyharmonic Robin function}.
On a smooth domain, $K_{p,\Omega}(\cdot,y)$ is equivalently characterized by
\[
 (-\Delta)^pK_{p,\Omega}(\cdot,y)=\delta_y\quad\text{in }\Omega,
 \qquad
 (-\Delta)^jK_{p,\Omega}(\cdot,y)=0\quad\text{on }\partial\Omega,
 \quad j=0,\ldots,p-1.
\]
For noninteger $s>1$, we call $R_{s,\Omega}$ the Robin function associated
with the spectral power $A^s$.
\end{definition}

Our first theorem extends Borell's joint convexity result for the
Dirichlet Green function \cite{Borell1984}, corresponding to $s=1$,
to arbitrary real orders in the range $N>2s$.  This result is also related to the Brunn--Minkowski
theory for Greenian potentials developed in \cite{Borell1985}.

\begin{theorem}
\label{thm:spectral-main}
Let $\Omega\subset\mathbb R^N$ be a bounded open convex set, let $s>0$, and
assume $N>2s$.  Define
\begin{equation}\label{eq:spectral-convex-transform}
 \Psi_{s,\Omega}(x,y):=
 \begin{cases}
 K_{s,\Omega}(x,y)^{-1/(N-2s)},&x\ne y,\\[1mm]
 0,&x=y.
 \end{cases}
\end{equation}
Then $\Psi_{s,\Omega}$ is continuous and jointly convex on
$\Omega\times\Omega$.
\end{theorem}

As a consequence of \cref{thm:spectral-main}, we obtain strict convexity
of the associated Robin function.

\begin{theorem}
\label{thm:robin-real-main}
Let $s>0$, assume $N>2s$, and let
$\Omega\subset\mathbb R^N$ be a bounded open convex set.  Then
$H_{s,\Omega}$ extends uniquely to a real-analytic function on
$\Omega\times\Omega$, and $R_{s,\Omega}$ is strictly convex on $\Omega$.
Moreover,
\begin{equation}\label{eq:robin-real-bound-intro}
 R_{s,\Omega}(y)
 \ge a_{N,s}[2\,\dist(y,\partial\Omega)]^{2s-N},
 \qquad y\in\Omega.
\end{equation}
Consequently,
\[
 R_{s,\Omega}(y)\longrightarrow+\infty
 \qquad\text{as }y\to\partial\Omega,
\]
and there exists a unique point $y_{s,\Omega}\in\Omega$ such that
$\nabla R_{s,\Omega}(y_{s,\Omega})=0$.  This point is the unique global
minimizer of $R_{s,\Omega}$.
\end{theorem}

This theorem requires no boundary smoothness.  For $s=1$ and $N\ge3$,
it recovers the uniqueness of the classical Robin center.  Specializing
to fractional and integer spectral orders gives the following corollaries.

\begin{corollary}
\label{cor:fractional-robin-center}
Let $0<s<1$, $N>2s$, and let $\Omega\subset\mathbb R^N$ be a bounded open
convex set.  Then the spectral fractional Robin function is real analytic and
strictly convex, diverges to $+\infty$ at the boundary, and has a unique
critical point, which is its unique global minimizer.
\end{corollary}

\begin{corollary}
\label{cor:navier-robin-center}
Let $p\in\mathbb N$, $N>2p$, and let $\Omega\subset\mathbb R^N$ be a bounded
open convex set.  Then the Navier polyharmonic Robin function is real
analytic and strictly convex, diverges to $+\infty$ at the boundary, and has
a unique critical point $y_{p,\Omega}$, which is its unique global minimizer.
\end{corollary}

Strict convexity does not by itself imply nondegeneracy.  To obtain
second-order information, we first prove global Hessian positivity in the
fractional range.

\begin{theorem}
\label{thm:fractional-hessian-intro}
Let $0<s<1$, $N\ge2$, $N>2s$, $0<\vartheta<1$, and let
$\Omega\subset\mathbb R^N$ be a bounded convex domain of class
$C^{2,\vartheta}$.  Then
\[
 D^2R_{s,\Omega}(y)>0
 \qquad\text{for every }y\in\Omega.
\]
In particular, the unique spectral fractional Robin center is nondegenerate
and has Morse index zero.
\end{theorem}

We next establish nondegeneracy of the unique critical point for general
real spectral orders.

\begin{theorem}
\label{thm:real-order-nondegeneracy-intro}
Let $s>0$, $N\ge2$, $N>2s$, $0<\vartheta<1$, and let
$\Omega\subset\mathbb R^N$ be a bounded convex domain of class
$C^{2,\vartheta}$.  Let $y_{s,\Omega}$ be the unique Robin center supplied
by \cref{thm:robin-real-main}.  Then
\begin{equation}\label{eq:real-order-center-positive}
 D^2R_{s,\Omega}(y_{s,\Omega})>0.
\end{equation}
Thus the center is nondegenerate and has Morse index zero.  If, in addition,
$0<s<1$ or $N>3s$, then
\begin{equation}\label{eq:real-order-global-positive}
 D^2R_{s,\Omega}(y)>0
 \qquad\text{for every }y\in\Omega.
\end{equation}
\end{theorem}

For $s=p\in\mathbb N$, this theorem gives a unique nondegenerate Navier
polyharmonic Robin center whenever $N>2p$, and a globally positive definite
Hessian if $N>3p$.  The assumption that $\partial\Omega$ is of class $C^{2,\vartheta}$ does not
depend on the spectral order.  The condition $N>3s$ arises from the argument used to
obtain global Hessian positivity; we make no claim that it is optimal.

The proofs are based on three main ingredients.  First, the heat-kernel
log-concavity of Brascamp--Lieb \cite{BrascampLieb}, combined with the
Borell--Brascamp--Lieb inequality, yields a convexity theorem for the
Green kernels at arbitrary real spectral orders.  For noninteger powers,
the iterated Green-operator argument available at positive integer orders
does not apply.  Instead, the proof uses parabolic scaling and the
heat-semigroup representation.  Second, a near-diagonal expansion of the
Green kernel yields convexity of the Robin function.  Analyticity of its
regular part, together with boundary divergence, implies strict convexity
and uniqueness of the critical point.

For the Hessian estimates, we develop the translation--curvature method of
Li--Liu--Ma \cite{LiLiuMa} in the weighted Stinga--Torrea extension
\cite{StingaTorrea2010}.  The essential issue is to justify the conormal
and boundary terms for the degenerate extension under only
$C^{2,\vartheta}$ regularity.  Third, the spectral composition identity
$A^{-\sigma}A^{-\sigma}=A^{-2\sigma}$ is used to derive a finite-part
formula relating Robin functions of orders $\sigma$ and $2\sigma$.
A two-term expansion with a $C^2$ remainder is used to transfer positive
definiteness from one spectral order to the next.
At critical points this transfer works under $N>4\sigma$, whereas the
global Hessian conclusion requires the stronger condition $N>6\sigma$.
This distinction explains the different dimensional ranges in
\cref{thm:real-order-nondegeneracy-intro}.

The classical order $s=1$ belongs to this framework for $N\ge3$.
The planar logarithmic case lies outside the present range, and the doubling
argument does not recover global Hessian positivity at $s=1$, $N=3$;
that conclusion is included in the classical theorem of Li--Liu--Ma
\cite{LiLiuMa}.

As an application, consider the biharmonic case.  For $p=2$ and
$N\ge5$, the above results give a unique nondegenerate biharmonic Robin
center on smooth bounded convex domains.  Existing results on
single-bubble concentration for nearly critical Navier problems
\cite{BenAyedElMehdi2006,ElMehdi2006,TakahashiNondeg2008,DengYu2023}
can therefore be applied without separately assuming nondegeneracy of
the corresponding Robin critical point.

The paper is organized as follows.  \cref{sec:spectral-BM} proves the
Green-kernel convexity theorem for arbitrary real spectral orders.  \cref{sec:robin-transfer}
develops the Robin theory for arbitrary real spectral orders and proves
\cref{thm:robin-real-main}.  \cref{sec:fractional-hessian} is devoted
to the fractional translation--curvature identity and the proof of
\cref{thm:fractional-hessian-intro}.  \cref{sec:real-order-doubling}
develops the finite-part doubling principle and proves
\cref{thm:real-order-nondegeneracy-intro}.  The boundary heat-kernel derivative estimate
used in the fractional argument is proved in
Appendix~\ref{app:boundary-heat}.

\section{Green-kernel convexity for inverse spectral powers}
\label{sec:spectral-BM}

The proof of \cref{thm:spectral-main} uses a log-concavity inequality
for Dirichlet heat kernels on Minkowski combinations of convex domains,
parabolic scaling, and the one-dimensional Borell--Brascamp--Lieb inequality.

We first record the pointwise Dynkin--Hunt identity for the Dirichlet heat
kernel.  It will
be used both in the diagonal analysis below and later in the fractional
Robin regularity argument.

\begin{lemma}\label{lem:pointwise-hunt}
Let $A=-\Delta_D$ be the Friedrichs realization of the Dirichlet Laplacian on
$L^2(\Omega)$, let $p_\Omega$ denote the continuous symmetric kernel of
$e^{-tA}$, and let $(B_t)_{t\ge0}$ be Brownian motion with generator $\Delta$.
If
\[
 \tau_\Omega:=\inf\{t>0:B_t\notin\Omega\},
\]
then for every $t>0$ and every $x,y\in\Omega$,
\begin{equation}\label{eq:not-feeling-hunt}
 0\le g_t(x-y)-p_\Omega(t,x,y)
 =\mathbb E_x\!\left[
   \mathbf 1_{\{\tau_\Omega<t\}}
   g_{t-\tau_\Omega}(B_{\tau_\Omega}-y)
 \right].
\end{equation}
\end{lemma}

\begin{proof}
The Dirichlet semigroup generated by the Friedrichs realization is the
semigroup of Brownian motion killed upon leaving $\Omega$; equivalently, its
continuous symmetric kernel is the killed Brownian heat kernel.  In the
normalization with generator $\frac12\Delta$, the corresponding pointwise
Brownian-bridge representation is \cite[Theorem~2.1]{Grillo1997}; normalizing
the conditional Wiener measure to a probability measure and replacing
time by $2t$ gives the present convention, with the free heat kernel as
prefactor.

Let $f\in C_c^\infty(\Omega)$ and extend it by zero to $\R^N$, denoting the
extension by $\widetilde f$.  The killed-semigroup identity gives
\[
 \int_\Omega p_\Omega(t,x,y)f(y)\,\dd y
 =\mathbb E_x\!\left[\widetilde f(B_t);\ t<\tau_\Omega\right].
\]
Applying the strong Markov property at $\tau_\Omega$ to the complementary
event yields
\begin{align*}
 \int_\Omega\bigl[g_t(x-y)-p_\Omega(t,x,y)\bigr]f(y)\,\dd y
 &=\mathbb E_x\!\left[
   \mathbf1_{\{\tau_\Omega<t\}}
   P_{t-\tau_\Omega}\widetilde f(B_{\tau_\Omega})
 \right]\\
 &=\int_\Omega
   \mathbb E_x\!\left[
   \mathbf1_{\{\tau_\Omega<t\}}
   g_{t-\tau_\Omega}(B_{\tau_\Omega}-y)
   \right]f(y)\,\dd y,
\end{align*}
where Tonelli's theorem is applied after taking absolute values.  The event $\{\tau_\Omega=t\}$ contributes nothing because
$B_{\tau_\Omega}\in\partial\Omega$ by path continuity and
$\widetilde f=0$ on $\partial\Omega$.  Thus \eqref{eq:not-feeling-hunt}
holds for almost every $y$.

To extend this identity from almost every $y$ to every $y$, fix
$y_0\in\Omega$ and
choose $0<\rho<\frac12\dist(y_0,\partial\Omega)$.  For
$y\in\overline{B_\rho(y_0)}$ and on $\{\tau_\Omega<t\}$,
\[
 |B_{\tau_\Omega}-y|
 \ge \dist(\overline{B_\rho(y_0)},\partial\Omega)=:d_0>0.
\]
Hence
\[
 \sup_{\substack{0<u\le t\\ |z-y|\ge d_0}}g_u(z-y)<\infty,
\]
so dominated convergence shows that the expectation on the right-hand side of
\eqref{eq:not-feeling-hunt} is continuous in $y$ near $y_0$.  The left-hand
side is continuous as well, because $g_t$ and the Dirichlet heat kernel
$p_\Omega(t,x,\cdot)$ are continuous.  Since the two continuous functions are
equal almost everywhere, they are equal everywhere.  Nonnegativity follows
from the same formula.
\end{proof}

\begin{lemma}
\label{lem:Ks-basic}
Let $s>0$ and $N>2s$.  Then
\[
 0<K_{s,\Omega}(x,y)<\infty,
 \qquad x\ne y,
\]
and
\[
 K_{s,\Omega}\in
 C\bigl((\Omega\times\Omega)\setminus\operatorname{Diag}\bigr),
 \qquad
 \operatorname{Diag}:=\{(x,x):x\in\Omega\}.
\]
Moreover, locally uniformly for $y$ in compact subsets of $\Omega$,
\begin{equation}\label{eq:Ks-diagonal-asymptotic}
 K_{s,\Omega}(x,y)
 =c_{N,s}|x-y|^{2s-N}\bigl(1+o(1)\bigr)
 \qquad\text{as }x\to y,
\end{equation}
where
\[
 c_{N,s}=\frac{\Gamma(\frac N2-s)}{4^s\pi^{N/2}\Gamma(s)}.
\]
Consequently, the function
\[
 (x,y)\longmapsto
 \begin{cases}
 K_{s,\Omega}(x,y)^{-1/(N-2s)},&x\ne y,\\
 0,&x=y,
 \end{cases}
\]
is continuous on $\Omega\times\Omega$.
\end{lemma}

\begin{proof}
The Dirichlet heat kernel satisfies the Gaussian domination
\[
 0<p_\Omega(t,x,y)\le g_t(x-y)
 :=(4\pi t)^{-N/2}\exp\!\left(-\frac{|x-y|^2}{4t}\right).
\]
For $x\ne y$ this gives integrability of
$t^{s-1}p_\Omega(t,x,y)$ near $t=0$ because of the exponential factor, and
integrability near $+\infty$ because $s<N/2$.  Positivity follows from the
positivity of the Dirichlet heat kernel on the convex, hence connected,
domain $\Omega$.

We first prove continuity away from the diagonal.  Let
$Q\Subset(\Omega\times\Omega)\setminus\operatorname{Diag}$ and set
\[
 \delta_Q:=\min_{(x,y)\in Q}|x-y|>0.
\]
For each fixed $t>0$, the map $(x,y)\mapsto p_\Omega(t,x,y)$ is continuous
in the interior.  Moreover, uniformly for $(x,y)\in Q$,
\[
 t^{s-1}p_\Omega(t,x,y)
 \le
 \begin{cases}
 (4\pi)^{-N/2}
 t^{s-1-N/2}\exp\!\left(-\dfrac{\delta_Q^2}{4t}\right),&0<t\le1,\\[2mm]
 (4\pi)^{-N/2}t^{s-1-N/2},&t\ge1.
 \end{cases}
\]
The first function is integrable on $(0,1]$ because of the exponential
factor, while the second is integrable on $[1,\infty)$ precisely because
$s<N/2$.  Dominated convergence in the representation
\[
 K_{s,\Omega}(x,y)
 =\frac1{\Gamma(s)}\int_0^\infty
 t^{s-1}p_\Omega(t,x,y)\,\dd t
\]
therefore shows that $K_{s,\Omega}$ is continuous on $Q$.  Since $Q$ was an arbitrary compact subset of the off-diagonal region,
$K_{s,\Omega}\in C((\Omega\times\Omega)\setminus\operatorname{Diag})$.

We justify the diagonal asymptotic without using any boundary smoothness.
Fix $K\Subset\Omega$ and choose $K_1\Subset\Omega$ containing a closed
neighborhood of $K$.  For $y\in K$ and $x$ sufficiently close to $y$, both
$x$ and $y$ belong to $K_1$.  Put
$d_{K_1}:=\dist(K_1,\partial\Omega)>0$ and let $\tau_\Omega$ be the first exit
time of the Brownian motion in \cref{lem:pointwise-hunt}.  By
\eqref{eq:not-feeling-hunt}, on $\{\tau_\Omega<t\}$ one has
$B_{\tau_\Omega}\in\partial\Omega$ and hence
$|B_{\tau_\Omega}-y|\ge d_{K_1}$.  Therefore there exist
$t_{K_1},c_{K_1},C_{K_1}>0$ such that, for $0<t\le t_{K_1}$ and
$x,y\in K_1$,
\begin{equation}\label{eq:not-feeling-boundary}
 0\le g_t(x-y)-p_\Omega(t,x,y)
 \le C_{K_1}t^{-N/2}e^{-c_{K_1}/t}.
\end{equation}
Indeed, for $t_{K_1}$ sufficiently small the function
$u\mapsto u^{-N/2}e^{-d_{K_1}^2/(4u)}$ is increasing on
$(0,t_{K_1}]$, so the right-hand side of \eqref{eq:not-feeling-hunt} is
bounded by its value at $u=t$ up to a dimensional constant.

Since $s<N/2$, \eqref{eq:not-feeling-boundary} for small $t$ and the Gaussian
domination for $t\ge t_{K_1}$ show that
\[
 E_{s,K_1}(x,y):=
 \frac1{\Gamma(s)}\int_0^\infty
 t^{s-1}\bigl[g_t(x-y)-p_\Omega(t,x,y)\bigr] \,\dd t
\]
is uniformly bounded for $x,y\in K_1$.  On the other hand, the whole-space
integral is explicit:
\[
 \frac1{\Gamma(s)}\int_0^\infty t^{s-1}g_t(x-y)\,\dd t
 =c_{N,s}|x-y|^{2s-N},
\]
where
\[
 c_{N,s}=\frac{\Gamma(\frac N2-s)}{4^s\pi^{N/2}\Gamma(s)}.
\]
Consequently
\[
 K_{s,\Omega}(x,y)
 =c_{N,s}|x-y|^{2s-N}-E_{s,K_1}(x,y),
 \qquad x,y\in K_1,\ x\ne y.
\]
Because $2s-N<0$ and $E_{s,K_1}$ is uniformly bounded, this proves
\eqref{eq:Ks-diagonal-asymptotic} uniformly for $y\in K$ as $x\to y$.
Hence
\[
 K_{s,\Omega}(x,y)^{-1/(N-2s)}\longrightarrow0
 \qquad\text{as }x\to y,
\]
locally uniformly with respect to the pole $y$.  Combined with the off-diagonal continuity
proved above, this yields the asserted continuity of the inverse-power
transform on all of $\Omega\times\Omega$.
\end{proof}

Let
\[
 g_h(z)=(4\pi h)^{-N/2}\exp\!\left(-\frac{|z|^2}{4h}\right)
\]
be the free heat kernel.  For a convex open set $D\subset\R^N$, $M\in\mathbb N$
with $h=t/M$, and $x,y\in D$, set
\begin{equation}\label{eq:discrete-kernel}
 K_{D,M}(t;x,y)
 :=\int_{D^{M-1}}\prod_{j=0}^{M-1}g_h(z_{j+1}-z_j)
 \,\dd z_1\cdots\dd z_{M-1},
\end{equation}
where $z_0=x$ and $z_M=y$.

\begin{lemma}
\label{lem:dyadic-killed-kernel}
Let $D\subset\R^N$ be a convex open set, let $x,y\in D$, and let $t>0$.
For the finite-step kernel \eqref{eq:discrete-kernel}, one has
\begin{equation}\label{eq:dyadic-killed-limit}
 K_{D,2^k}(t;x,y)\longrightarrow p_D(t,x,y)
 \qquad\text{as }k\to\infty.
\end{equation}
\end{lemma}

\begin{proof}
We use Brownian motion with generator $\Delta$, so that its transition density
is $g_t$.  The Brownian-bridge representation of killed heat kernels used below
is classical.  In the normalization with generator $\frac12\Delta$ it is
stated, for example, in \cite[Theorem~2.1]{Grillo1997}.  We normalize its
conditional Wiener measure by its total mass, the free heat kernel, and
replace the time parameter by $2t$.  Let $\mathbb P_x$
denote the law of our Brownian motion started from $x$, and let
$\mathbb P^t_{x,y}$ denote the corresponding Brownian-bridge law from $x$ to
$y$ in time $t$.  The finite-dimensional density of the bridge gives
\begin{equation}\label{eq:bridge-discrete-representation}
 K_{D,2^k}(t;x,y)
 =g_t(y-x)\,\mathbb P^t_{x,y}(E_k),
\end{equation}
where
\begin{equation*}
 E_k:=\left\{B_{jt/2^k}\in D:\ j=1,\ldots,2^k-1\right\}.
\end{equation*}
The dyadic grids are nested, hence $E_{k+1}\subset E_k$, and therefore
\begin{equation}\label{eq:Ek-limit-prob}
 \mathbb P^t_{x,y}(E_k)\downarrow
 \mathbb P^t_{x,y}\!\left(\bigcap_{k\ge1}E_k\right).
\end{equation}

We next identify the limiting event.  Let
\[
 E_{\overline D}:=\{B_\tau\in\overline D\text{ for every }0\le\tau\le t\}.
\]
If a continuous path belongs to every $E_k$, then density of the dyadic times
implies that the whole path lies in $\overline D$.  Conversely, a path in
$E_{\overline D}$ belongs to every $E_k$ unless it hits $\partial D$ at a
dyadic time.  For each fixed $q\in(0,t)$, the Brownian-bridge marginal has the
Lebesgue density
\begin{equation*}
 z\longmapsto
 \frac{g_q(z-x)g_{t-q}(y-z)}{g_t(y-x)}.
\end{equation*}
Since the boundary of a convex open set has $N$-dimensional Lebesgue measure
zero, and the set of dyadic times is countable,
\[
 \mathbb P^t_{x,y}
 \{B_q\in\partial D\text{ for some dyadic }q\in(0,t)\}=0.
\]
Consequently
\begin{equation}\label{eq:dyadic-closure-event}
 \bigcap_{k\ge1}E_k=E_{\overline D}
 \qquad\mathbb P^t_{x,y}\text{-a.s.}
\end{equation}

To replace $\overline D$ by $D$, note the following pathwise
regularity property of unconditioned Brownian motion.  If $z\in\partial D$, convexity
provides a supporting unit normal $\nu$ such that
\[
 \overline D\subset\{w:(w-z)\cdot\nu\le0\}.
\]
For Brownian motion started at $z$, the projection
$(B_s-z)\cdot\nu$ is a one-dimensional Brownian motion (up to the variance normalization corresponding to the generator $\Delta$).  A one-dimensional
Brownian motion started at $0$ almost surely takes positive values in every
interval $(0,\varepsilon)$.  Hence Brownian motion started from any
$z\in\partial D$ exits $\overline D$ immediately with probability one.

Now fix $T<t$.  On $\mathcal F_T$, Brownian bridge and Brownian motion are
mutually absolutely continuous; more precisely,
\begin{equation}\label{eq:bridge-RN}
 \frac{\dd\mathbb P^t_{x,y}}{\dd\mathbb P_x}
 \Big|_{\mathcal F_T}
 =\frac{g_{t-T}(y-B_T)}{g_t(y-x)}>0.
\end{equation}
Let $A_T$ be the event that the path touches $\partial D$ at some time
strictly before $T$ while remaining in $\overline D$ throughout $[0,T]$.
For Brownian motion, the strong Markov property at the first boundary hitting
time and the immediate-exit property established above give $\mathbb P_x(A_T)=0$.
By \eqref{eq:bridge-RN}, also
$\mathbb P^t_{x,y}(A_T)=0$.  Taking the countable union over
$T_n=t(1-1/n)$, $n\ge2$, shows that, under the bridge law, a path which stays in
$\overline D$ cannot touch $\partial D$ at any time strictly before $t$, except on a null set.  Since $y=B_t\in D$, continuity rules out a first boundary
touch at time $t$.  Therefore
\begin{equation}\label{eq:bridge-open-closure}
 \mathbb P^t_{x,y}(E_{\overline D})
 =
 \mathbb P^t_{x,y}
 \{B_\tau\in D\text{ for every }0<\tau<t\}.
\end{equation}

Finally, the killed-heat-kernel/bridge identity cited above gives
\begin{equation}\label{eq:killed-bridge}
 p_D(t,x,y)
 =g_t(y-x)\,
 \mathbb P^t_{x,y}
 \{B_\tau\in D\text{ for every }0<\tau<t\}.
\end{equation}
Combining \eqref{eq:bridge-discrete-representation},
\eqref{eq:Ek-limit-prob}, \eqref{eq:dyadic-closure-event},
\eqref{eq:bridge-open-closure}, and \eqref{eq:killed-bridge} proves
\eqref{eq:dyadic-killed-limit}.
\end{proof}

\begin{lemma}\label{lem:heat-domain-Minkowski}
Let $D_0,D_1\subset\R^N$ be convex open sets and $0<\theta<1$.  Put
\[
 D_\theta=(1-\theta)D_0+\theta D_1,
\quad
 x_\theta=(1-\theta)x_0+\theta x_1,
\quad
 y_\theta=(1-\theta)y_0+\theta y_1.
\]
Then for every $t>0$,
\begin{equation}\label{eq:heat-domain-Minkowski}
 p_{D_\theta}(t,x_\theta,y_\theta)
 \ge
 p_{D_0}(t,x_0,y_0)^{1-\theta}
 p_{D_1}(t,x_1,y_1)^\theta.
\end{equation}
\end{lemma}

\begin{proof}
For the finite-step kernels \eqref{eq:discrete-kernel}, define on
$\R^{N(M-1)}$
\[
 f_i(Z):=\mathbf 1_{D_i^{M-1}}(Z)
 \prod_{j=0}^{M-1}g_h(z_{j+1}-z_j),
 \qquad i=0,1,\theta,
\]
where the endpoints in the product are $(x_i,y_i)$.  If
$Z_i=(z_{i,1},\ldots,z_{i,M-1})\in D_i^{M-1}$ and
$Z_\theta=(1-\theta)Z_0+\theta Z_1$, then
$Z_\theta\in D_\theta^{M-1}$.  The Gaussian is log-concave, hence, segment by
segment,
\[
 f_\theta(Z_\theta)
 \ge f_0(Z_0)^{1-\theta}f_1(Z_1)^\theta.
\]
If one of $Z_0,Z_1$ lies outside the corresponding product domain, the
right-hand side is zero and the same inequality remains true.  Thus the
Pr\'ekopa--Leindler hypothesis holds on all of $\R^{N(M-1)}$, and it gives
\begin{equation}\label{eq:finite-step-PL}
 K_{D_\theta,M}(t;x_\theta,y_\theta)
 \ge
 K_{D_0,M}(t;x_0,y_0)^{1-\theta}
 K_{D_1,M}(t;x_1,y_1)^\theta.
\end{equation}
Choose $M=2^k$ and let $k\to\infty$.  Applying \cref{lem:dyadic-killed-kernel} to $D_0,D_1,D_\theta$ in
\eqref{eq:finite-step-PL} yields \eqref{eq:heat-domain-Minkowski}.  This finite-dimensional argument is based on the Pr\'ekopa--Leindler
inequality, as in the heat-flow results of Brascamp--Lieb \cite{BrascampLieb}.
\end{proof}

We next use the parabolic scaling of Dirichlet heat kernels to formulate a
convexity inequality adapted to the spectral parameter.

The Dirichlet heat kernel obeys
\begin{equation}\label{eq:heat-scaling}
 p_\Omega(r^2,x,y)
 =r^{-N}p_{\Omega/r}\!\left(1,\frac xr,\frac yr\right),
 \qquad r>0.
\end{equation}

\begin{proposition}\label{prop:Hbeta}
For every $\beta>0$, define
\begin{equation}\label{eq:Hbeta}
 \mathcal H_\beta(x,y,r)
 :=r^{1-\beta}p_\Omega(r^2,x,y)^{-\beta/N},
 \qquad (x,y,r)\in\Omega\times\Omega\times(0,\infty).
\end{equation}
Then $\mathcal H_\beta$ is jointly convex on
$\Omega\times\Omega\times(0,\infty)$.
\end{proposition}

\begin{proof}
Let
\[
 r_\theta=(1-\theta)r_0+\theta r_1,
\qquad
 \alpha_0=\frac{(1-\theta)r_0}{r_\theta},
\qquad
 \alpha_1=\frac{\theta r_1}{r_\theta}.
\]
Then $\alpha_0+\alpha_1=1$.  With $D_i=\Omega/r_i$, convexity of $\Omega$
gives the exact identity
\[
 \alpha_0D_0+\alpha_1D_1=\Omega/r_\theta=:D_\theta,
\]
and similarly
\[
 \frac{x_\theta}{r_\theta}
 =\alpha_0\frac{x_0}{r_0}+\alpha_1\frac{x_1}{r_1},
\qquad
 \frac{y_\theta}{r_\theta}
 =\alpha_0\frac{y_0}{r_0}+\alpha_1\frac{y_1}{r_1}.
\]
Set
\[
 q_i=p_{D_i}\!\left(1,\frac{x_i}{r_i},\frac{y_i}{r_i}\right).
\]
By \cref{lem:heat-domain-Minkowski},
$q_\theta\ge q_0^{\alpha_0}q_1^{\alpha_1}$.  Using
\eqref{eq:heat-scaling},
\[
 \mathcal H_\beta(x_i,y_i,r_i)=r_iq_i^{-\beta/N}.
\]
Hence
\[
 q_\theta^{-\beta/N}
 \le(q_0^{-\beta/N})^{\alpha_0}(q_1^{-\beta/N})^{\alpha_1}
 \le\alpha_0q_0^{-\beta/N}+\alpha_1q_1^{-\beta/N},
\]
where the last inequality follows from the weighted arithmetic--geometric
mean inequality.  Multiplication by $r_\theta$ yields
\[
 \mathcal H_\beta(x_\theta,y_\theta,r_\theta)
 \le(1-\theta)\mathcal H_\beta(x_0,y_0,r_0)
   +\theta\mathcal H_\beta(x_1,y_1,r_1).
\]
\end{proof}

\subsection{Power means and Green-kernel convexity}
\label{sec:power-means-BBL}

We fix the notation for generalized power means and recall the form of the
Borell--Brascamp--Lieb inequality used both here and in
\cref{sec:real-order-doubling}.  Let $0<\theta<1$ and $a,b\ge0$.  For
$q\in\R$ define
\begin{equation}\label{eq:power-mean-definition}
 M_q(a,b;\theta):=
 \begin{cases}
 \bigl((1-\theta)a^q+\theta b^q\bigr)^{1/q},
   &q>0,\quad a,b\ge0,\\[1mm]
 a^{1-\theta}b^\theta,
   &q=0,\quad a,b\ge0,\\[1mm]
 \bigl((1-\theta)a^q+\theta b^q\bigr)^{1/q},
   &q<0,\quad a,b>0.
 \end{cases}
\end{equation}
For $q<0$ we adopt the continuous extension
\begin{equation}\label{eq:negative-mean-zero-convention}
 M_q(a,b;\theta)=0
 \qquad\text{if }ab=0.
\end{equation}
For every fixed $q$, the map $(a,b)\mapsto M_q(a,b;\theta)$ is
nondecreasing in each variable and positively homogeneous:
\begin{equation}\label{eq:power-mean-homogeneous}
 M_q(ca,cb;\theta)=cM_q(a,b;\theta),
 \qquad c\ge0.
\end{equation}

\begin{lemma}
\label{lem:BBL-general}
Let $d\ge1$, let $q>-1/d$, and let
$f_0,f_1,f_\theta:\R^d\to[0,\infty)$ be measurable integrable functions,
with $0<\int_{\R^d}f_i<\infty$ for $i=0,1$.  Assume that
\begin{equation}\label{eq:BBL-pointwise-general}
 f_\theta((1-\theta)x_0+\theta x_1)
 \ge M_q(f_0(x_0),f_1(x_1);\theta)
 \qquad\text{for all }x_0,x_1\in\R^d.
\end{equation}
Then
\begin{equation}\label{eq:BBL-integral-general}
 \int_{\R^d}f_\theta
 \ge
 M_{q/(1+dq)}\!\left(
   \int_{\R^d}f_0,
   \int_{\R^d}f_1;\theta
 \right).
\end{equation}
\end{lemma}

\begin{proof}
This is the Borell--Brascamp--Lieb inequality in the range $q>-1/d$;
see \cite{BrascampLieb}.  The convention
\eqref{eq:negative-mean-zero-convention} is the standard one when $q<0$.
We state the result explicitly to fix the exponent and the zero-value
convention used below.
\end{proof}

We also need the following elementary properties of the negative-order
power means.

\begin{lemma}
\label{lem:negative-mean-truncation}
Let $q<0$, $L>0$, and $T_L(a):=\min\{a,L\}$ for $a\ge0$.  Then
\begin{equation}\label{eq:truncation-power-mean}
 M_q(T_L(a),T_L(b);\theta)
 \le T_L(M_q(a,b;\theta))
 \qquad\text{for all }a,b\ge0.
\end{equation}
Moreover,
\begin{equation}\label{eq:power-mean-square}
 M_q(a,b;\theta)^2=M_{q/2}(a^2,b^2;\theta).
\end{equation}
Consequently, if $c\ge M_q(a,b;\theta)$, then
\begin{equation}\label{eq:truncated-mean-consequence}
 T_L(c)^2
 \ge M_{q/2}(T_L(a)^2,T_L(b)^2;\theta).
\end{equation}
\end{lemma}

\begin{proof}
If $ab=0$, then the left-hand side of
\eqref{eq:truncation-power-mean} is zero by
\eqref{eq:negative-mean-zero-convention}, so the assertion is immediate.
Assume therefore that $a,b>0$.  Since $T_L(a)\le a$ and $T_L(b)\le b$ while
$q<0$,
\[
 T_L(a)^q\ge a^q,
 \qquad
 T_L(b)^q\ge b^q.
\]
Because $1/q<0$, raising the corresponding weighted-sum inequality to the
power $1/q$ reverses the inequality and gives
\[
 M_q(T_L(a),T_L(b);\theta)\le M_q(a,b;\theta).
\]
On the other hand, $T_L(a),T_L(b)\le L$, hence
$T_L(a)^q,T_L(b)^q\ge L^q$, and therefore
\[
 M_q(T_L(a),T_L(b);\theta)\le L.
\]
Combining the last two inequalities proves
\eqref{eq:truncation-power-mean}.  Identity
\eqref{eq:power-mean-square} follows directly from the definition:
\[
 M_{q/2}(a^2,b^2;\theta)
 =\bigl((1-\theta)a^q+\theta b^q\bigr)^{2/q}
 =M_q(a,b;\theta)^2,
\]
with cases involving zero interpreted according to
\eqref{eq:negative-mean-zero-convention}.  Finally, if
$c\ge M_q(a,b;\theta)$, monotonicity of $T_L$, followed by
\eqref{eq:truncation-power-mean} and
\eqref{eq:power-mean-square}, yields
\[
 T_L(c)
 \ge T_L(M_q(a,b;\theta))
 \ge M_q(T_L(a),T_L(b);\theta),
\]
and squaring proves \eqref{eq:truncated-mean-consequence}.
\end{proof}

We now integrate the power-mean inequality to establish the joint
convexity of the Green-kernel transform.

Fix $s>0$ and assume $N>2s$.  Put
\[
 D_s=N-2s+1,
\qquad
 F_s(x,y,r)=r^{2s-1}p_\Omega(r^2,x,y).
\]

\begin{lemma}\label{lem:Fs-power}
Let $D_s=N-2s+1$ and
$F_s(x,y,r)=r^{2s-1}p_\Omega(r^2,x,y)$ as above.  Then
$F_s^{-1/D_s}$ is jointly convex on
$\Omega\times\Omega\times(0,\infty)$.  Equivalently, $F_s$ is
$q_s$-concave there, with
\[
 q_s=-\frac1{N-2s+1}\in(-1,0).
\]
\end{lemma}

\begin{proof}
Choose
\[
 \beta_s=\frac{N}{N-2s+1}=\frac{N}{D_s}.
\]
Then
\[
 1-\beta_s=\frac{1-2s}{D_s},
\qquad
 \frac{\beta_s}{N}=\frac1{D_s},
\]
and therefore
\[
 \mathcal H_{\beta_s}(x,y,r)
 =r^{(1-2s)/D_s}p_\Omega(r^2,x,y)^{-1/D_s}
 =F_s(x,y,r)^{-1/D_s}.
\]
Apply \cref{prop:Hbeta}.
\end{proof}

\begin{proof}[Proof of \cref{thm:spectral-main}]
We first consider points for which both endpoint pairs and their
convex combination lie off the diagonal.  Thus let
$(x_i,y_i)\in\Omega\times\Omega$, $x_i\ne y_i$ for $i=0,1$, and assume also
$x_\theta\ne y_\theta$.  In this situation
\cref{lem:Ks-basic} guarantees that the three integrals defining the inverse spectral kernels below are finite.

Apply \cref{lem:Fs-power} to the one-dimensional variable $r$ and extend
$r\mapsto F_s(x,y,r)$ by zero to the whole real line.  We verify explicitly
that the pointwise hypothesis of \cref{lem:BBL-general} is preserved by this
extension.  If both input radii $r_0,r_1$ are positive, then
$r_\theta=(1-\theta)r_0+\theta r_1>0$ and
\cref{lem:Fs-power} gives
\[
 F_s(x_\theta,y_\theta,r_\theta)
 \ge M_{q_s}(F_s(x_0,y_0,r_0),F_s(x_1,y_1,r_1);\theta).
\]
If at least one of $r_0,r_1$ is nonpositive, then the corresponding extended
endpoint value is zero; since $q_s<0$, the right-hand side is zero by
\eqref{eq:negative-mean-zero-convention}, whereas the left-hand side is
nonnegative.  Thus the BBL pointwise hypothesis holds on all of $\R$.

Because $N>2s$, one has $q_s=-1/(N-2s+1)>-1$.  Moreover, for the off-diagonal
points under consideration, \cref{lem:Ks-basic} implies that the extended
functions have finite positive integrals.  Applying
\cref{lem:BBL-general} with $d=1$ therefore yields
\[
 \int_0^\infty F_s(x_\theta,y_\theta,r)\,\dd r
 \ge
 M_{q_s/(1+q_s)}\!\left(
 \int_0^\infty F_s(x_0,y_0,r)\,\dd r,
 \int_0^\infty F_s(x_1,y_1,r)\,\dd r;\theta\right).
\]
Now
\[
 \frac{q_s}{1+q_s}=-\frac1{N-2s},
\]
and the substitution $t=r^2$, together with the positive homogeneity
\eqref{eq:power-mean-homogeneous}, gives
\[
 K_{s,\Omega}(x,y)
 =\frac2{\Gamma(s)}\int_0^\infty
 r^{2s-1}p_\Omega(r^2,x,y)\,\dd r.
\]
Hence, whenever the endpoints and interpolated point are off the diagonal,
\begin{equation}\label{eq:spectral-power-mean}
 K_{s,\Omega}(x_\theta,y_\theta)
 \ge
 M_{-1/(N-2s)}\!\left(
 K_{s,\Omega}(x_0,y_0),K_{s,\Omega}(x_1,y_1);\theta\right).
\end{equation}
Since the exponent is negative, \eqref{eq:spectral-power-mean} is equivalent
to
\begin{equation}\label{eq:spectral-convex-offdiag}
 \Psi_{s,\Omega}(x_\theta,y_\theta)
 \le (1-\theta)\Psi_{s,\Omega}(x_0,y_0)
    +\theta\Psi_{s,\Omega}(x_1,y_1).
\end{equation}

It remains to remove the off-diagonal restrictions without applying the
Borell--Brascamp--Lieb inequality to the divergent diagonal integral.  By \cref{lem:Ks-basic}, $\Psi_{s,\Omega}$ defined in
\eqref{eq:spectral-convex-transform} is continuous on $\Omega\times\Omega$.  If the
interpolated point satisfies $x_\theta=y_\theta$, then
\eqref{eq:spectral-convex-offdiag} with the left-hand side interpreted through
\eqref{eq:spectral-convex-transform} is immediate, since that left-hand side
is $0$ and the right-hand side is nonnegative.  If one of the endpoints lies
on the diagonal, approximate that endpoint by off-diagonal points in
$\Omega\times\Omega$.  For each approximating pair, either the corresponding
interpolated point is on the diagonal, in which case the convexity inequality
is immediate as noted above, or it is off the diagonal, in which case
\eqref{eq:spectral-convex-offdiag} applies.  Passing to the limit by continuity
of $\Psi_{s,\Omega}$ gives the required inequality for the original
endpoints.  Thus $\Psi_{s,\Omega}$ is convex on all of
$\Omega\times\Omega$.
\end{proof}

\begin{corollary}
\label{cor:laplace-green-power-concavity}
Assume $N\ge3$.  Let $(x_i,y_i)\in\Omega\times\Omega$ with
$x_i\ne y_i$ for $i=0,1$, let $0<\theta<1$, and set
\[
 x_\theta=(1-\theta)x_0+\theta x_1,
 \qquad
 y_\theta=(1-\theta)y_0+\theta y_1.
\]
If $x_\theta\ne y_\theta$, then
\[
 G_{1,\Omega}(x_\theta,y_\theta)
 \ge
 M_{-1/(N-2)}\!\left(
 G_{1,\Omega}(x_0,y_0),G_{1,\Omega}(x_1,y_1);\theta\right).
\]
Equivalently, the inverse-power transform in \cref{thm:spectral-main},
extended by $0$ on the diagonal, is convex on $\Omega\times\Omega$.
\end{corollary}

\begin{corollary}
\label{cor:polyharmonic-green-power-concavity}
Let $\ell\in\mathbb N$ and $N>2\ell$.  Let
$(x_i,y_i)\in\Omega\times\Omega$ with $x_i\ne y_i$ for $i=0,1$,
let $0<\theta<1$, and set
\[
 x_\theta=(1-\theta)x_0+\theta x_1,
 \qquad
 y_\theta=(1-\theta)y_0+\theta y_1.
\]
If $x_\theta\ne y_\theta$, then
\[
 G_{\ell,\Omega}(x_\theta,y_\theta)
 \ge
 M_{-1/(N-2\ell)}\!\left(
 G_{\ell,\Omega}(x_0,y_0),G_{\ell,\Omega}(x_1,y_1);\theta\right).
\]
Equivalently, the transform $G_{\ell,\Omega}^{-1/(N-2\ell)}$, extended by
$0$ on the diagonal, is jointly convex on $\Omega\times\Omega$.
\end{corollary}

\begin{remark}

For $s=1$, \cref{thm:spectral-main} is Borell's classical joint convexity
theorem for the Dirichlet Green function \cite{Borell1984}.  Borell's
subsequent work \cite{Borell1985} developed a broader Brunn--Minkowski
theory for Greenian potentials.  Our argument uses the Dirichlet heat
kernel to treat the two-point Green kernel directly, including at
noninteger spectral orders.  The heat-kernel log-concavity result
used here goes back to
Brascamp--Lieb \cite{BrascampLieb}, while the modern characterization
of $F$-concavities preserved by the Dirichlet heat flow is developed in
\cite{IshigeSalaniTakatsu2024}.  For $0<s<1$, the kernel $K_{s,\Omega}$ is
the Green kernel associated with the spectral fractional power
$A^s=(-\Delta_D)^s$ of the Dirichlet Laplacian; the potential theory of this
operator via subordinate killed Brownian motion was studied, among others,
by Song--Vondra\v{c}ek \cite{SongVondracek2003}.
\cref{thm:spectral-main} extends the exact joint inverse-power convexity from
the Green kernel associated with $A$ to the Green kernels associated with all
real powers $A^s$ by combining Minkowski-domain
interpolation, the scale-covariant family \eqref{eq:Hbeta}, and
the one-dimensional Borell--Brascamp--Lieb inequality.  In particular,
the inverse-power convexity holds with exponent $-1/(N-2s)$ for every
$s>0$ with $N>2s$.
\end{remark}

\section{Robin functions for arbitrary inverse spectral powers}
\label{sec:robin-transfer}

In this section we prove \cref{thm:robin-real-main}.  We derive
convexity of the Robin function from the joint convexity of the Green
kernel and establish regularity across the diagonal by estimates for the
difference between the free and Dirichlet heat kernels.  The resulting
theorem applies to both spectral fractional and Navier polyharmonic Robin
functions, giving \cref{cor:fractional-robin-center,cor:navier-robin-center}.

Let $s>0$, $N>2s$, and set
\begin{equation*}
 m:=N-2s>0,
 \qquad
 \Phi_s(z):=a_{N,s}|z|^{-m},
 \qquad
 a_{N,s}:=\frac{\Gamma(\frac N2-s)}{4^s\pi^{N/2}\Gamma(s)}.
\end{equation*}
For $x\ne y$ write
\begin{equation}\label{eq:general-regularpart-s}
 K_{s,\Omega}(x,y)=\Phi_s(x-y)-H_{s,\Omega}(x,y).
\end{equation}

\begin{proposition}
\label{prop:robin-transfer}
Let $\Omega\subset\R^N$ be a bounded open convex set, let $s>0$, and assume
$N>2s$.  Suppose that the off-diagonal regular part
$H_{s,\Omega}$ in \eqref{eq:general-regularpart-s} extends continuously to
$\Omega\times\Omega$, and define
\[
 R_{s,\Omega}(y):=H_{s,\Omega}(y,y).
\]
Then $R_{s,\Omega}$ is convex on $\Omega$ and satisfies
\begin{equation}\label{eq:robin-transfer-bound}
 R_{s,\Omega}(y)
 \ge a_{N,s}[2\,\dist(y,\partial\Omega)]^{2s-N},
 \qquad y\in\Omega.
\end{equation}
In particular,
\[
 R_{s,\Omega}(y)\longrightarrow+\infty
 \qquad\text{as }y\to\partial\Omega.
\]
If, in addition, $R_{s,\Omega}$ is real analytic in $\Omega$, then
$R_{s,\Omega}$ is strictly convex.  Consequently, it has exactly one
critical point, which is its unique global minimizer.
\end{proposition}

\begin{proof}
Fix $0\ne h\in\R^N$, put $r:=|h|$, and set
\[
 \Omega_h:=\{y\in\Omega:y+h\in\Omega\}.
\]
The set $\Omega_h$ is convex.  By \cref{thm:spectral-main},
\[
 F_s(x,y):=K_{s,\Omega}(x,y)^{-1/m}
\]
is jointly convex, hence
\[
 f_h(y):=F_s(y+h,y),\qquad y\in\Omega_h,
\]
is convex.  From \eqref{eq:general-regularpart-s},
\begin{align*}
 f_h(y)
 &=\bigl(a_{N,s}r^{-m}-H_{s,\Omega}(y+h,y)\bigr)^{-1/m}\\
 &=a_{N,s}^{-1/m}r
 \left(1-\frac{r^m}{a_{N,s}}H_{s,\Omega}(y+h,y)\right)^{-1/m}.
\end{align*}
Let $K\Subset\Omega$.  For all sufficiently small $h$, $K\subset\Omega_h$.
The continuous extension of $H_{s,\Omega}$ is uniformly bounded on a fixed
compact neighborhood of $\{(y,y):y\in K\}$, so Taylor's theorem gives,
uniformly for $y\in K$,
\begin{equation*}
 f_h(y)
 =a_{N,s}^{-1/m}r
 +\frac1m a_{N,s}^{-1-1/m}r^{m+1}H_{s,\Omega}(y+h,y)
 +O_K(r^{2m+1}).
\end{equation*}
Define
\[
 Q_h(y):=
 \frac{m a_{N,s}^{1+1/m}}{r^{m+1}}
 \left(f_h(y)-a_{N,s}^{-1/m}r\right).
\]
Then $Q_h$ is convex on $\Omega_h$ and, locally uniformly in $\Omega$,
\begin{equation*}
 Q_h(y)=H_{s,\Omega}(y+h,y)+O_K(r^m)
 \longrightarrow R_{s,\Omega}(y).
\end{equation*}
Applying the convexity inequality for $Q_h$ on any fixed segment
$[y_0,y_1]\Subset\Omega$ and passing to the limit shows that
$R_{s,\Omega}$ is convex.

For the boundary estimate, fix $y\in\Omega$ and write
$d:=\dist(y,\partial\Omega)$.  Choose
$\zeta\in\partial\Omega$ with $|y-\zeta|=d$, and let $P$ be a supporting
hyperplane to the compact convex set $\overline\Omega$ at $\zeta$.  Denote by
$\mathcal H$ the open half-space bounded by $P$ which contains $\Omega$, and
let $y^*$ be the reflection of $y$ across $P$.  Domain monotonicity of Dirichlet heat kernels gives
\[
 p_\Omega(t,x,y)\le p_{\mathcal H}(t,x,y).
\]
After integration against $t^{s-1}/\Gamma(s)$,
\[
 K_{s,\Omega}(x,y)\le K_{s,\mathcal H}(x,y).
\]
The half-space reflection formula
\[
 p_{\mathcal H}(t,x,y)=g_t(x-y)-g_t(x-y^*)
\]
and the whole-space Mellin integral yield
\[
 K_{s,\mathcal H}(x,y)=\Phi_s(x-y)-\Phi_s(x-y^*).
\]
Hence
\[
 H_{s,\Omega}(x,y)\ge \Phi_s(x-y^*).
\]
Letting $x\to y$ and using continuity of the extension gives
\[
 R_{s,\Omega}(y)\ge a_{N,s}|y-y^*|^{2s-N}.
\]
If $\delta:=\dist(y,P)$, then $|y-y^*|=2\delta$ and
$\delta\le|y-\zeta|=d$.  Since $2s-N<0$,
\[
 R_{s,\Omega}(y)
 \ge a_{N,s}(2\delta)^{2s-N}
 \ge a_{N,s}(2d)^{2s-N},
\]
which is \eqref{eq:robin-transfer-bound}.  In particular
$R_{s,\Omega}(y)\to+\infty$ as $y\to\partial\Omega$.

Assume now, in addition, that $R_{s,\Omega}$ is real analytic in $\Omega$.
Suppose that $R_{s,\Omega}$ is not strictly convex.
Then there exist $y_0\ne y_1$ and $t_0\in(0,1)$ for which equality holds in
the convexity inequality.  Setting $\eta:=y_1-y_0$ and
\[
 g(t):=R_{s,\Omega}(y_0+t\eta),
\]
one-dimensional convexity implies that $g$ is affine on $[0,1]$.
Let
\[
 I:=\{t\in\R:y_0+t\eta\in\Omega\}=(a,b).
\]
Because $\Omega$ is open, bounded and convex, $(a,b)$ is a bounded open
interval containing $[0,1]$.  The function $g$ is real analytic on $(a,b)$,
so the identity theorem for real-analytic functions implies that $g$ is
affine throughout $(a,b)$.  On the other hand, as $t\downarrow a$, the point $y_0+t\eta$ approaches
$\partial\Omega$, and the boundary estimate obtained above gives
$g(t)\to+\infty$, contradicting the finite endpoint limit of an affine
function.  Thus $R_{s,\Omega}$ is strictly convex.

Finally, fix $y_0\in\Omega$ and put $c_0:=R_{s,\Omega}(y_0)$.  The boundary
divergence established above implies that the sublevel set
\[
 \{y\in\Omega:R_{s,\Omega}(y)\le c_0\}
\]
has compact closure contained in $\Omega$.  Hence $R_{s,\Omega}$ attains a
global minimum at some interior point $y_{s,\Omega}$ and
$\nabla R_{s,\Omega}(y_{s,\Omega})=0$.  Conversely, every critical point of a
differentiable convex function is a global minimizer.  Strict convexity
therefore implies that $y_{s,\Omega}$ is the unique critical point and the
unique global minimizer.
\end{proof}

\subsection{Analyticity of the regular part}
\label{sec:real-robin-regularity}

Throughout this subsection $s>0$, $N>2s$, and
$\Omega\subset\R^N$ is a bounded open convex set.  We use the Robin
function of \cref{def:spectral-robin-real} and put
\[
 q_\Omega(t,x,y):=g_t(x-y)-p_\Omega(t,x,y).
\]
We establish the following regularity property:
the regular part extends real analytically to $\Omega\times\Omega$ for
every real $s>0$ in the range $N>2s$.  In particular, it extends continuously
across the diagonal, and its restriction to the diagonal is real analytic,
as required by \cref{prop:robin-transfer}.

We begin with derivative estimates for the free Gaussian kernel.

\begin{lemma}
\label{lem:gaussian-derivative-bounds}
Let $\gamma$ be a multiindex and $m:=|\gamma|$.  There is a dimensional
constant $C_0>1$ such that, for every $t>0$ and $z\in\R^N$,
\begin{equation}\label{eq:gaussian-derivative}
 |D^\gamma g_t(z)|
 \le C_0^{m+1}m^{m/2}t^{-(N+m)/2}
      \exp\!\left(-\frac{|z|^2}{8t}\right),
\end{equation}
where $0^0:=1$.  Consequently
\begin{align}
 \|D^\gamma g_t\|_{L^1(\R^N)}
 &\le C_0^{m+1}m^{m/2}t^{-m/2},
 \label{eq:gaussian-L1}\\
 \|D^\gamma g_t\|_{L^2(\R^N)}
 &\le C_0^{m+1}m^{m/2}t^{-m/2-N/4}.
 \label{eq:gaussian-L2}
\end{align}
Moreover, for every $d>0$ there are constants $A_d>1$ and $c_d>0$ such
that, whenever $0<u\le r\le1$ and $|z|\ge d$,
\begin{equation}\label{eq:gaussian-offdiag-factorial}
 |D^\gamma g_u(z)|
 \le A_d^{m+1}m!\,e^{-c_d/r}.
\end{equation}
\end{lemma}

\begin{proof}
Write
\[
 g_t(z)=(4\pi t)^{-N/2}\prod_{j=1}^N e^{-u_j^2},
 \qquad u_j:=\frac{z_j}{2\sqrt t}.
\]
Let $H_n$ denote the physicists' Hermite polynomial, so that
\[
 \frac{\dd^n}{\dd u^n}e^{-u^2}=(-1)^nH_n(u)e^{-u^2}.
\]
The required factorial derivative bound follows from the Hermite representation.
The generating function
\[
 e^{2uw-w^2}=\sum_{n=0}^\infty H_n(u)\frac{w^n}{n!}
\]
and Cauchy's estimate on $|w|=\rho$ give
\[
 |H_n(u)|\le n!\rho^{-n}e^{2|u|\rho+\rho^2}.
\]
Since $2|u|\rho-u^2/2\le2\rho^2$, choosing
$\rho=\sqrt{n/6}$ for $n\ge1$ and absorbing the case $n=0$ yields
\begin{equation}\label{eq:hermite-elementary}
 e^{-u^2/2}|H_n(u)|
 \le C^{n+1}n^{n/2}.
\end{equation}
Applying \eqref{eq:hermite-elementary} coordinatewise and using
$\prod_j\gamma_j^{\gamma_j/2}\le m^{m/2}$ gives
\eqref{eq:gaussian-derivative}.  Integrating the remaining Gaussian proves
\eqref{eq:gaussian-L1} and \eqref{eq:gaussian-L2}.

For the last assertion, assume $|z|\ge d$ and put $M:=(N+m)/2$.  Splitting
the exponential in \eqref{eq:gaussian-derivative},
\[
 u^{-M}e^{-d^2/(8u)}
 \le e^{-d^2/(16r)}u^{-M}e^{-d^2/(16u)}.
\]
The last factor satisfies
\[
 \sup_{u>0}u^{-M}e^{-d^2/(16u)}
 =\left(\frac{16M}{ed^2}\right)^M.
\]
Hence
\[
 |D^\gamma g_u(z)|
 \le C_d^{m+1}m^{m/2}M^Me^{-M}e^{-d^2/(16r)}.
\]
Stirling's inequalities, with the finitely many small values of $m$
absorbed into the constant, imply
\[
 m^{m/2}M^Me^{-M}\le C_N^{m+1}m!.
\]
This proves \eqref{eq:gaussian-offdiag-factorial} with
$c_d=d^2/16$ after enlarging $A_d$.
\end{proof}

The Dynkin--Hunt representation gives the following estimates for
derivatives at one endpoint.

\begin{lemma}
\label{lem:fractional-one-endpoint}
For every $K\Subset\Omega$ there exist $A_K>1$ and $c_K>0$ such that for
all multiindices $\alpha,\beta$ and $0<t\le1$,
\begin{align}
 \sup_{\substack{x\in\Omega\\y\in K}}
 |D_y^\beta q_\Omega(t,x,y)|
 &\le A_K^{|\beta|+1}|\beta|!e^{-c_K/t},
 \label{eq:q-one-y}\\
 \sup_{\substack{x\in K\\y\in\Omega}}
 |D_x^\alpha q_\Omega(t,x,y)|
 &\le A_K^{|\alpha|+1}|\alpha|!e^{-c_K/t}.
 \label{eq:q-one-x}
\end{align}
\end{lemma}

\begin{proof}
Let $d:=\dist(K,\partial\Omega)>0$ and choose the closed interior neighborhood
\[
 K_*:=\{z\in\R^N:\dist(z,K)\le d/4\}\Subset\Omega.
\]
Then $\dist(K_*,\partial\Omega)\ge3d/4$.  The Dynkin--Hunt formula \eqref{eq:not-feeling-hunt} gives
\begin{equation}\label{eq:q-hunt-fractional}
 q_\Omega(t,x,y)
 =\mathbb E_x\!\left[
   \mathbf1_{\{\tau_\Omega<t\}}
   g_{t-\tau_\Omega}(B_{\tau_\Omega}-y)
 \right]
\end{equation}
for $x,y\in\Omega$.  On $\{\tau_\Omega<t\}$, path continuity gives
$B_{\tau_\Omega}\in\partial\Omega$.  Therefore, for every $y'\in K_*$,
\[
 |B_{\tau_\Omega}-y'|\ge3d/4.
\]
Applying \eqref{eq:gaussian-offdiag-factorial} with
$u=t-\tau_\Omega$ and $r=t$ gives, for $b:=|\beta|$,
\[
 \sup_{y'\in K_*}
 |D_{y'}^\beta
 g_{t-\tau_\Omega}(B_{\tau_\Omega}-y')|
 \le A_K^{b+1}b!e^{-c_K/t}.
\]
The right-hand side is independent of the Brownian path and is therefore
integrable.  Differentiation with respect to $y$ is justified by induction on
the order.  The case of order zero is \eqref{eq:q-hunt-fractional}.  Suppose
that for a multiindex $\gamma$ we already have
\[
 D_y^\gamma q_\Omega(t,x,y)
 =\mathbb E_x\!\left[
   \mathbf1_{\{\tau_\Omega<t\}}
   D_y^\gamma g_{t-\tau_\Omega}(B_{\tau_\Omega}-y)
 \right].
\]
For a coordinate vector $e_j$, the difference quotient of the integrand in
the $e_j$ direction is, by the one-dimensional mean-value theorem, an
$(|\gamma|+1)$st derivative of the same Gaussian evaluated at a point
$y+\theta h e_j\in K_*$, provided $|h|$ is small.  The preceding off-diagonal
factorial estimate with $\beta=\gamma+e_j$ therefore supplies a deterministic
integrable majorant, independent of $h$.  Dominated convergence gives the
formula with $\gamma+e_j$ in place of $\gamma$.  Induction proves the
pointwise derivative formula for every $\beta$, together with
\eqref{eq:q-one-y}.  Finally,
$p_\Omega(t,x,y)=p_\Omega(t,y,x)$ and $g_t(x-y)=g_t(y-x)$, so
$q_\Omega(t,x,y)=q_\Omega(t,y,x)$; interchanging the endpoints gives
\eqref{eq:q-one-x}.  In particular, no differentiation of the Brownian law
with respect to its starting point is required.
\end{proof}

We next establish the corresponding bounds for derivatives in both
spatial variables.

\begin{lemma}
\label{lem:fractional-mixed-q}
For every $K\Subset\Omega$ there are constants $C_K>1$ and $c_K>0$ such
that, for $0<t\le1$, $x,y\in K$, and multiindices $\alpha,\beta$ with
$a:=|\alpha|$, $b:=|\beta|$,
\begin{align}
 |D_x^\alpha D_y^\beta q_\Omega(t,x,y)|
 \le C_K^{a+b+1}e^{-c_K/t}
 \Big(&a!\,b^{b/2}t^{-b/2}
      +b!\,a^{a/2}t^{-a/2}
      +a!b!\Big).
 \label{eq:q-mixed-small}
\end{align}
\end{lemma}

\begin{proof}
Put $r:=t/2$.  The free and killed semigroup identities give
\begin{align}
 q_\Omega(t,x,y)
 &=\int_{\Omega^c}g_r(x-z)g_r(z-y)\,\dd z \notag\\
 &\quad+\int_\Omega\Big[
 q_\Omega(r,x,z)g_r(z-y)
 +g_r(x-z)q_\Omega(r,z,y)
 -q_\Omega(r,x,z)q_\Omega(r,z,y)
 \Big]\,\dd z.
 \label{eq:q-semigroup-splitting}
\end{align}
We first justify all differentiations in
\eqref{eq:q-semigroup-splitting} at a fixed time $t\in(0,1]$, so that
$r=t/2\in(0,1/2]$ is fixed during this step.  Choose a compact interior
neighborhood $K_1$ with
\[
 K\subset\operatorname{int}K_1,
 \qquad K_1\Subset\Omega,
 \qquad d_1:=\dist(K_1,\Omega^c)>0.
\]
All coordinate difference quotients are taken sufficiently small that the
corresponding segments of shifted $x$- and $y$-points remain in $K_1$.

We construct explicitly a common integrable envelope for translated
Gaussian derivatives.  Choose $R\ge1$ with $K_1\subset B_R(0)$.  For every multiindex
$\delta$, with $n:=|\delta|$, set
\[
 M_{\delta,r}(z):=
 \sup_{w\in K_1}|D^\delta g_r(z-w)|,
 \qquad z\in\R^N.
\]
This is a measurable function, since the supremum of the continuous
integrand over the compact parameter set $K_1$ is continuous in $z$.
By \eqref{eq:gaussian-derivative}, with
\[
 B_{\delta,r}:=C_0^{n+1}n^{n/2}r^{-(N+n)/2},
\]
and the convention $0^0=1$, we have
\[
 M_{\delta,r}(z)
 \le B_{\delta,r}
 \begin{cases}
  1,&|z|\le2R,\\
  \exp\!\left(-\dfrac{|z|^2}{32r}\right),&|z|>2R.
 \end{cases}
\]
Indeed, in the second case $|z-w|\ge|z|-R\ge|z|/2$ for every
$w\in K_1$.  The displayed bound is integrable on $\R^N$, and hence
$M_{\delta,r}\in L^1(\R^N)$.  It dominates all translations with centers
in $K_1$, including those arising from difference quotients.  Derivatives
of $g_r(w-z)$ obey the same bound, by the evenness of the Gaussian.

For the term
\[
 I_1(x,y):=\int_\Omega q_\Omega(r,x,z)g_r(z-y)\,\dd z,
\]
the one-endpoint estimate \eqref{eq:q-one-x}, used with $K_1$, gives
\[
 \sup_{\substack{x'\in K_1\\z\in\Omega}}
 |D_x^\alpha q_\Omega(r,x',z)|
 \le A_{K_1}^{a+1}a!e^{-c_{K_1}/r}.
\]
Thus, uniformly for $x',y'\in K_1$, the absolute value of the formally
differentiated integrand of orders $(\alpha,\beta)$ is bounded by
\[
 A_{K_1}^{a+1}a!e^{-c_{K_1}/r}
 M_{\beta,r}(z)\mathbf1_\Omega(z),
\]
which is an integrable function of $z$ independent of $x',y'$.
The same construction applies at every derivative order.  At each
successive coordinate differentiation, the mean-value theorem bounds the
difference quotient by an endpoint derivative of one higher order along a
segment contained in $K_1$; the preceding envelope, with that higher order,
is independent of the increment.  Induction on the total differentiation
order and dominated convergence therefore give
\[
 D_x^\alpha D_y^\beta I_1(x,y)
 =\int_\Omega D_x^\alpha q_\Omega(r,x,z)
 D_y^\beta g_r(z-y)\,\dd z.
\]
For
\[
 I_2(x,y):=\int_\Omega g_r(x-z)q_\Omega(r,z,y)\,\dd z,
\]
the corresponding integrable envelope is
\[
 A_{K_1}^{b+1}b!e^{-c_{K_1}/r}
 M_{\alpha,r}(z)\mathbf1_\Omega(z),
\]
by \eqref{eq:q-one-y}; hence the same differentiation argument applies.
For
\[
 I_3(x,y):=\int_\Omega q_\Omega(r,x,z)q_\Omega(r,z,y)\,\dd z,
\]
both endpoint derivatives are uniformly bounded by
\eqref{eq:q-one-x}--\eqref{eq:q-one-y}.  Their product is therefore
dominated, at every fixed pair of derivative orders, by a constant times
$\mathbf1_\Omega$, which is integrable because $|\Omega|<\infty$.
Finally, for
\[
 I_0(x,y):=\int_{\Omega^c}g_r(x-z)g_r(z-y)\,\dd z,
\]
we use the separation
\[
 |x'-z|\ge d_1
 \qquad(x'\in K_1,\ z\in\Omega^c).
\]
By \eqref{eq:gaussian-offdiag-factorial}, uniformly for $x',y'\in K_1$
and $z\in\Omega^c$,
\[
 |D_x^\alpha g_r(x'-z)D_y^\beta g_r(z-y')|
 \le A_{d_1}^{a+1}a!e^{-c_{d_1}/r}M_{\beta,r}(z).
\]
The right-hand side is integrable on $\Omega^c$, and the same bound at
one higher endpoint order controls each difference quotient.  Thus all
mixed derivatives of $I_0$ can also be passed under the integral.
Continuity of the resulting derivatives follows by the same envelopes
and dominated convergence, using continuity of the one-endpoint
derivatives from \cref{lem:fractional-one-endpoint}.  This justifies
successive differentiation of all four terms in
\eqref{eq:q-semigroup-splitting} to arbitrary finite order.  The envelopes
in this justification may depend on the fixed time $r$ and on the
derivative orders; the quantitative estimates uniform in $0<t\le1$ are
obtained separately below.

We estimate the four differentiated terms.  For $I_1$,
\eqref{eq:q-one-x} and \eqref{eq:gaussian-L1} yield
\[
 |D_x^\alpha D_y^\beta I_1(x,y)|
 \le C_K^{a+b+1}a!\,b^{b/2}t^{-b/2}e^{-c_K/t}.
\]
The term $I_2$ is bounded symmetrically by
\[
 C_K^{a+b+1}b!\,a^{a/2}t^{-a/2}e^{-c_K/t}.
\]
For $I_3$, boundedness of $\Omega$ and the two one-endpoint estimates give
\[
 |D_x^\alpha D_y^\beta I_3(x,y)|
 \le C_K^{a+b+1}a!b!e^{-c_K/t}.
\]
For $I_0$, \eqref{eq:gaussian-offdiag-factorial} in the $x$ factor and
\eqref{eq:gaussian-L1} in the $y$ factor give the same bound as for $I_1$.
Summing the four estimates proves \eqref{eq:q-mixed-small}.
\end{proof}

\begin{lemma}
\label{lem:finite-order-heat-defect}
Let $K\Subset\Omega$ and let $M\in\mathbb N_0$.  Then there exist constants
$C=C(K,M)>0$, $c=c(K)>0$, and $P=P(M,N)\ge0$ such that, for $0<t\le1$,
\begin{align}
 &\sup_{\substack{x\in\Omega\\y\in K}}
   |D_y^\beta q_\Omega(t,x,y)|
 +\sup_{\substack{x\in K\\y\in\Omega}}
   |D_x^\alpha q_\Omega(t,x,y)|
 \le Ct^{-P}e^{-c/t},
 \qquad |\alpha|,|\beta|\le M,
 \label{eq:finite-defect-one-endpoint}\\
 &\sup_{x,y\in K}
   |D_x^\alpha D_y^\beta q_\Omega(t,x,y)|
 \le Ct^{-P}e^{-c/t},
 \qquad |\alpha|+|\beta|\le M.
 \label{eq:finite-defect-mixed}
\end{align}
The constants are uniform for all derivatives in the indicated finite sets.
\end{lemma}

\begin{proof}
This is an immediate finite-order consequence of the stronger factorial
estimates above.  For the one-endpoint bounds, apply
\cref{lem:fractional-one-endpoint}; because only the finitely many orders
$0,\ldots,M$ occur, the factors $A_K^{m+1}m!$ are absorbed into a single
constant, and one may take $P=0$ there.  For the mixed bounds, apply
\cref{lem:fractional-mixed-q}.  Since $0<t\le1$ and
$|\alpha|+|\beta|\le M$, every factor $t^{-|\alpha|/2}$ or
$t^{-|\beta|/2}$ is bounded by $t^{-M/2}$; the finitely many factorial and
power factors are again absorbed into $C$.  Enlarging $P$ if desired gives
the common form stated above.
\end{proof}

\begin{lemma}
\label{lem:fractional-gamma-factorial}
Let $s>0$ and $c>0$.  There exists $C=C(s,c)>1$ such that for every integer
$m\ge0$,
\begin{equation}\label{eq:gamma-factorial}
 m^{m/2}\int_0^1 t^{s-1-m/2}e^{-c/t}\,\dd t
 \le C^{m+1}m!,
\end{equation}
with $0^0:=1$.
\end{lemma}

\begin{proof}
Set
\[
 J_m:=\int_0^1 t^{s-1-m/2}e^{-c/t}\,\dd t.
\]
The substitution $u=c/t$ gives
\[
 J_m=c^{s-m/2}\int_c^\infty u^{m/2-s-1}e^{-u}\,\dd u.
\]
For the finitely many integers $m\le2s+2$, the required bound follows after
enlarging $C$.  For $m>2s+2$,
\[
 J_m\le c^{s-m/2}\Gamma\!\left(\frac m2-s\right).
\]
Stirling's upper bound for the Gamma function and lower bound for $m!$ give
\[
 \frac{m^{m/2}J_m}{m!}
 \le C_{s,c}\left(C_c\right)^m m^{-s-1}
 \le C_{s,c}^{m+1}.
\]
This is \eqref{eq:gamma-factorial}.
\end{proof}

\begin{lemma}
\label{lem:factorial-analytic-criterion}
Let $U\subset\R^d$ be open and let $f\in C^\infty(U)$.  Assume that for every
compact set $K\Subset U$ there is a constant $C_K\ge1$ such that
\begin{equation}\label{eq:factorial-analytic-criterion}
 \sup_{x\in K}|D^\gamma f(x)|
 \le C_K^{|\gamma|+1}|\gamma|!
 \qquad\text{for every multiindex }\gamma.
\end{equation}
Then $f$ is real analytic in $U$.
\end{lemma}

\begin{proof}
Fix $x_0\in U$ and choose $r>0$ so that
$\overline{B_{2r}(x_0)}\Subset U$.  Let $C\ge1$ be the constant in
\eqref{eq:factorial-analytic-criterion} for this compact ball.  For a vector
$h\in\R^d$ with $|h|<r$ set
\[
 \phi_h(t):=f(x_0+th),\qquad 0\le t\le1.
\]
For every integer $n\ge0$, the multinomial formula gives
\[
 \phi_h^{(n)}(t)
 =\sum_{|\gamma|=n}\frac{n!}{\gamma!}
 h^\gamma D^\gamma f(x_0+th).
\]
Since the whole segment $\{x_0+th:0\le t\le1\}$ lies in
$\overline{B_{2r}(x_0)}$, the factorial bound implies
\begin{align*}
 |\phi_h^{(n)}(t)|
 &\le C^{n+1}n!
 \sum_{|\gamma|=n}\frac{n!}{\gamma!}|h^\gamma|\\
 &=C^{n+1}n!\,|h|_1^n,
\end{align*}
where $|h|_1:=\sum_{j=1}^d|h_j|$.  Taylor's formula in the one-dimensional
variable $t$, evaluated at $t=1$, therefore gives
\begin{align*}
 f(x_0+h)
 &=\sum_{n=0}^m\frac1{n!}(h\cdot\nabla)^n f(x_0)+\mathcal R_m(h),\\
 |\mathcal R_m(h)|
 &\le \frac1{(m+1)!}
 \sup_{0\le t\le1}|\phi_h^{(m+1)}(t)|
 \le C^{m+2}|h|_1^{m+1}.
\end{align*}
If $|h|<r$ and $C|h|_1<1$, the last quantity tends to zero as
$m\to\infty$.  Finally,
\[
 \frac1{n!}(h\cdot\nabla)^n f(x_0)
 =\sum_{|\gamma|=n}\frac{D^\gamma f(x_0)}{\gamma!}h^\gamma,
\]
so the resulting series is precisely the multivariable Taylor series of
$f$ at $x_0$.  Hence $f$ agrees with its Taylor series in a neighborhood of
$x_0$.  Since $x_0$ was arbitrary, $f\in C^\omega(U)$.
\end{proof}

These estimates yield the analytic continuation of the regular part
across the diagonal.

\begin{lemma}
\label{lem:real-order-analytic}
Let $s>0$ and assume $N>2s$.  Then the function
\[
 H_{s,\Omega}(x,y)=\Phi_s(x-y)-K_{s,\Omega}(x,y),
 \qquad x\ne y,
\]
extends uniquely to a real-analytic function on $\Omega\times\Omega$.
More precisely, for every $K\Subset\Omega$ there exists $C_{K,s}>1$ such that
for all multiindices $\alpha,\beta$,
\begin{equation}\label{eq:real-order-analytic-factorial}
 \sup_{(x,y)\in K\times K}
 |D_x^\alpha D_y^\beta H_{s,\Omega}(x,y)|
 \le C_{K,s}^{|\alpha|+|\beta|+1}
 (|\alpha|+|\beta|)!.
\end{equation}
The dependence of $C_{K,s}$ on $s$ is intentional; no uniformity of this
analytic bound with respect to the spectral parameter is asserted.
Consequently $R_{s,\Omega}(y)=H_{s,\Omega}(y,y)$ is real analytic in
$\Omega$.  No restriction $s<1$ is used in this conclusion.
\end{lemma}

\begin{proof}
We construct the extension directly from the heat-kernel defect.  Set
\begin{equation}\label{eq:real-order-H-heat}
 \widetilde H_{s,\Omega}(x,y)
 :=\frac1{\Gamma(s)}\int_0^\infty
 t^{s-1}q_\Omega(t,x,y)\,\dd t.
\end{equation}
For $x,y$ in a fixed compact subset of $\Omega$, the small-time integral
converges by \eqref{eq:not-feeling-boundary}; the large-time integral
converges by the estimates below.

Let $K\Subset\Omega$, let $a=|\alpha|$, $b=|\beta|$ and $k=a+b$.  From
\cref{lem:fractional-mixed-q,lem:fractional-gamma-factorial}, together with
$a!b!\le k!$, we obtain
\begin{equation}\label{eq:real-order-smalltime-integrated}
 \int_0^1 t^{s-1}
 \sup_{(x,y)\in K\times K}
 |D_x^\alpha D_y^\beta q_\Omega(t,x,y)|\,\dd t
 \le C_{K,s}^{k+1}k!.
\end{equation}
Indeed, the pointwise estimate \eqref{eq:q-mixed-small} is uniform on
$K\times K$.  After taking the supremum in $(x,y)$, the first term is
integrable by \cref{lem:fractional-gamma-factorial} and bounded by
$C_{K,s}^{k+1}a!b!$; the second term is treated symmetrically, and the third one
is bounded directly by $C_{K,s}^{k+1}a!b!$.  Thus the left-hand side of
\eqref{eq:real-order-smalltime-integrated} is finite and provides an explicit
$L^1(0,1)$ majorant, uniform for $(x,y)\in K\times K$.

For $t\ge1$, $D_x^\alpha D_y^\beta g_t(x-y)$ is, up to sign, a derivative
of $g_t$ of total order $k$.  Hence \eqref{eq:gaussian-derivative} gives
\[
 |D_x^\alpha D_y^\beta g_t(x-y)|
 \le C^{k+1}k^{k/2}t^{-(N+k)/2}.
\]
Since $N>2s$,
\begin{equation}\label{eq:real-order-large-g}
 \int_1^\infty t^{s-1}
 \sup_{(x,y)\in K\times K}
 |D_x^\alpha D_y^\beta g_t(x-y)|\,\dd t
 \le C_{K,s}^{k+1}k!.
\end{equation}

For the killed kernel let $A=-\Delta_D$ and let $\lambda_1>0$ be its first
Dirichlet eigenvalue.  At the fixed time $1/2$,
$p_\Omega(1/2)=g_{1/2}-q_\Omega(1/2)$; therefore
\eqref{eq:gaussian-L2} and the one-endpoint estimates imply
\begin{align}
 \sup_{x\in K}
 \|D_x^\alpha p_\Omega(1/2,x,\cdot)\|_{L^2(\Omega)}
 &\le C_{K,s}^{a+1}a!,\label{eq:p-half-L2-x}\\
 \sup_{y\in K}
 \|D_y^\beta p_\Omega(1/2,\cdot,y)\|_{L^2(\Omega)}
 &\le C_{K,s}^{b+1}b!.\label{eq:p-half-L2-y}
\end{align}
To justify differentiation of the large-time semigroup formula without
assuming mixed differentiability a priori, we start from the undifferentiated
identity, valid for $t\ge1$,
\begin{equation}\label{eq:p-large-semigroup-zero}
 p_\Omega(t,x,y)
 =\Big\langle
 p_\Omega(1/2,x,\cdot),
 e^{-(t-1)A}p_\Omega(1/2,\cdot,y)
 \Big\rangle_{L^2(\Omega)}.
\end{equation}
Choose $K\Subset K_1\Subset\Omega$.  The bounds
\eqref{eq:p-half-L2-x}--\eqref{eq:p-half-L2-y}, with $K_1$ in place of $K$,
hold for every derivative order.  For a multiindex $\gamma$, set
\[
 F_\gamma(x):=D_x^\gamma p_\Omega(1/2,x,\cdot)\in L^2(\Omega),
\]
and, for a multiindex $\delta$, set
\[
 G_\delta(y):=D_y^\delta p_\Omega(1/2,\cdot,y)\in L^2(\Omega).
\]
The maps $F_\gamma$ and $G_\delta$ are continuous as
$L^2(\Omega)$-valued maps on $K_1$.  For the free part this follows from
translation continuity of Gaussian derivatives in $L^2(\R^N)$.  For the
defect, symmetry gives
$D_x^\gamma q_\Omega(1/2,x,z)=D_y^\gamma q_\Omega(1/2,z,x)$; the pointwise
endpoint derivative formula proved in \cref{lem:fractional-one-endpoint}
is continuous in the parameter $x$ by dominated convergence on the same compact interior neighborhood, while the uniform bound in that lemma and
$|\Omega|<\infty$ yield $L^2$ continuity by another application of dominated
convergence.  The same argument applies at the $y$ endpoint.

Their $L^2$-valued differentiability follows directly from the fundamental
theorem of calculus.  Fix a coordinate vector
$e_j$, a multiindex $\gamma$, and $x\in K$.  For $|h|$ sufficiently small,
the segment $\{x+\theta h e_j:0\le\theta\le1\}$ is contained in $K_1$.
For every $z\in\Omega$, the one-dimensional fundamental theorem of
calculus gives
\[
 D_x^\gamma p_\Omega(1/2,x+h e_j,z)
 -D_x^\gamma p_\Omega(1/2,x,z)
 =\int_0^h
 D_x^{\gamma+e_j}p_\Omega(1/2,x+s e_j,z)\,\dd s.
\]
Since $s\mapsto F_{\gamma+e_j}(x+s e_j)$ is continuous as an
$L^2(\Omega)$-valued map, the right-hand side is a Bochner integral and the
same identity holds in $L^2(\Omega)$.  Hence
\[
 \frac{F_\gamma(x+h e_j)-F_\gamma(x)}{h}
 =\int_0^1F_{\gamma+e_j}(x+\theta h e_j)\,\dd\theta,
\]
and therefore
\begin{align*}
 &\left\|
 \frac{F_\gamma(x+h e_j)-F_\gamma(x)}{h}
 -F_{\gamma+e_j}(x)
 \right\|_{L^2(\Omega)}\\
 &\qquad\le
 \int_0^1
 \|F_{\gamma+e_j}(x+\theta h e_j)-F_{\gamma+e_j}(x)\|_{L^2(\Omega)}
 \,\dd\theta
 \longrightarrow0
 \qquad(h\to0).
\end{align*}
Thus $\partial_{x_j}F_\gamma=F_{\gamma+e_j}$ in $L^2(\Omega)$.  Since these
coordinate derivatives are continuous, the standard Banach-valued
$C^1$ criterion and induction on $|\gamma|$ show that
\[
 x\longmapsto p_\Omega(1/2,x,\cdot)
 \quad\text{belongs to }C^\infty_{\mathrm{loc}}(\Omega;L^2(\Omega)),
\]
with
\[
 D_x^\alpha[p_\Omega(1/2,x,\cdot)]
 =D_x^\alpha p_\Omega(1/2,x,\cdot).
\]
The same argument gives
\[
 y\longmapsto p_\Omega(1/2,\cdot,y)
 \quad\text{in }C^\infty_{\mathrm{loc}}(\Omega;L^2(\Omega)),
\]
with the analogous identity for all $y$-derivatives.

For fixed $t\ge1$, put $T_t:=e^{-(t-1)A}$.  Since $T_t$ is bounded on
$L^2(\Omega)$, the map
\[
 B_t(u,v):=\langle u,T_tv\rangle_{L^2(\Omega)}
\]
is a continuous bilinear form on $L^2(\Omega)\times L^2(\Omega)$.  Applying the chain rule for continuous bilinear maps to
\eqref{eq:p-large-semigroup-zero} then yields
\begin{align*}
 D_x^\alpha D_y^\beta p_\Omega(t,x,y)
 =\Big\langle
 D_x^\alpha p_\Omega(1/2,x,\cdot),
 e^{-(t-1)A}D_y^\beta p_\Omega(1/2,\cdot,y)
 \Big\rangle_{L^2(\Omega)}.
\end{align*}
Since
$\|e^{-(t-1)A}\|_{L^2\to L^2}=e^{-\lambda_1(t-1)}$, after enlarging $C_{K,s}$
if necessary, \eqref{eq:p-half-L2-x}--\eqref{eq:p-half-L2-y} yield
\[
 |D_x^\alpha D_y^\beta p_\Omega(t,x,y)|
 \le C_{K,s}^{k+1}a!b!e^{-\lambda_1(t-1)}
 \le C_{K,s}^{k+1}k!e^{-\lambda_1(t-1)}.
\]
Thus
\begin{equation}\label{eq:real-order-large-p}
 \int_1^\infty t^{s-1}
 \sup_{(x,y)\in K\times K}
 |D_x^\alpha D_y^\beta p_\Omega(t,x,y)|\,\dd t
 \le C_{K,s}^{k+1}k!.
\end{equation}
Combining \eqref{eq:real-order-large-g} and
\eqref{eq:real-order-large-p} gives the same uniform-integrability bound for
$q_\Omega$ on $[1,\infty)$.

Consequently, \eqref{eq:real-order-smalltime-integrated} together with the
large-time bounds provides, for every compact $K\Subset\Omega$ and every
pair of multiindices $\alpha,\beta$, an $L^1(0,\infty)$ majorant for
\[
 t^{s-1}|D_x^\alpha D_y^\beta q_\Omega(t,x,y)|
\]
that is independent of $(x,y)\in K\times K$.  Successive dominated
convergence in
\eqref{eq:real-order-H-heat} therefore yields
\[
 D_x^\alpha D_y^\beta\widetilde H_{s,\Omega}(x,y)
 =\frac1{\Gamma(s)}\int_0^\infty t^{s-1}
 D_x^\alpha D_y^\beta q_\Omega(t,x,y)\,\dd t,
\]
and the preceding estimates prove
\eqref{eq:real-order-analytic-factorial}.  Applying
\cref{lem:factorial-analytic-criterion} in dimension $2N$ (with the joint
multiindex $(\alpha,\beta)$) yields
$\widetilde H_{s,\Omega}\in C^\omega(\Omega\times\Omega)$.

Finally, if $x\ne y$, both Mellin integrals may be separated and
\[
 \widetilde H_{s,\Omega}(x,y)
 =\frac1{\Gamma(s)}\int_0^\infty t^{s-1}g_t(x-y)\,\dd t
  -K_{s,\Omega}(x,y)
 =\Phi_s(x-y)-K_{s,\Omega}(x,y).
\]
Therefore $\widetilde H_{s,\Omega}$ is the required extension.  Its uniqueness
follows from the identity theorem for real-analytic functions on the connected
set $\Omega\times\Omega$.  Composing with the analytic diagonal embedding $y\mapsto(y,y)$ shows that
$R_{s,\Omega}$ is real analytic on $\Omega$.
\end{proof}

\begin{proof}[Proof of \cref{thm:robin-real-main}]
By \cref{lem:real-order-analytic}, the regular part extends real analytically
to $\Omega\times\Omega$ for every $s>0$ with $N>2s$.  In particular,
$H_{s,\Omega}$ extends continuously across the diagonal and
$R_{s,\Omega}$ is real analytic in $\Omega$.  Hence
\cref{prop:robin-transfer}, with the normalization
\eqref{eq:real-fundamental-intro}, gives convexity, the lower bound
\eqref{eq:robin-real-bound-intro}, boundary divergence, strict convexity,
and uniqueness of the critical point.  The unique critical point is therefore
also the unique global minimizer.
\end{proof}

\begin{proof}[Proof of \cref{cor:fractional-robin-center,cor:navier-robin-center}]
The first corollary is \cref{thm:robin-real-main} specialized to
$0<s<1$, where $A^s=(-\Delta_D)^s$ is the spectral fractional Dirichlet
Laplacian.  The second is the specialization $s=p\in\mathbb N$ together
with the Navier interpretation of the integer spectral powers recorded
below.
\end{proof}

\begin{remark}
\label{rem:integer-navier-interpretation}
When $s=p\in\mathbb N$, the spectral power $A^p$ is the iterated
Dirichlet Laplacian with
\[
 D(A^p)=\{u\in D(A):Au\in D(A),\ldots,A^{p-1}u\in D(A)\}.
\]
For $u\in D(A^p)$ one has $A^ju\in D(A)\subset H_0^1(\Omega)$ for
$j=0,\ldots,p-1$.  Since $Au=-\Delta u$ distributionally, the conditions
$A^ju\in H_0^1(\Omega)$ are exactly the Navier conditions
$(-\Delta)^ju=0$ on $\partial\Omega$ in the natural trace sense; for
sufficiently regular functions they agree with the classical pointwise
conditions.  Thus $K_{p,\Omega}=G_{p,\Omega}$ is the Green kernel of this
iterated-Dirichlet (Navier) realization and $R_{p,\Omega}$ in
\cref{def:spectral-robin-real} is its Robin function in the range $N>2p$.  In particular, this statement does not concern the clamped boundary
conditions $u=\partial_\nu u=0$ for the biharmonic operator.

At integer orders, analyticity can also be obtained by a local argument.
After cancellation of the common diagonal singularity,
\[
 (-\Delta_x)^pH_{p,\Omega}=0,
 \qquad
 (-\Delta_y)^pH_{p,\Omega}=0
 \quad\text{in }\mathscr D'(\Omega\times\Omega),
\]
so
\[
 \bigl[(-\Delta_x)^p+(-\Delta_y)^p\bigr]H_{p,\Omega}=0.
\]
The operator on the left is strongly elliptic with analytic constant
coefficients; interior analytic regularity therefore recovers the same
real-analytic extension obtained from the heat-kernel defect argument.  Although this observation is not used above, it clarifies the relation
between the spectral formulation and the classical Navier problem at integer
orders.
\end{remark}

\section{Second-order rigidity of spectral fractional Robin functions}
\label{sec:fractional-hessian}

By \cref{cor:fractional-robin-center}, the spectral fractional Robin
function already has a unique critical point on every bounded convex domain
in the range $N>2s$.  In this section we prove the
second-order assertion in \cref{thm:fractional-hessian-intro}, namely global
positive definiteness of its Hessian on $C^{2,\vartheta}$ convex domains.
Throughout this section $0<s<1$, $N\ge2$, $N>2s$, $0<\vartheta<1$, and
$\Omega\subset\R^N$ is a bounded convex domain of class $C^{2,\vartheta}$.  We use the
kernel, regular part and Robin function introduced in
\cref{sec:robin-transfer}; in particular $H_{s,\Omega}$ is real
analytic across the diagonal.

Set
\begin{equation}\label{eq:kappa}
 \kappa_s:=2^{2s-1}\frac{\Gamma(s)}{\Gamma(1-s)}>0,
 \qquad
 \mathcal C_\Omega:=\Omega\times(0,\infty),
 \qquad
 \Sigma_\Omega:=\partial\Omega\times(0,\infty).
\end{equation}
The extension variable is denoted by $z>0$, and we define
\begin{equation}\label{eq:extension-G}
 \mathcal G_s(x,z;y)
 :=\frac1{\Gamma(s)}\int_0^\infty
 \tau^{s-1}e^{-z^2/(4\tau)}p_\Omega(\tau,x,y)\,\dd\tau.
\end{equation}
This is the Stinga--Torrea extension of $K_{s,\Omega}(\cdot,y)$; thus
\begin{align}
 -\operatorname{div}_{x,z}\bigl(z^{1-2s}\nabla_{x,z}\mathcal G_s\bigr)
 &=0 &&\text{in }\mathcal C_\Omega,\label{eq:ext-PDE}\\
 \mathcal G_s&=0 &&\text{on }\Sigma_\Omega,\label{eq:ext-lat}\\
 -\kappa_s z^{1-2s}\partial_z\mathcal G_s(\cdot,z;y)
 &\rightharpoonup\delta_y &&\text{as }z\downarrow0.\label{eq:ext-Neumann}
\end{align}
The conormal limit in \eqref{eq:ext-Neumann} will also be derived directly
from the heat representation below.  For $x\ne y$,
\begin{equation}\label{eq:ext-trace}
 \mathcal G_s(x,0;y)=K_{s,\Omega}(x,y).
\end{equation}
For $\xi\in\partial\Omega$ define
\begin{equation}\label{eq:P-def}
 \mathcal P_s(\xi,z;y):=-\partial_{\nu_\xi}\mathcal G_s(\xi,z;y),
\end{equation}
and for $\eta\in\R^N$ define the simultaneous-translation field
\begin{equation}\label{eq:U-def}
 \mathcal U_{\eta,y}(x,z)
 :=(\partial_{x,\eta}+\partial_{y,\eta})\mathcal G_s(x,z;y).
\end{equation}
On $\partial\Omega$ write
\begin{equation}\label{eq:a-etaT}
 a:=\nu\cdot\eta,
 \qquad
 \eta_T:=\eta-a\nu,
\end{equation}
and use the convention
\begin{equation}\label{eq:II-conv}
 \II(X,Y):=\langle D_X\nu,Y\rangle,
 \qquad H_{\partial\Omega}:=\tr_{T\partial\Omega}\II.
\end{equation}
Hence $\II\ge0$ and $H_{\partial\Omega}\ge0$ on a convex domain.

\begin{proposition}
\label{prop:main}
For every $y\in\Omega$ and every $\eta\in\R^N$,
\begin{equation}\label{eq:main-identity}
\begin{aligned}
 D^2R_{s,\Omega}(y)[\eta,\eta]
 ={}&2\kappa_s
 \int_{\mathcal C_\Omega}
 z^{1-2s}|\nabla_{x,z}\mathcal U_{\eta,y}|^2\,\dd x\,\dd z\\
 &+\kappa_s
 \int_{\Sigma_\Omega}
 z^{1-2s}
 \Bigl[\II(\eta_T,\eta_T)+H_{\partial\Omega}(\nu\cdot\eta)^2\Bigr]
 \mathcal P_s^2\,\dd S\,\dd z.
\end{aligned}
\end{equation}
In particular,
\begin{equation}\label{eq:Hess-positive}
 D^2R_{s,\Omega}(y)>0
 \qquad\text{for every }y\in\Omega.
\end{equation}
\end{proposition}

\begin{remark}
\label{rem:classical-translation-curvature}
The energy--curvature structure of \cref{prop:main} is inherited from
the classical identity of Li--Liu--Ma \cite{LiLiuMa}.  The proof below
establishes its weighted counterpart directly on the extension cylinder:
the conormal measure and translation field at $z=0$ replace the classical
interior pole calculation.  The resulting positivity is the starting point
for \cref{sec:real-order-doubling}.
\end{remark}

The proof of the Hessian identity requires the following parabolic
boundary regularity estimate.  For a flat lateral boundary with homogeneous Dirichlet data, we use Lieberman's local
$H^{2+\vartheta}$ Schauder estimate
\cite[Theorem~4.22, p.~70]{Lieberman1996}.  The stability of the hypotheses
under an $H^{2+\vartheta}$ change of variables, as required when flattening a
$C^{2,\vartheta}$ boundary, is discussed in \cite[p.~76]{Lieberman1996}.
The standard localization and approximation procedure from Chapter~V,
together with the flat-boundary existence and up-to-boundary
$H^{2+\vartheta}$ regularity result
\cite[Theorem~5.13, pp.~93--94]{Lieberman1996}, provides the required
boundary regularity without any a priori assumption on second boundary
derivatives of the caloric function.  In particular, no derivative of the
boundary of order greater than two enters the argument.

\begin{lemma}
\label{lem:boundary-heat}
Let $0<\vartheta<1$, let $\Omega$ be a bounded domain of class
$C^{2,\vartheta}$, and let $K\Subset\Omega$.  Put
$d:=\dist(K,\partial\Omega)>0$.  There exists $\rho_0>0$ such that, for every
pair of multiindices $\mu,\beta$ with $|\mu|\le1$ and $|\beta|\le2$, there
are constants $C,c>0$ for which
\begin{equation}\label{eq:heat-small-boundary}
 \sup_{\substack{x\in\overline\Omega,\,\dist(x,\partial\Omega)\le\rho_0\\
                  y\in K}}
 |D_y^\mu D_x^\beta p_\Omega(t,x,y)|
 \le C t^{-(N+|\mu|+|\beta|)/2}e^{-c/t},
 \qquad0<t\le1.
\end{equation}
The derivatives in $x$ extend continuously to the lateral boundary.  Moreover,
after changing $C,c$ if necessary,
\begin{equation}\label{eq:heat-large-boundary}
 \sup_{\substack{x\in\overline\Omega,\,\dist(x,\partial\Omega)\le\rho_0\\
                  y\in K}}
 |D_y^\mu D_x^\beta p_\Omega(t,x,y)|
 \le Ce^{-ct},
 \qquad t\ge1.
\end{equation}
For small time the constants depend only on $N,\vartheta,d$ and on a fixed
uniform $C^{2,\vartheta}$ boundary character of $\Omega$; no boundary
derivative of order larger than two is used.
\end{lemma}

The proof is given in Appendix~\ref{app:boundary-heat}.  After flattening
a local boundary graph $x_N=\varphi(x')$, the principal
coefficients of the transformed heat operator involve only $D\varphi$,
whereas its first-order coefficients involve $D^2\varphi$.  Parabolic scaling preserves uniform ellipticity and gives
uniform $C^{0,\vartheta}$ coefficient bounds; the flat-boundary Schauder
estimate therefore controls the required $x$ derivatives up to order two.
The exponential factor $e^{-c/t}$ comes only from the positive separation
between the compact set of pole locations and the lateral boundary.  In particular, derivatives of the boundary
parametrization of order greater than two are not required.

\subsection{Boundary regularity and the translation field}

Recall $q_\Omega(t,x,y)=g_t(x-y)-p_\Omega(t,x,y)$ from
\cref{sec:robin-transfer}.  The interior estimates required below follow from
\cref{lem:finite-order-heat-defect}.  In particular, for any fixed
$K_0\Subset\Omega$ and for all derivatives of total order at most two,
\begin{equation}\label{eq:q-mixed-finite}
 \sup_{x,y\in K_0}|D_x^\alpha D_y^\beta q_\Omega(t,x,y)|
 \le Ct^{-P}e^{-c/t},
 \qquad 0<t\le1,
\end{equation}
and the corresponding one-endpoint estimates hold with one variable ranging
over all of $\Omega$.  We combine these interior estimates with the boundary estimate
\cref{lem:boundary-heat}.

\begin{lemma}
\label{lem:extension-bounds}
Let $K\Subset\Omega$ and let $|\eta|=1$.  There exist $C,c>0$, uniform for
$y\in K$, such that the following hold.

\begin{enumerate}[label=\textup{(\roman*)},leftmargin=2.7em]
\item For $\xi\in\partial\Omega$ and $0<z\le1$,
\begin{equation}\label{eq:P-small}
 |\mathcal P_s(\xi,z;y)|
 +|D_y\mathcal P_s(\xi,z;y)|
 +|\nabla_{\partial\Omega}\mathcal P_s(\xi,z;y)|
 \le C.
\end{equation}
For $z\ge1$ the left-hand side is bounded by $Ce^{-cz}$.

\item For $z\ge1$ the Green extension satisfies
\begin{equation}\label{eq:G-large}
 \sup_{\substack{x\in\overline\Omega\\y\in K}}
 \bigl(|\mathcal G_s(x,z;y)|+|\nabla_{x,z}\mathcal G_s(x,z;y)|\bigr)
 \le Ce^{-cz}.
\end{equation}

\item For every $x\in\overline\Omega$ and $0<z\le1$,
\begin{equation}\label{eq:U-small}
 |\mathcal U_{\eta,y}(x,z)|
 +|\nabla_x\mathcal U_{\eta,y}(x,z)|\le C,
 \qquad
 |\partial_z\mathcal U_{\eta,y}(x,z)|\le Cz.
\end{equation}
For $z\ge1$,
\begin{equation}\label{eq:U-large}
 |\mathcal U_{\eta,y}|+|\nabla_{x,z}\mathcal U_{\eta,y}|
 \le Ce^{-cz}.
\end{equation}

\item The field $\mathcal U_{\eta,y}$ is $C^\infty$ in the open cylinder
$\Omega\times(0,\infty)$; on every slab
$\overline\Omega\times[z_1,z_2]$ with $0<z_1<z_2<\infty$ it extends
continuously together with its first $x$ derivatives to the lateral boundary,
by the heat representation, \cref{lem:boundary-heat}, and dominated
convergence.  It extends continuously to
$\overline\Omega\times\{0\}$ and satisfies
\begin{equation}\label{eq:U-weighted-harmonic-early}
 -\operatorname{div}_{x,z}\bigl(z^{1-2s}\nabla_{x,z}\mathcal U_{\eta,y}\bigr)=0
 \qquad\text{in }\mathcal C_\Omega.
\end{equation}
Its trace satisfies the quadratic convergence estimate
\begin{equation}\label{eq:U-trace-rate}
 \|\mathcal U_{\eta,y}(\cdot,z)
 -\mathcal U_{\eta,y}(\cdot,0)\|_{L^\infty(\Omega)}
 \le Cz^2,
 \qquad0<z\le1.
\end{equation}
Moreover,
\begin{equation}\label{eq:U-energy}
 \int_{\mathcal C_\Omega}
 z^{1-2s}|\nabla_{x,z}\mathcal U_{\eta,y}|^2\,\dd x\,\dd z<\infty,
\end{equation}
and
\begin{equation}\label{eq:U-bottom-flux-zero}
 \lim_{z\downarrow0}
 \|z^{1-2s}\partial_z\mathcal U_{\eta,y}(\cdot,z)\|_{L^\infty(\Omega)}=0.
\end{equation}
\end{enumerate}
\end{lemma}

\begin{proof}
From \eqref{eq:extension-G} and differentiation under the integral sign,
justified initially for $z>0$, one has on the lateral boundary
\begin{equation}\label{eq:P-heat}
 \mathcal P_s(\xi,z;y)
 =-\frac1{\Gamma(s)}\int_0^\infty
 \tau^{s-1}e^{-z^2/(4\tau)}
 \partial_{\nu_\xi}p_\Omega(\tau,\xi,y)\,\dd\tau.
\end{equation}
The same formula holds after one $y$ derivative.  For a tangential vector
$\tau$, differentiating $-\nu\cdot\nabla_xp_\Omega$ gives a linear
combination of $D_x^2p_\Omega[\tau,\nu]$ and
$\nabla_xp_\Omega\cdot D_\tau\nu$.  Since
$\nu\in C^{1,\vartheta}(\partial\Omega)$ and
\cref{lem:boundary-heat} controls the first two $x$ derivatives, the same
heat representation and domination apply to
$\nabla_{\partial\Omega}\mathcal P_s$.  By \cref{lem:boundary-heat}, for $0<\tau\le1$ each resulting
integrand is bounded by
\[
 C\tau^{s-1-M}e^{-c/\tau},
\]
which is integrable at $0$ for every fixed $M$.  For $\tau\ge1$ it is
bounded by $C\tau^{s-1}e^{-c\tau}$, which is integrable at infinity.  Since
$e^{-z^2/(4\tau)}\le1$, this proves \eqref{eq:P-small}.

For $z\ge1$, split the $\tau$ integral into $(0,1)$ and $(1,\infty)$.  On
$(0,1)$,
\[
 e^{-z^2/(4\tau)}\le e^{-z^2/4}.
\]
On $(1,\infty)$, for some $c_1,c_2>0$,
\[
 c_1\tau+\frac{z^2}{4\tau}
 \ge c_2z
\]
by the arithmetic--geometric mean inequality.  Hence
\[
 e^{-c_1\tau-z^2/(4\tau)}
 \le e^{-c_2z}e^{-c_1\tau/2},
\]
and the claimed exponential decay follows.

We also require the corresponding large-$z$ estimate for the Green
extension.  For $z\ge1$, split \eqref{eq:extension-G} into
$\tau\in(0,1)$ and $\tau\in(1,\infty)$.  On $(0,1)$, the Gaussian bound
$p_\Omega\le g_\tau$ gives the zeroth-order estimate.  For one $x$
derivative, cover $\overline\Omega$ by a fixed boundary collar and an interior
compact set.  The scale-invariant boundary estimate used in
\cref{lem:boundary-heat} on the collar, and the standard interior parabolic
derivative estimate on the compact set, imply the global polynomial bound
\[
 \sup_{\substack{x\in\overline\Omega\\y\in K}}
 |\nabla_xp_\Omega(\tau,x,y)|\le C\tau^{-M_1},
 \qquad0<\tau\le1.
\]
No off-diagonal exponential factor is needed here.  A $z$ derivative of the
extension kernel introduces the factor $z/(2\tau)$.  Consequently every
small-time term needed for
$|\mathcal G_s|+|\nabla_{x,z}\mathcal G_s|$ is bounded by an integral of the
form
\[
 C z^j\int_0^1\tau^{-M}e^{-z^2/(4\tau)}\,\dd\tau,
 \qquad j\in\{0,1\}.
\]
With $u=z^2/(4\tau)$ this is bounded by $Ce^{-c z^2}$ for $z\ge1$ (after
absorbing the resulting polynomial factor in $z$ into the Gaussian decay).
On $(1,\infty)$, the fixed-time derivative bounds and semigroup
factorization used in the proof of \cref{lem:boundary-heat}, together with
the spectral gap, give the required finite-order bounds by
$C\tau^M e^{-c_0\tau}$, while
\[
 c_0\tau+\frac{z^2}{4\tau}\ge c_1z.
\]
After retaining half of the factor $e^{-c_0\tau}$ for integrability, this
yields \eqref{eq:G-large}.

We turn to the simultaneous translation field.  Let
\[
 S_\eta:=\partial_{x,\eta}+\partial_{y,\eta}.
\]
Because the free heat kernel depends only on $x-y$,
\begin{equation}\label{eq:Sg-zero}
 S_\eta g_\tau(x-y)=0.
\end{equation}
Consequently, whenever $x,y$ belong to a fixed interior compact set,
\[
 S_\eta p_\Omega=-S_\eta q_\Omega.
\]
Choose a thin boundary collar as in \cref{lem:boundary-heat}, and let
$K_0\Subset\Omega$ be a compact set containing $K$ together with the
complement of a slightly thinner collar.  On $K_0\times K_0$,
\eqref{eq:q-mixed-finite} gives
\begin{equation}\label{eq:Sp-interior}
 |S_\eta p_\Omega(\tau,x,y)|
 +|\nabla_xS_\eta p_\Omega(\tau,x,y)|
 \le C\tau^{-M}e^{-c/\tau},
 \qquad0<\tau\le1.
\end{equation}
In the boundary collar the same estimate follows from
\cref{lem:boundary-heat}.  Therefore
\begin{equation}\label{eq:Sp-global-small}
 \sup_{\substack{x\in\overline\Omega\\y\in K}}
 \Bigl(|S_\eta p_\Omega|+|\nabla_xS_\eta p_\Omega|\Bigr)
 \le C\tau^{-M}e^{-c/\tau},
 \qquad0<\tau\le1.
\end{equation}
The large-time semigroup argument in the proof of
\cref{lem:boundary-heat}, together with fixed-time smoothing, yields
\begin{equation}\label{eq:Sp-global-large}
 \sup_{\substack{x\in\overline\Omega\\y\in K}}
 \Bigl(|S_\eta p_\Omega|+|\nabla_xS_\eta p_\Omega|\Bigr)
 \le Ce^{-c\tau},
 \qquad\tau\ge1.
\end{equation}

Differentiating \eqref{eq:extension-G} in the simultaneous translation
direction and using \eqref{eq:Sp-global-small}--\eqref{eq:Sp-global-large}
gives, for $z>0$,
\begin{equation}\label{eq:U-heat}
 \mathcal U_{\eta,y}(x,z)
 =\frac1{\Gamma(s)}\int_0^\infty
 \tau^{s-1}e^{-z^2/(4\tau)}
 S_\eta p_\Omega(\tau,x,y)\,\dd\tau.
\end{equation}
The same bounds with the exponential factor removed from the extension
kernel provide an integrable majorant at $z=0$.  We therefore define
\begin{equation}\label{eq:U-trace-definition}
 \mathcal U_{\eta,y}(x,0)
 :=\frac1{\Gamma(s)}\int_0^\infty
 \tau^{s-1}S_\eta p_\Omega(\tau,x,y)\,\dd\tau,
\end{equation}
and dominated convergence shows that \eqref{eq:U-heat} converges to this
trace uniformly for $x\in\overline\Omega$, $y\in K$.  The same estimates
after one $x$ derivative prove the first part of \eqref{eq:U-small}, and the
same large-$z$ argument used above proves \eqref{eq:U-large} for
$\mathcal U$ and $\nabla_x\mathcal U$.

For $z>0$ the factor $e^{-z^2/(4\tau)}$ removes the small-time diagonal
singularity, so \eqref{eq:U-heat} may be differentiated in $x,y,z$ as needed.
Differentiating the extension equation \eqref{eq:ext-PDE} in the horizontal
variables $x$ and $y$, or equivalently differentiating the heat
representation, gives \eqref{eq:U-weighted-harmonic-early}.  This establishes the weighted harmonic equation before it is used in the
Green identities below.

Differentiating \eqref{eq:U-heat} with respect to $z$ gives
\begin{equation}\label{eq:Uz-heat}
 \partial_z\mathcal U_{\eta,y}(x,z)
 =-\frac{z}{2\Gamma(s)}\int_0^\infty
 \tau^{s-2}e^{-z^2/(4\tau)}
 S_\eta p_\Omega(\tau,x,y)\,\dd\tau.
\end{equation}
The integral on the right is uniformly bounded for $0\le z\le1$: near
$\tau=0$ the factor $e^{-c/\tau}$ from
\eqref{eq:Sp-global-small} absorbs every negative power of $\tau$, while
for $\tau\ge1$ the spectral-gap decay is integrable.  Hence
$|\partial_z\mathcal U|\le Cz$, proving the second estimate in
\eqref{eq:U-small}.  For $0<\delta<z$, the fundamental theorem of calculus
gives
\[
 \mathcal U(x,z)-\mathcal U(x,\delta)
 =\int_\delta^z\partial_\zeta\mathcal U(x,\zeta)\,\dd\zeta.
\]
Letting $\delta\downarrow0$ by the uniform trace convergence following
\eqref{eq:U-trace-definition}, and using $|\partial_\zeta\mathcal U|\le C\zeta$,
proves \eqref{eq:U-trace-rate}.  For $z\ge1$, the estimate for $\partial_z\mathcal U$ follows by splitting
the integral in \eqref{eq:Uz-heat} at $\tau=1$, using
\eqref{eq:Sp-global-small} for $0<\tau\le1$ and the spectral-gap estimate
for $\tau\ge1$.

Finally,
\[
 \int_0^1 z^{1-2s}|\nabla_x\mathcal U|^2\,\dd z
 \le C\int_0^1z^{1-2s}\,\dd z<\infty
\]
because $s<1$, while
\[
 \int_0^1 z^{1-2s}|\partial_z\mathcal U|^2\,\dd z
 \le C\int_0^1z^{3-2s}\,\dd z<\infty.
\]
The tail $z\ge1$ is integrable by exponential decay.  This proves
\eqref{eq:U-energy}.  Moreover,
\[
 z^{1-2s}|\partial_z\mathcal U|
 \le Cz^{2-2s}\longrightarrow0
\]
uniformly in $x$, proving \eqref{eq:U-bottom-flux-zero}.
\end{proof}

We next identify the conormal contribution at $z=0$ and establish the
Green identity needed in the Hessian calculation.

The next lemma gives an explicit form of the boundary limit at $z=0$ in the
Green representation.  This form is needed because the test function itself
depends on $z$.

\begin{lemma}
\label{lem:first-green}
Fix $K\Subset\Omega$, $y\in K$, and $\eta\in\R^N$.  For $z>0$ define the
positive measure
\begin{equation}\label{eq:mu-z}
 \dd\mu_z(x)
 :=-\kappa_s z^{1-2s}
 \partial_z\mathcal G_s(x,z;y)\,\dd x.
\end{equation}
Then
\begin{equation}\label{eq:mu-mass}
 0\le\mu_z(\Omega)\le1,
 \qquad
 \mu_z(\Omega)\longrightarrow1
 \quad(z\downarrow0),
\end{equation}
and
\begin{equation}\label{eq:mu-weak}
 \mu_z\rightharpoonup\delta_y
 \qquad\text{weakly against }C(\overline\Omega).
\end{equation}
Moreover,
\begin{align}
 \mathcal U_{\eta,y}&=-a\mathcal P_s
 &&\text{on }\Sigma_\Omega,\label{eq:U-lateral}\\
 \mathcal U_{\eta,y}(y,0)&=-DR_{s,\Omega}(y)[\eta],\label{eq:U-pole}
\end{align}
and the Robin function satisfies the first-variation formula
\begin{equation}\label{eq:first-variation}
 DR_{s,\Omega}(y)[\eta]
 =\kappa_s
 \int_{\Sigma_\Omega}
 z^{1-2s}a(\xi)\mathcal P_s(\xi,z;y)^2
 \,\dd S_\xi\,\dd z.
\end{equation}
\end{lemma}

\begin{proof}
Differentiate \eqref{eq:extension-G} in $z$:
\[
 -\partial_z\mathcal G_s(x,z;y)
 =\frac{z}{2\Gamma(s)}\int_0^\infty
 \tau^{s-2}e^{-z^2/(4\tau)}p_\Omega(\tau,x,y)\,\dd\tau.
\]
Hence $\mu_z$ is positive.  If
\[
 S_\Omega(\tau,y):=\int_\Omega p_\Omega(\tau,x,y)\,\dd x,
\]
then $0\le S_\Omega\le1$ by the sub-Markov property.  A change of variables
$r=z^2/(4\tau)$ gives
\begin{equation}\label{eq:mu-mass-formula}
 \mu_z(\Omega)
 =\frac1{\Gamma(1-s)}\int_0^\infty
 r^{-s}e^{-r}
 S_\Omega\!\left(\frac{z^2}{4r},y\right)\dd r.
\end{equation}
The normalization in \eqref{eq:kappa} has been used here.  This proves the
upper bound in \eqref{eq:mu-mass}.  Since
$S_\Omega(\tau,y)\to1$ as $\tau\downarrow0$ for every interior point $y$,
dominated convergence in \eqref{eq:mu-mass-formula} yields
$\mu_z(\Omega)\to1$.

More generally, for $\varphi\in C(\overline\Omega)$,
\begin{align}
 \int_\Omega\varphi(x)\,\dd\mu_z(x)
 =\frac1{\Gamma(1-s)}\int_0^\infty
 r^{-s}e^{-r}
 \left[
 \int_\Omega
 p_\Omega\!\left(\frac{z^2}{4r},x,y\right)\varphi(x)\,\dd x
 \right]\dd r.
 \label{eq:mu-test}
\end{align}
For each fixed $r>0$, the bracket converges to $\varphi(y)$ as
$z\downarrow0$.  This is the standard small-time concentration of the
Dirichlet heat kernel at an interior point; equivalently it follows from the
killed Brownian representation and continuity of $\varphi$ at $y$.  The
bracket is bounded by $\|\varphi\|_\infty$, so dominated convergence in
\eqref{eq:mu-test} proves \eqref{eq:mu-weak}.

We next prove the lateral boundary data.  Since
$\mathcal G_s(\xi,z;y)=0$ for every $y\in\Omega$, one has
$\partial_{y,\eta}\mathcal G_s=0$ on $\Sigma_\Omega$.  Also
\[
 \nabla_x\mathcal G_s(\xi,z;y)
 =\partial_\nu\mathcal G_s(\xi,z;y)\nu(\xi)
 =-\mathcal P_s(\xi,z;y)\nu(\xi),
\]
whence \eqref{eq:U-lateral}.

We next identify the value of the translation-field trace at the pole without
introducing an extension of the regular part.  For $x\ne y$, the trace
formula \eqref{eq:ext-trace} and \eqref{eq:U-trace-definition} give
\[
 \mathcal U_{\eta,y}(x,0)
 =S_\eta K_{s,\Omega}(x,y).
\]
Using $K_{s,\Omega}=\Phi_s-H_{s,\Omega}$ and the fact that the Riesz kernel
$\Phi_s(x-y)$ depends only on $x-y$, we obtain
\[
 \mathcal U_{\eta,y}(x,0)
 =-S_\eta H_{s,\Omega}(x,y),\qquad x\ne y.
\]
The left-hand side is continuous at $x=y$ by
\eqref{eq:U-trace-definition}, while the right-hand side extends continuously
there because $H_{s,\Omega}$ is $C^2$ across the diagonal.  Letting $x\to y$
therefore yields
\[
 \mathcal U_{\eta,y}(y,0)
 =-\left.(\partial_{x,\eta}+\partial_{y,\eta})H_{s,\Omega}(x,y)\right|_{x=y}
 =-DR_{s,\Omega}(y)[\eta],
\]
which proves \eqref{eq:U-pole}.

It remains to prove \eqref{eq:first-variation}.  Set
$\mathcal U:=\mathcal U_{\eta,y}$ and
$\mathcal G:=\mathcal G_s(\cdot,\cdot;y)$, and fix $0<\varepsilon<1<R$.
To justify the Green identities at the lateral boundary, first work on
$\Omega_\delta\times(\varepsilon,R)$, where
$\Omega_\delta=\{x\in\Omega:\operatorname{dist}(x,\partial\Omega)>\delta\}$.
For small $\delta>0$ this domain has $C^2$ boundary in the horizontal
variables, and both functions are smooth near its closure.  The normal map
$T_\delta(\xi)=\xi-\delta\nu(\xi)$ parametrizes $\partial\Omega_\delta$;
its surface Jacobian tends uniformly to $1$, and
$\nu_\delta\circ T_\delta\to\nu$ uniformly.  The heat representations and
\cref{lem:boundary-heat,lem:extension-bounds} give continuous boundary traces
of $\mathcal G,\mathcal U$ and their first derivatives on the fixed slab.
Thus, as $\delta\downarrow0$, the lateral, top, bottom, and volume integrals
in both Green identities converge to those on
$\Omega\times(\varepsilon,R)$.  In particular, Green's second identity gives
\begin{equation}\label{eq:second-green-truncated}
 0=\int_{\partial(\Omega\times(\varepsilon,R))}
 z^{1-2s}\bigl(
 \mathcal U\,\partial_n\mathcal G
 -\mathcal G\,\partial_n\mathcal U
 \bigr)\,\dd S.
\end{equation}
The contribution from $z=R$ tends to zero as $R\to\infty$ by
\eqref{eq:G-large} and \eqref{eq:U-large}; the polynomial factor
$R^{1-2s}$ is dominated by the exponential decay.  On the lateral
boundary $\mathcal G=0$ and
$\partial_n\mathcal G=\partial_\nu\mathcal G=-\mathcal P_s$, so the lateral
term converges absolutely to
\[
 -\int_{\Sigma_\Omega}z^{1-2s}\mathcal U\mathcal P_s\,\dd S\,\dd z.
\]

On the boundary $z=\varepsilon$, the outward normal is $-e_z$.  The first
corresponding term equals
\begin{equation}\label{eq:bottom-first}
 -\varepsilon^{1-2s}
 \int_\Omega\mathcal U(x,\varepsilon)\partial_z\mathcal G(x,\varepsilon)\,\dd x
 =\frac1{\kappa_s}
 \int_\Omega\mathcal U(x,\varepsilon)\,\dd\mu_\varepsilon(x).
\end{equation}
The dependence of the test function on $\varepsilon$ is handled as follows:
\begin{align*}
 &\left|
 \int\mathcal U(x,\varepsilon)\,\dd\mu_\varepsilon
 -\mathcal U(y,0)
 \right|\\
 &\quad\le
 \|\mathcal U(\cdot,\varepsilon)-\mathcal U(\cdot,0)\|_\infty
 \mu_\varepsilon(\Omega)
 +\left|
 \int\mathcal U(x,0)\,\dd\mu_\varepsilon-
 \mathcal U(y,0)
 \right|.
\end{align*}
The first term is $O(\varepsilon^2)$ by
\eqref{eq:U-trace-rate} and \eqref{eq:mu-mass}; the second tends to zero by
\eqref{eq:mu-weak}.  Thus \eqref{eq:bottom-first} tends to
$\kappa_s^{-1}\mathcal U(y,0)$.

The second term on $z=\varepsilon$ is
\[
 \varepsilon^{1-2s}
 \int_\Omega\mathcal G(x,\varepsilon)\partial_z\mathcal U(x,\varepsilon)
 \,\dd x.
\]
By \cref{lem:extension-bounds},
$|\partial_z\mathcal U|\le C\varepsilon$.  Moreover
$e^{-\varepsilon^2/(4\tau)}\le1$ and the heat-kernel domination
$p_\Omega\le g_\tau$ imply
\[
 0\le\mathcal G_s(x,\varepsilon;y)
 \le K_{s,\Omega}(x,y)
 \le \Phi_s(x-y),
\]
and $\Phi_s(\cdot-y)\in L^1(\Omega)$ because $N>2s$.  Therefore the second
term on $z=\varepsilon$ is $O(\varepsilon^{2-2s})$ and tends to zero.

Letting first $R\to\infty$ and then $\varepsilon\downarrow0$ in
\eqref{eq:second-green-truncated} gives
\[
 -\int_{\Sigma_\Omega}z^{1-2s}\mathcal U\mathcal P_s\,\dd S\,\dd z
 +\frac1{\kappa_s}\mathcal U(y,0)=0.
\]
Using \eqref{eq:U-lateral} and \eqref{eq:U-pole} yields
\eqref{eq:first-variation}.
\end{proof}

\subsection{Energy and boundary-curvature identities}

\begin{lemma}
\label{lem:energy-green}
For every $y\in K\Subset\Omega$ and $\eta\in\R^N$,
\begin{equation}\label{eq:energy-green}
 \int_{\Sigma_\Omega}z^{1-2s}
 \mathcal U_{\eta,y}\,\partial_\nu\mathcal U_{\eta,y}
 \,\dd S\,\dd z
 =
 \int_{\mathcal C_\Omega}z^{1-2s}
 |\nabla_{x,z}\mathcal U_{\eta,y}|^2\,\dd x\,\dd z.
\end{equation}
\end{lemma}

\begin{proof}
By \eqref{eq:U-weighted-harmonic-early}, the translation field is weighted
harmonic in $\mathcal C_\Omega$.  The estimate \eqref{eq:U-bottom-flux-zero} gives
$z^{1-2s}\partial_z\mathcal U(\cdot,z)\to0$ uniformly as $z\downarrow0$.  Multiply \eqref{eq:U-weighted-harmonic-early}
by $\mathcal U$ and integrate over
$\Omega\times(\varepsilon,R)$.  Integration by parts, justified by the
interior parallel-domain approximation in \cref{lem:first-green}, gives
\begin{align}
 \int_{\Omega\times(\varepsilon,R)}
 z^{1-2s}|\nabla\mathcal U|^2
 ={}&\int_{\partial\Omega\times(\varepsilon,R)}
 z^{1-2s}\mathcal U\partial_\nu\mathcal U\,\dd S\,\dd z\notag\\
 &+R^{1-2s}\int_\Omega
 \mathcal U(x,R)\partial_z\mathcal U(x,R)\,\dd x\notag\\
 &-\varepsilon^{1-2s}\int_\Omega
 \mathcal U(x,\varepsilon)\partial_z\mathcal U(x,\varepsilon)\,\dd x.
 \label{eq:energy-truncated}
\end{align}
The term on $z=R$ tends to zero exponentially.  The term on $z=\varepsilon$ is bounded by
$C\varepsilon^{2-2s}|\Omega|$ by \eqref{eq:U-small}, hence tends to zero
because $s<1$.  Finally, \eqref{eq:U-energy} and monotone convergence of the
nonnegative interior integrand allow $\varepsilon\downarrow0$ and
$R\to\infty$.  The lateral integral converges absolutely by
\cref{lem:extension-bounds}.  This proves \eqref{eq:energy-green}.
\end{proof}

We now express the pole derivative of the lateral Poisson kernel in
terms of the second fundamental form of the boundary.

\begin{lemma}
\label{lem:boundary-hessian}
For $y\in K\Subset\Omega$, $\eta\in\R^N$, and
$(\xi,z)\in\Sigma_\Omega$,
\begin{equation}\label{eq:P-y-derivative}
 \partial_{y,\eta}\mathcal P_s
 =-\partial_\nu\mathcal U_{\eta,y}
 -\nabla_{\partial\Omega}\mathcal P_s\cdot\eta_T
 +aH_{\partial\Omega}\mathcal P_s.
\end{equation}
Moreover,
\begin{equation}\label{eq:div-geometry}
 \operatorname{div}_{\partial\Omega}(a\eta_T)
 =\II(\eta_T,\eta_T)-H_{\partial\Omega}a^2.
\end{equation}
All terms in \eqref{eq:P-y-derivative} are integrable with the weight
$z^{1-2s}$ in the combinations used below.
\end{lemma}

\begin{proof}
Since $\mathcal P_s=-\partial_\nu\mathcal G_s$ and the normal field is
independent of the pole $y$,
\begin{align*}
 \partial_{y,\eta}\mathcal P_s
 &=-\partial_\nu(\partial_{y,\eta}\mathcal G_s)\\
 &=-\partial_\nu\mathcal U_{\eta,y}
   +\partial_\nu(\partial_{x,\eta}\mathcal G_s)\\
 &=-\partial_\nu\mathcal U_{\eta,y}
   +D_x^2\mathcal G_s[\nu,\eta].
\end{align*}
We compute the last term on the lateral boundary.  We first justify the
finite boundary differentiations used below.  For $|\beta|\le2$,
\cref{lem:boundary-heat} and the corresponding interior estimates give, on
every fixed boundary collar,
\[
 |D_x^\beta p_\Omega(t,x,y)|
 \le
 \begin{cases}
 Ct^{-M}e^{-c/t},&0<t\le1,\\
 Ce^{-ct},&t\ge1,
 \end{cases}
\]
uniformly for $y$ in compact subsets of $\Omega$; the derivatives extend
continuously to $\partial\Omega$.  After multiplication by
$t^{s-1}e^{-z^2/(4t)}$, these bounds are integrable in $t$ for every fixed
$z>0$.  Dominated convergence in the heat representation therefore gives
\[
 D_x^\beta\mathcal G_s(x,z;y)
 =\frac1{\Gamma(s)}\int_0^\infty
 t^{s-1}e^{-z^2/(4t)}D_x^\beta p_\Omega(t,x,y)\,\dd t,
 \qquad |\beta|\le2,
\]
with continuous boundary traces.  Hence
$\mathcal G_s(\cdot,z;y)\in C^2(\overline\Omega)$ for every $z>0$, locally
uniformly for $z$ away from $0$ and $y$ in compact subsets of $\Omega$.

We shall also use the weighted equation on the lateral boundary, so we record
explicitly the corresponding regularity in the extension variable.  Fix a
compact slab $0<z_0\le z\le z_1<\infty$.  Differentiating the heat
representation gives
\begin{align}
 \partial_z\mathcal G_s(x,z;y)
 &=-\frac{z}{2\Gamma(s)}\int_0^\infty
 t^{s-2}e^{-z^2/(4t)}p_\Omega(t,x,y)\,\dd t,\label{eq:Gz-heat}\\
 \partial_{zz}\mathcal G_s(x,z;y)
 &=\frac1{\Gamma(s)}\int_0^\infty
 t^{s-1}e^{-z^2/(4t)}
 \left(\frac{z^2}{4t^2}-\frac1{2t}\right)
 p_\Omega(t,x,y)\,\dd t.\label{eq:Gzz-heat}
\end{align}
For $0<t\le1$, the standard Gaussian bound
$p_\Omega(t,x,y)\le C t^{-N/2}e^{-|x-y|^2/(Ct)}$ and $z\ge z_0$ show that
the absolute values of the two integrands are bounded, respectively, by
\[
 C t^{s-2-N/2}e^{-z_0^2/(8t)}
 \quad\text{and}\quad
 C\bigl(t^{s-2-N/2}+t^{s-3-N/2}\bigr)e^{-z_0^2/(8t)},
\]
which are integrable at $t=0$ for every $s>0$.  For $t\ge1$, the
large-time spectral decay of $p_\Omega$ gives an integrable majorant.  The same estimates hold locally uniformly for $x$ up to
$\partial\Omega$ and $y$ in compact subsets of $\Omega$.  Hence dominated
convergence in \eqref{eq:Gz-heat}--\eqref{eq:Gzz-heat}, together with the
continuous Dirichlet trace $p_\Omega(t,\xi,y)=0$ for
$\xi\in\partial\Omega$ and $t>0$, shows that
$\partial_z\mathcal G_s$ and $\partial_{zz}\mathcal G_s$ extend
continuously to the lateral boundary, locally uniformly on positive-$z$
slabs, and their lateral traces satisfy
\begin{equation}\label{eq:z-derivative-lateral-traces}
 \partial_z\mathcal G_s(\xi,z;y)
 =\partial_{zz}\mathcal G_s(\xi,z;y)=0,
 \qquad (\xi,z)\in\Sigma_\Omega.
\end{equation}
Together with the $x$-derivative boundary regularity established above,
this justifies all boundary differentiations below, including the restriction
of the weighted equation to $\Sigma_\Omega$, under the stated
$C^{2,\vartheta}$ hypothesis.

Let $\tau$ be tangent to $\partial\Omega$.  Since
$\mathcal G_s(\cdot,z;y)=0$ on $\partial\Omega$,
\[
 \nabla_x\mathcal G_s=-\mathcal P_s\nu.
\]
Tangentially differentiating
$\partial_\nu\mathcal G_s=-\mathcal P_s$ gives
\[
 -\nabla_{\partial\Omega}\mathcal P_s\cdot\tau
 =D_x^2\mathcal G_s[\tau,\nu]
 +\nabla_x\mathcal G_s\cdot D_\tau\nu.
\]
The last term vanishes because $D_\tau\nu$ is tangent while
$\nabla_x\mathcal G_s$ is normal.  Hence
\begin{equation}\label{eq:mixed-Hessian}
 D_x^2\mathcal G_s[\nu,\tau]
 =-\nabla_{\partial\Omega}\mathcal P_s\cdot\tau.
\end{equation}

For the normal--normal component, fix $z>0$.  Expanding the weighted equation
\eqref{eq:ext-PDE} gives
\begin{equation}\label{eq:weighted-expanded}
 \Delta_x\mathcal G_s
 +\partial_{zz}\mathcal G_s
 +\frac{1-2s}{z}\partial_z\mathcal G_s=0.
\end{equation}
By \eqref{eq:z-derivative-lateral-traces},
\[
 \partial_z\mathcal G_s=\partial_{zz}\mathcal G_s=0
 \qquad\text{on }\Sigma_\Omega.
\]
Likewise the tangential trace is identically zero, hence
$\Delta_{\partial\Omega}\mathcal G_s=0$.  With the convention
\eqref{eq:II-conv}, the boundary Laplacian decomposition is
\[
 \Delta_x f
 =D_x^2f[\nu,\nu]+H_{\partial\Omega}\partial_\nu f
 +\Delta_{\partial\Omega}(f|_{\partial\Omega}).
\]
Applying this to \eqref{eq:weighted-expanded} on the lateral boundary yields
\[
 D_x^2\mathcal G_s[\nu,\nu]
 +H_{\partial\Omega}\partial_\nu\mathcal G_s=0,
\]
and therefore
\begin{equation}\label{eq:normal-Hessian}
 D_x^2\mathcal G_s[\nu,\nu]=H_{\partial\Omega}\mathcal P_s.
\end{equation}
Writing $\eta=\eta_T+a\nu$ and combining
\eqref{eq:mixed-Hessian}--\eqref{eq:normal-Hessian} proves
\eqref{eq:P-y-derivative}.

We also verify \eqref{eq:div-geometry}.  Since
$a=\nu\cdot\eta$, for every tangent vector $\tau$,
\[
 \tau(a)=\langle D_\tau\nu,\eta_T\rangle.
\]
The Weingarten map is self-adjoint, hence
\[
 \nabla_{\partial\Omega}a\cdot\eta_T
 =\II(\eta_T,\eta_T).
\]
Moreover, for a local orthonormal tangent frame $\{e_i\}_{i=1}^{N-1}$,
\begin{align*}
 \operatorname{div}_{\partial\Omega}\eta_T
 &=\sum_i\langle D_{e_i}(\eta-a\nu),e_i\rangle
 =-a\sum_i\langle D_{e_i}\nu,e_i\rangle
 =-aH_{\partial\Omega}.
\end{align*}
Consequently
\[
 \operatorname{div}_{\partial\Omega}(a\eta_T)
 =\nabla_{\partial\Omega}a\cdot\eta_T
 +a\operatorname{div}_{\partial\Omega}\eta_T
 =\II(\eta_T,\eta_T)-H_{\partial\Omega}a^2.
\]
The weighted integrability of the terms follows from
\cref{lem:extension-bounds}, since $z^{1-2s}\in L^1(0,1)$ and all relevant
quantities decay exponentially for $z\to\infty$.
\end{proof}

With these identities in hand, we prove the Hessian formula and
its strict positivity.

\begin{proof}[Proof of \cref{prop:main}]
By \cref{lem:first-green},
\begin{equation}\label{eq:grad-start}
 DR_{s,\Omega}(y)[\eta]
 =\kappa_s\int_{\Sigma_\Omega}
 z^{1-2s}a\mathcal P_s^2\,\dd S\,\dd z.
\end{equation}
The factor $a=\nu\cdot\eta$ does not depend on the pole $y$.  Choose a compact
neighborhood $K\Subset\Omega$ of the fixed pole such that
$y+t\eta\in K$ for all sufficiently small $|t|$.  We claim that
\eqref{eq:grad-start} may be differentiated once more in the direction
$\eta$.  Indeed, by \cref{lem:extension-bounds}, uniformly for the pole in
this compact set,
\[
 z^{1-2s}|\mathcal P_s\partial_{y,\eta}\mathcal P_s|
 \le
 \begin{cases}
 Cz^{1-2s},&0<z\le1,\\
 Ce^{-cz},&z\ge1,
 \end{cases}
\]
and the right-hand side is integrable because $0<s<1$.  The derivative
$\partial_{y,\eta}\mathcal P_s$ is jointly continuous for $z>0$ by its heat
representation and the same domination.  Dominated convergence therefore
yields
\begin{equation}\label{eq:Hess-first}
 D^2R_{s,\Omega}(y)[\eta,\eta]
 =2\kappa_s\int_{\Sigma_\Omega}
 z^{1-2s}a\mathcal P_s
 \partial_{y,\eta}\mathcal P_s\,\dd S\,\dd z.
\end{equation}
Insert \eqref{eq:P-y-derivative}:
\begin{align}
 D^2R_{s,\Omega}(y)[\eta,\eta]
 ={}&-2\kappa_s\int_{\Sigma_\Omega}
 z^{1-2s}a\mathcal P_s\partial_\nu\mathcal U\,\dd S\,\dd z
 \label{eq:Hess-split}\\
 &-2\kappa_s\int_{\Sigma_\Omega}
 z^{1-2s}a\mathcal P_s
 \nabla_{\partial\Omega}\mathcal P_s\cdot\eta_T\,\dd S\,\dd z\notag\\
 &+2\kappa_s\int_{\Sigma_\Omega}
 z^{1-2s}H_{\partial\Omega}a^2\mathcal P_s^2\,\dd S\,\dd z.\notag
\end{align}

On the lateral boundary \eqref{eq:U-lateral} gives
$\mathcal U=-a\mathcal P_s$.  Hence the first term in
\eqref{eq:Hess-split} is
\[
 2\kappa_s\int_{\Sigma_\Omega}
 z^{1-2s}\mathcal U\partial_\nu\mathcal U\,\dd S\,\dd z
 =2\kappa_s\int_{\mathcal C_\Omega}
 z^{1-2s}|\nabla\mathcal U|^2\,\dd x\,\dd z
\]
by \cref{lem:energy-green}.

For the second term, fix $z>0$.  Since $\partial\Omega$ has no boundary,
\begin{align*}
 &-2\int_{\partial\Omega}
 a\mathcal P_s\nabla_{\partial\Omega}\mathcal P_s\cdot\eta_T\,\dd S\\
 &\qquad=-\int_{\partial\Omega}
 a\eta_T\cdot\nabla_{\partial\Omega}(\mathcal P_s^2)\,\dd S\\
 &\qquad=\int_{\partial\Omega}
 \mathcal P_s^2\operatorname{div}_{\partial\Omega}(a\eta_T)\,\dd S\\
 &\qquad=\int_{\partial\Omega}
 \bigl[\II(\eta_T,\eta_T)-H_{\partial\Omega}a^2\bigr]
 \mathcal P_s^2\,\dd S,
\end{align*}
where \cref{lem:boundary-hessian} was used in the last line.  Multiplying by
$\kappa_s z^{1-2s}$ and integrating in $z$ is justified by
\cref{lem:extension-bounds}.  Combining this term with the last term in
\eqref{eq:Hess-split} gives
\[
 \kappa_s\int_{\Sigma_\Omega}z^{1-2s}
 \bigl[\II(\eta_T,\eta_T)+H_{\partial\Omega}a^2\bigr]\mathcal P_s^2
 \,\dd S\,\dd z.
\]
This proves the identity \eqref{eq:main-identity}.

Assume now that $\Omega$ is convex and $\eta\ne0$.  Then
$\II\ge0$ and $H_{\partial\Omega}\ge0$, so the second term in
\eqref{eq:main-identity} is nonnegative.  It remains to show that the first
term is strictly positive.  Suppose to the contrary that
\[
 \int_{\mathcal C_\Omega}z^{1-2s}|\nabla\mathcal U|^2\,\dd x\,\dd z=0.
\]
Since the weight is strictly positive for $z>0$ and the cylinder is
connected, $\mathcal U$ is constant in $\mathcal C_\Omega$.

The linear functional $x\mapsto x\cdot\eta$ attains its maximum and minimum
on the compact convex body $\overline\Omega$ at points
$\xi_+,\xi_-\in\partial\Omega$.  Since a $C^{2,\vartheta}$ boundary is $C^1$, the tangent hyperplane and the
outward unit normal are unique at every boundary point; the supporting-plane
condition at these extrema therefore gives
\[
 \nu(\xi_+)=\frac{\eta}{|\eta|},
 \qquad
 \nu(\xi_-)=-\frac{\eta}{|\eta|}.
\]
Fix any $z_0>0$.  The positive heat-kernel representation
\eqref{eq:extension-G} gives $\mathcal G_s>0$ in
$\Omega\times(0,\infty)$.  Near $(\xi_\pm,z_0)$ the weight
$z^{1-2s}$ is smooth and strictly positive, so the extension equation is
uniformly elliptic there.  Since a $C^{2,\vartheta}$ boundary satisfies
the interior sphere condition, the classical Hopf boundary lemma applies at
$(\xi_\pm,z_0)$ and gives
\[
 \mathcal P_s(\xi_\pm,z_0;y)>0.
\]
Using \eqref{eq:U-lateral},
\[
 \mathcal U(\xi_+,z_0)
 =-|\eta|\mathcal P_s(\xi_+,z_0;y)<0,
 \qquad
 \mathcal U(\xi_-,z_0)
 =|\eta|\mathcal P_s(\xi_-,z_0;y)>0,
\]
contradicting constancy of $\mathcal U$.  Hence the weighted Dirichlet
energy is strictly positive for every $\eta\ne0$, and
\eqref{eq:Hess-positive} follows.
\end{proof}

The positivity conclusion in \eqref{eq:Hess-positive}, together with
\cref{cor:fractional-robin-center}, completes the proof of
\cref{thm:fractional-hessian-intro}.

\section{Finite-part doubling and nondegeneracy at real spectral orders}
\label{sec:real-order-doubling}

We establish a Hessian comparison between the Robin functions of spectral
orders $\sigma$ and $2\sigma$, starting from the positivity result in
\cref{thm:fractional-hessian-intro}.
For $N>4\sigma$, the truncated square energies $J_{\sigma,L}$ defined below
have a convex negative power and satisfy
\[
 J_{\sigma,L}
 =A_L-R_{2\sigma,\Omega}
 +\varepsilon_L B_{N,\sigma}R_{\sigma,\Omega}
 +o_{C^2_{\rm loc}}(\varepsilon_L),
 \qquad A_L\to\infty,\quad\varepsilon_L\downarrow0,
 \quad B_{N,\sigma}>0.
\]
The leading finite part is identified using
$A^{-\sigma}A^{-\sigma}=A^{-2\sigma}$; the next coefficient controls
directions in which the Hessian of $R_{2\sigma,\Omega}$ may be degenerate.
At a critical point, the quadratic gradient term arising from the chain
rule is of lower order.  At a general point it requires
$A_L^{-1}=o(\varepsilon_L)$, which gives the stronger condition
$N>6\sigma$.

The convexity input is \cref{thm:spectral-main}:
\begin{equation}\label{eq:input-joint-convexity}
 (x,y)\longmapsto K_{\sigma,\Omega}(x,y)^{-1/(N-2\sigma)}
 \quad\text{is convex on }\Omega\times\Omega.
\end{equation}
The analytic regularity in \cref{lem:real-order-analytic} provides the
necessary mixed derivatives in a neighborhood of the diagonal.  We first
prove convexity of the truncated energies and their $C^2$ expansion, then formulate the
Hessian transfer and iterate it to prove \cref{thm:real-order-nondegeneracy-intro}.

For $\sigma>0$ and $N>4\sigma$, define
\begin{equation}\label{eq:real-J}
 J_{\sigma,L}(y)
 :=\int_\Omega T_L(K_{\sigma,\Omega}(x,y))^2\,\dd x,
 \qquad
 T_L(r):=\min\{r,L\}.
\end{equation}
For pointwise evaluations, we define the following continuous extension
to the diagonal:
\begin{equation}\label{eq:capped-kernel-diagonal}
 \widehat K_{\sigma,L}(x,y):=
 \begin{cases}
 T_L(K_{\sigma,\Omega}(x,y)),&x\ne y,\\
 L,&x=y.
 \end{cases}
\end{equation}
Indeed, \cref{lem:Ks-basic} gives
$K_{\sigma,\Omega}(x,y)\to+\infty$ as $x\to y$, so
$\widehat K_{\sigma,L}$ is continuous on $\Omega\times\Omega$.  Changing the
value on the diagonal does not affect the integral in \eqref{eq:real-J}.

\begin{proposition}
\label{prop:real-J-convex}
Let $\sigma>0$ and $N>4\sigma$.  Then for every $L>0$,
\begin{equation}\label{eq:real-J-convex}
 y\longmapsto J_{\sigma,L}(y)^{-1/(N-4\sigma)}
\end{equation}
is convex on $\Omega$.
\end{proposition}

\begin{proof}
Put $m:=N-2\sigma$ and
\[
 q_0:=-\frac1m<0.
\]
We first establish the truncated pointwise inequality on
$\Omega\times\Omega$.  Suppose initially that both endpoint pairs
$(x_i,y_i)$, $i=0,1$, are off the diagonal and put
$x_t=(1-t)x_0+tx_1$, $y_t=(1-t)y_0+ty_1$.  If $x_t\ne y_t$, then
\eqref{eq:input-joint-convexity}, equivalently its negative-power mean form,
gives
\begin{equation}\label{eq:real-kernel-mean}
 K_{\sigma,\Omega}(x_t,y_t)
 \ge M_{q_0}\bigl(K_{\sigma,\Omega}(x_0,y_0),
                   K_{\sigma,\Omega}(x_1,y_1);t\bigr).
\end{equation}
By \cref{lem:negative-mean-truncation}, this implies
\begin{equation}\label{eq:real-truncated-square-pointwise}
 \widehat K_{\sigma,L}(x_t,y_t)^2
 \ge
 M_r\!\left(
 \widehat K_{\sigma,L}(x_0,y_0)^2,
 \widehat K_{\sigma,L}(x_1,y_1)^2;t
 \right),
 \qquad
 r:=\frac{q_0}{2}=-\frac1{2m}.
\end{equation}
If instead $x_t=y_t$, then the left-hand side of
\eqref{eq:real-truncated-square-pointwise} equals $L^2$, whereas each
truncated endpoint value is at most $L$; every power mean lies between its two
arguments, so the right-hand side is at most $L^2$.  Thus
\eqref{eq:real-truncated-square-pointwise} also holds in this case.

If one or both endpoint pairs lie on the diagonal, approximate those endpoint
pairs by off-diagonal pairs in $\Omega\times\Omega$.  For every approximating
configuration the preceding argument applies, including the possibility that
the interpolated pair lies on the diagonal.  Passing to the limit using the
continuity of $\widehat K_{\sigma,L}$ from
\eqref{eq:capped-kernel-diagonal} proves
\eqref{eq:real-truncated-square-pointwise} for arbitrary
$x_0,x_1,y_0,y_1\in\Omega$.

We next verify the pointwise hypothesis of the Borell--Brascamp--Lieb inequality.  Fix
$y_0,y_1\in\Omega$, set $y_t=(1-t)y_0+ty_1$, and define on all of $\R^N$
\[
 f_i(x):=\mathbf1_\Omega(x)\widehat K_{\sigma,L}(x,y_i)^2,
 \qquad i=0,1,t.
\]
If $x_0,x_1\in\Omega$, convexity of $\Omega$ gives
$x_t=(1-t)x_0+tx_1\in\Omega$, and
\eqref{eq:real-truncated-square-pointwise} yields
\[
 f_t(x_t)\ge M_r(f_0(x_0),f_1(x_1);t).
\]
If at least one of $x_0,x_1$ lies outside $\Omega$, then the corresponding
endpoint value is zero.  Since $r<0$, the right-hand side is zero by
\eqref{eq:negative-mean-zero-convention}, whereas $f_t\ge0$.  Thus the BBL
pointwise hypothesis holds for every $x_0,x_1\in\R^N$.

The functions $f_i$ are integrable, because
$0\le f_i\le L^2\mathbf1_\Omega$ and $|\Omega|<\infty$, and their integrals
are exactly $J_{\sigma,L}(y_i)$.  The admissibility condition for
\cref{lem:BBL-general} in dimension $N$ is
\[
 r>-\frac1N
 \quad\Longleftrightarrow\quad
 2m>N
 \quad\Longleftrightarrow\quad
 N>4\sigma,
\]
which is precisely our hypothesis.  Hence BBL gives
\[
 J_{\sigma,L}(y_t)
 \ge M_q(J_{\sigma,L}(y_0),J_{\sigma,L}(y_1);t),
 \qquad
 q:=\frac{r}{1+Nr}
 =-\frac1{2m-N}
 =-\frac1{N-4\sigma}<0.
\]
Raising this inequality to the negative power $q$ reverses its direction and
therefore gives
\[
 J_{\sigma,L}(y_t)^q
 \le (1-t)J_{\sigma,L}(y_0)^q
    +tJ_{\sigma,L}(y_1)^q.
\]
Since $q=-1/(N-4\sigma)$, this is exactly the convexity asserted in
\eqref{eq:real-J-convex}.
\end{proof}

The finite part of the square energy is identified by the following
spectral factorization.

\begin{lemma}
\label{lem:real-factorization}
Let $\sigma>0$ and $N>4\sigma$.  Then for every $x\ne y$,
\begin{align}
 \Phi_{2\sigma}(x-y)
 &=\int_{\R^N}\Phi_\sigma(x-z)\Phi_\sigma(z-y)\,\dd z,
 \label{eq:real-free-factorization}\\
 K_{2\sigma,\Omega}(x,y)
 &=\int_\Omega K_{\sigma,\Omega}(x,z)
 K_{\sigma,\Omega}(z,y)\,\dd z.
 \label{eq:real-domain-factorization}
\end{align}
Both integrals are absolutely convergent.  If
$\widetilde K_{\sigma,y}$ denotes the extension of
$K_{\sigma,\Omega}(\cdot,y)$ by zero outside $\Omega$, then
\begin{equation}\label{eq:real-finitepart-identity}
 R_{2\sigma,\Omega}(y)
 =\int_{\R^N}
 \Bigl[\Phi_\sigma(z-y)^2-
       \widetilde K_{\sigma,y}(z)^2\Bigr] \,\dd z.
\end{equation}
The integrand in \eqref{eq:real-finitepart-identity} belongs to
$L^1(\R^N)$.
\end{lemma}

\begin{proof}
Use the free heat-kernel representation
\[
 \Phi_\sigma(x-y)
 =\frac1{\Gamma(\sigma)}\int_0^\infty
 t^{\sigma-1}g_t(x-y)\,\dd t.
\]
All kernels are nonnegative.  Tonelli and the free semigroup property give
\begin{align*}
 &\int_{\R^N}\Phi_\sigma(x-z)\Phi_\sigma(z-y)\,\dd z\\
 &\quad=\frac1{\Gamma(\sigma)^2}
 \int_0^\infty\!\int_0^\infty
 t^{\sigma-1}u^{\sigma-1}g_{t+u}(x-y)\,\dd t\,\dd u.
\end{align*}
Put $v=t+u$ and $\theta=t/v$.  Since
$\dd t\,\dd u=v\,\dd v\,\dd\theta$,
\begin{align*}
 \text{right-hand side}
 &=\frac{B(\sigma,\sigma)}{\Gamma(\sigma)^2}
 \int_0^\infty v^{2\sigma-1}g_v(x-y)\,\dd v\\
 &=\frac1{\Gamma(2\sigma)}
 \int_0^\infty v^{2\sigma-1}g_v(x-y)\,\dd v
 =\Phi_{2\sigma}(x-y).
\end{align*}
Because $N>4\sigma$, the final quantity is finite for $x\ne y$, which also
proves absolute convergence.  Replacing $g_t$ by $p_\Omega(t,\cdot,\cdot)$ and using the Dirichlet
semigroup property, the identical change of variables yields
\eqref{eq:real-domain-factorization}.

Subtracting the two factorizations gives, for $x\ne y$,
\begin{align}
 H_{2\sigma,\Omega}(x,y)
 ={}&\int_{\R^N\setminus\Omega}
 \Phi_\sigma(x-z)\Phi_\sigma(z-y)\,\dd z\notag\\
 &+\int_\Omega
 \Bigl[\Phi_\sigma(x-z)\Phi_\sigma(z-y)
 -K_{\sigma,\Omega}(x,z)K_{\sigma,\Omega}(z,y)\Bigr]\,\dd z.
 \label{eq:real-H2-split}
\end{align}
Inside $\Omega$, write $K_\sigma=\Phi_\sigma-H_\sigma$ and use symmetry:
\begin{align}
 &\Phi_\sigma(x-z)H_{\sigma,\Omega}(z,y)
 +H_{\sigma,\Omega}(x,z)\Phi_\sigma(z-y)
 -H_{\sigma,\Omega}(x,z)H_{\sigma,\Omega}(z,y).
 \label{eq:real-H2-expanded}
\end{align}
Choose $\delta>0$ with $B_{2\delta}(y)\Subset\Omega$ and let
$x\to y$ with $x\in B_{\delta/2}(y)$.  Since
$0\le H_{\sigma,\Omega}(x,z)\le\Phi_\sigma(x-z)$ in $\Omega$, each term of
\eqref{eq:real-H2-expanded}, as well as the exterior integrand in
\eqref{eq:real-H2-split}, is bounded in absolute value, for
$z\notin B_\delta(y)$, by
\[
 \Phi_\sigma(x-z)\Phi_\sigma(z-y)
 \le C|z-y|^{-2m},\qquad m=N-2\sigma,
\]
where $|x-z|\ge|z-y|/2$.  This majorant is integrable on
$\R^N\setminus B_\delta(y)$ because $2m>N$, so dominated convergence
applies there.  In $B_\delta(y)$, the last two terms of
\eqref{eq:real-H2-expanded} pass to the
limit by dominated convergence because $H_{\sigma,\Omega}$ is bounded and
\[
 \Phi_\sigma(z-y)=a_{N,\sigma}|z-y|^{2\sigma-N}
 \in L^1_{\rm loc}(\R^N).
\]
For the first term, the translated singularity is handled by translation
continuity in $L^1_{\rm loc}$: with $h=x-y$, $w=z-y$, and
$f(w)=\Phi_\sigma(w)\mathbf1_{B_{2\delta}(0)}(w)$,
\begin{align*}
 &\left|\int_{B_\delta(y)}
 [\Phi_\sigma(x-z)-\Phi_\sigma(y-z)]H_{\sigma,\Omega}(z,y)\,\dd z\right|\\
 &\qquad\le
 \|H_{\sigma,\Omega}(\cdot,y)\|_{L^\infty(B_\delta(y))}
 \|\tau_hf-f\|_{L^1(\R^N)}\longrightarrow0.
\end{align*}
Thus $x\to y$ in \eqref{eq:real-H2-split} gives
\begin{align*}
 R_{2\sigma,\Omega}(y)
 ={}&\int_{\R^N\setminus\Omega}\Phi_\sigma(z-y)^2\,\dd z\\
 &+\int_\Omega
 \Bigl[2\Phi_\sigma(z-y)H_{\sigma,\Omega}(z,y)
       -H_{\sigma,\Omega}(z,y)^2\Bigr]\,\dd z,
\end{align*}
which is exactly \eqref{eq:real-finitepart-identity}.

Finally, near $z=y$ the interior integrand is bounded in absolute value by
$C(|z-y|^{2\sigma-N}+1)$, whose radial integral is
$C\int_0^1(r^{2\sigma-1}+r^{N-1})\,\dd r<\infty$.  Outside $\Omega$ there
is no pole and, at infinity,
\[
 \Phi_\sigma(z-y)^2=O(|z|^{4\sigma-2N}),
\]
whose radial integral behaves as
$\int^\infty r^{4\sigma-N-1}\,\dd r$ and converges precisely when
$N>4\sigma$.  This proves absolute integrability.
\end{proof}

\subsection{Asymptotic expansion of truncated energies}

We use the following convention for differentiable remainders.  Let
$K\Subset\Omega$, let $k\in\mathbb N_0$, and let $\eta_\lambda>0$ be a
scalar scale as $\lambda\to\infty$.  We write
\begin{equation}\label{eq:Ck-big-O-convention}
 f_\lambda=O_{C^k(K)}(\eta_\lambda)
\end{equation}
if $f_\lambda$ is $C^k$ on a neighborhood of $K$ for all sufficiently
large $\lambda$ and there is a constant $C_K$, independent of $\lambda$, such that
\[
 \max_{|\mu|\le k}\sup_{y\in K}|D_y^\mu f_\lambda(y)|
 \le C_K\eta_\lambda.
\]
We write $f_\lambda=o_{C^k(K)}(\eta_\lambda)$ if the same left-hand side,
divided by $\eta_\lambda$, tends to zero.  If the functions also depend on
an auxiliary parameter, such as $\theta\in\mathbb S^{N-1}$, the notation
means that the estimate is uniform in that parameter.  We use the identical
convention for families indexed by $L\to\infty$.

\begin{lemma}
\label{lem:real-high-level}
Let $\sigma>0$, $N\ge2$, and $N>4\sigma$.  Put
\begin{equation}\label{eq:real-exponents-pre}
 m:=N-2\sigma,
 \qquad
 \gamma:=\frac Nm=1+\frac{2\sigma}{m},
 \qquad
 A_{N,\sigma}:=\frac{|\mathbb S^{N-1}|}{N}
 a_{N,\sigma}^{N/m}.
\end{equation}
For $y\in\Omega$ and $\lambda>0$, let
\begin{equation}\label{eq:real-V}
 V_y(\lambda):=
 |\{x\in\Omega:K_{\sigma,\Omega}(x,y)>\lambda\}|.
\end{equation}
Then for every $K\Subset\Omega$, as $\lambda\to\infty$,
\begin{equation}\label{eq:real-V-shifted}
 V_y(\lambda)
 =A_{N,\sigma}
  (\lambda+R_{\sigma,\Omega}(y))^{-\gamma}
 +O_{C^2(K)}
 \left(\lambda^{-\gamma-1-1/m}\right),
\end{equation}
uniformly for $y\in K$.  Consequently,
\begin{equation}\label{eq:real-V-expanded}
 V_y(\lambda)
 =A_{N,\sigma}\lambda^{-\gamma}
 -A_{N,\sigma}\gamma R_{\sigma,\Omega}(y)
  \lambda^{-\gamma-1}
 +O_{C^2(K)}
 \left(\lambda^{-\gamma-1-1/m}\right).
\end{equation}
\end{lemma}

\begin{proof}
Fix $K\Subset\Omega$ and choose
$0<\delta<\dist(K,\partial\Omega)/4$.  Set
\[
 h(y,r,\theta):=
 H_{\sigma,\Omega}(y+r\theta,y),
 \qquad \theta\in\mathbb S^{N-1}.
\]
By \cref{lem:real-order-analytic}, all mixed derivatives needed below are
bounded on the corresponding compact tube.  Since
$h(y,0,\theta)=R_{\sigma,\Omega}(y)$ identically in $y$,
Taylor's formula in $r$ gives, for $|\mu|\le2$,
\begin{equation}\label{eq:real-h-diagonal}
 D_y^\mu\bigl[h(y,r,\theta)-R_{\sigma,\Omega}(y)\bigr]
 =O(r)
\end{equation}
uniformly in $y\in K$, $\theta\in\mathbb S^{N-1}$.

The domination $p_\Omega\le g_t$ implies
\begin{equation}\label{eq:real-K-upper}
 0<K_{\sigma,\Omega}(x,y)
 \le\Phi_\sigma(x-y)=a_{N,\sigma}|x-y|^{-m}.
\end{equation}
Thus, for all sufficiently large $\lambda$, every superlevel set in
\eqref{eq:real-V} lies in $B_\delta(y)$, uniformly for $y\in K$.  After
shrinking $\delta$ if necessary,
\[
 \partial_rK_{\sigma,\Omega}(y+r\theta,y)
 =-m a_{N,\sigma}r^{-m-1}-\partial_rh(y,r,\theta)<0
\]
for $0<r<\delta$.  Hence the superlevel set is radially star-shaped about
$y$ and its boundary is uniquely represented by
\begin{equation}\label{eq:real-level-equation}
 K_{\sigma,\Omega}(y+r_\lambda(y,\theta)\theta,y)=\lambda.
\end{equation}

Put
\begin{equation}\label{eq:real-rho-q}
 S:=\lambda+R_{\sigma,\Omega}(y),
 \qquad
 \rho:=\left(\frac{a_{N,\sigma}}S\right)^{1/m},
 \qquad
 r_\lambda=\rho q_\lambda,
\end{equation}
and
\[
 g(y,r,\theta):=h(y,r,\theta)-R_{\sigma,\Omega}(y).
\]
The level equation becomes
\begin{equation}\label{eq:real-F-equation}
 \mathcal F(q,y,\theta,\lambda)
 :=q^{-m}-1-\frac{g(y,\rho q,\theta)}S=0.
\end{equation}
The uniform boundedness of $h$ first implies $q_\lambda\to1$ uniformly; hence
$q_\lambda\in[1/2,2]$ for large $\lambda$.  By
\eqref{eq:real-h-diagonal}, $g=O(\rho)$ on this interval, so
\begin{equation}\label{eq:real-q-zero}
 q_\lambda-1=O(\rho/S).
\end{equation}
Moreover,
\[
 \partial_q\mathcal F
 =-mq^{-m-1}-\frac\rho S g_r(y,\rho q,\theta)
\]
is bounded away from zero uniformly for large $\lambda$.  Hence the
parameterized implicit-function theorem applies uniformly on
$K\times\mathbb S^{N-1}$: for all sufficiently large $\lambda$, the unique
radial solution of \eqref{eq:real-F-equation} defines a function
\[
 q_\lambda=q_\lambda(y,\theta)
\]
which is $C^2$ in $y$ and continuous jointly in $(y,\theta)$.  The local
implicit branches agree globally because the positive level radius in
\eqref{eq:real-level-equation} is unique.  In particular, the first and
second implicit-differentiation formulas used below are legitimate, and all
constants may be chosen uniformly in $\theta\in\mathbb S^{N-1}$.

We record the parameter derivatives because the $C^2$ remainder is essential
later.  Since $S\asymp\lambda$ and $R_{\sigma,\Omega}\in C^3(K)$,
\begin{equation}\label{eq:real-rho-derivatives}
 D_y\rho=O(\rho/S),
 \qquad
 D_y^2\rho=O(\rho/S).
\end{equation}
For fixed $q\in[1/2,2]$, the composite
$G_q(y):=g(y,\rho(y)q,\theta)$ satisfies, by
\eqref{eq:real-h-diagonal}, the chain rule, and
\eqref{eq:real-rho-derivatives},
\[
 G_q=O(\rho),
 \qquad D_yG_q=O(\rho),
 \qquad D_y^2G_q=O(\rho).
\]
Consequently, at fixed $q$,
\begin{equation}\label{eq:real-F-derivatives}
 D_y\mathcal F=O(\rho/S),
 \qquad
 D_y^2\mathcal F=O(\rho/S),
 \qquad
 D_{yq}\mathcal F=O(\rho/S),
 \qquad
 \partial_{qq}\mathcal F=O(1).
\end{equation}
The first and second implicit differentiation formulas applied to
\eqref{eq:real-F-equation} therefore yield
\begin{equation}\label{eq:real-q-C2}
 |q_\lambda-1|+|D_yq_\lambda|+|D_y^2q_\lambda|
 \le C_K\frac\rho S
 =O\left(\lambda^{-1-1/m}\right).
\end{equation}

The polar volume formula gives
\begin{align*}
 V_y(\lambda)
 &=\frac1N\int_{\mathbb S^{N-1}}r_\lambda(y,\theta)^N\,\dd\theta\\
 &=\frac{\rho^N}{N}
 \int_{\mathbb S^{N-1}}q_\lambda(y,\theta)^N\,\dd\theta.
\end{align*}
Set $e_\lambda(y,\theta):=q_\lambda(y,\theta)^N-1$.  Since
$q_\lambda\in[1/2,2]$, \eqref{eq:real-q-C2} and the first two derivatives
of $q\mapsto q^N$ give
\[
 |e_\lambda|+|D_ye_\lambda|+|D_y^2e_\lambda|
 \le C_K\frac\rho S.
\]
Writing
\[
 E_\lambda(y):=\frac{\rho^N}{N}
 \int_{\mathbb S^{N-1}}e_\lambda(y,\theta)\,\dd\theta,
\]
and using
$D_y\rho^N,D_y^2\rho^N=O(\rho^N/S)$, the product rule gives
\[
 |E_\lambda|+|D_yE_\lambda|+|D_y^2E_\lambda|
 \le C_K\frac{\rho^{N+1}}S.
\]
Therefore
\[
 V_y(\lambda)
 =\frac{|\mathbb S^{N-1}|}{N}\rho^N
 +O_{C^2(K)}\left(\frac{\rho^{N+1}}S\right),
\]
which is \eqref{eq:real-V-shifted}.  Finally, $N>4\sigma$ and $N\ge2$
imply
\[
 m=N-2\sigma>\frac N2\ge1,
\]
with strict inequality $m>1$.  Hence $1+1/m<2$, and the Taylor
expansion
\[
 (\lambda+R)^{-\gamma}
 =\lambda^{-\gamma}
 -\gamma R\lambda^{-\gamma-1}
 +O_{C^2(K)}(\lambda^{-\gamma-2})
\]
has a remainder strictly smaller than the geometric remainder in
\eqref{eq:real-V-shifted}.  This proves
\eqref{eq:real-V-expanded}.
\end{proof}

The superlevel-set analysis gives the following expansion, with
remainder controlled in $C^2$ on compact subsets.

\begin{proposition}
\label{prop:real-finitepart-expansion}
Let $\sigma>0$, $N\ge2$, and $N>4\sigma$.  Put
\begin{equation}\label{eq:real-alpha-beta}
 m:=N-2\sigma,
 \qquad
 \alpha:=\frac{N-4\sigma}{m},
 \qquad
 \beta:=\frac{2\sigma}{m}.
\end{equation}
Then for every $K\Subset\Omega$ there exists $L_K$ such that
$J_{\sigma,L}\in C^2(K)$ for $L\ge L_K$, and
\begin{equation}\label{eq:real-J-expansion}
 J_{\sigma,L}(y)
 =C_{N,\sigma}L^\alpha
 -R_{2\sigma,\Omega}(y)
 +B_{N,\sigma}R_{\sigma,\Omega}(y)L^{-\beta}
 +O_{C^2(K)}\left(L^{-\beta-1/m}\right),
\end{equation}
where
\begin{align}
 C_{N,\sigma}
 &:=\frac{2m|\mathbb S^{N-1}|}{N(N-4\sigma)}
 a_{N,\sigma}^{N/m}>0,
 \label{eq:real-C}\\
 B_{N,\sigma}
 &:=\frac{|\mathbb S^{N-1}|}{\sigma}
 a_{N,\sigma}^{N/m}>0.
 \label{eq:real-B}
\end{align}
\end{proposition}

\begin{proof}
Let
\[
 V_{\rm free}(\lambda)
 :=|\{x\in\R^N:\Phi_\sigma(x)>\lambda\}|.
\]
By homogeneity,
\begin{equation}\label{eq:real-V-free}
 V_{\rm free}(\lambda)
 =A_{N,\sigma}\lambda^{-\gamma},
 \qquad
 \gamma:=\frac Nm=1+\beta.
\end{equation}
Define the free truncated energy
\[
 J_{\sigma,L}^{\rm free}
 :=\int_{\R^N}\min\{\Phi_\sigma(x),L\}^2\,\dd x.
\]
Since $N>4\sigma$ is equivalent to $\gamma<2$, the layer-cake formula gives
\begin{equation}\label{eq:real-free-energy}
 J_{\sigma,L}^{\rm free}
 =2\int_0^L\lambda V_{\rm free}(\lambda)\,\dd\lambda
 =\frac{2A_{N,\sigma}}{2-\gamma}L^{2-\gamma}
 =C_{N,\sigma}L^\alpha,
\end{equation}
where $2-\gamma=\alpha$ and the last constant is exactly
\eqref{eq:real-C}.

Extend $K_{\sigma,\Omega}(\cdot,y)$ by zero to $\R^N$ and denote it by
$\widetilde K_{\sigma,y}$.  The heat-kernel domination gives
$0\le\widetilde K_{\sigma,y}\le\Phi_\sigma(\cdot-y)$.  For $0\le b\le a$,
\[
 0\le\min\{a,L\}^2-\min\{b,L\}^2\le a^2-b^2.
\]
The right-hand side is integrable here by
\cref{lem:real-factorization}.  Dominated convergence in the difference of
truncated energies therefore yields
\begin{equation}\label{eq:real-finite-limit}
 J_{\sigma,L}(y)-C_{N,\sigma}L^\alpha
 \longrightarrow -R_{2\sigma,\Omega}(y).
\end{equation}

To obtain the differentiable remainder, use
\[
 \min\{u,L\}^2
 =2\int_0^L\lambda\mathbf1_{\{u>\lambda\}}\,\dd\lambda.
\]
Hence
\begin{equation}\label{eq:real-layercake}
 J_{\sigma,L}(y)=2\int_0^L\lambda V_y(\lambda)\,\dd\lambda,
 \qquad
 J_{\sigma,L}^{\rm free}=2\int_0^L
 \lambda V_{\rm free}(\lambda)\,\dd\lambda.
\end{equation}
Put
$F_L:=J_{\sigma,L}-C_{N,\sigma}L^\alpha$.  Subtracting
\eqref{eq:real-layercake} at levels $M>L$ and using
\eqref{eq:real-finite-limit} as $M\to\infty$ gives the exact tail identity
\begin{equation}\label{eq:real-tail-identity}
 F_L(y)+R_{2\sigma,\Omega}(y)
 =-\int_L^\infty
 2\lambda\,[V_y(\lambda)-V_{\rm free}(\lambda)]\,\dd\lambda.
\end{equation}
By \cref{lem:real-high-level}, uniformly in $C^2(K)$,
\begin{align}
 2\lambda[V_y(\lambda)-V_{\rm free}(\lambda)]
 ={}&-2A_{N,\sigma}\gamma R_{\sigma,\Omega}(y)
 \lambda^{-\gamma}
 +\mathcal E(y,\lambda),
 \label{eq:real-tail-integrand}\\
 \|\mathcal E(\cdot,\lambda)\|_{C^2(K)}
 &\le C_K\lambda^{-\gamma-1/m}.
 \label{eq:real-tail-error}
\end{align}
Because $\gamma=1+\beta>1$, both the leading term and the error are
integrable on $[L,\infty)$.  To justify the $C^2$ tail rigorously, choose
$K\Subset\operatorname{int}K_1\Subset\Omega$ and perform the level-radius
construction on $K_1$.  The parameterized implicit-function theorem applied
to \eqref{eq:real-F-equation} makes $D_y^\mu V_y(\lambda)$ jointly continuous
for $|\mu|\le2$ and large $\lambda$.  Equations
\eqref{eq:real-tail-integrand}--\eqref{eq:real-tail-error} yield the
integrable majorant
\[
 \sup_{y\in K_1}
 \left|D_y^\mu\bigl(2\lambda[V_y(\lambda)
 -V_{\rm free}(\lambda)]\bigr)\right|
 \le C_{K_1}\bigl(\lambda^{-\gamma}
 +\lambda^{-\gamma-1/m}\bigr),
 \qquad |\mu|\le2.
\]
Successive difference quotients in $y$, controlled by this majorant via the
mean-value theorem, permit two applications of dominated convergence in
\eqref{eq:real-tail-identity}.  Hence the tail belongs to $C^2(K)$, its
first and second derivatives with respect to $y$ are obtained by differentiation under the
$\lambda$-integral, and $J_{\sigma,L}\in C^2(K)$ for large $L$.
Termwise integration gives
\begin{align*}
 F_L+R_{2\sigma,\Omega}
 &=\frac{2A_{N,\sigma}\gamma}{\gamma-1}
 R_{\sigma,\Omega}L^{-(\gamma-1)}
 +O_{C^2(K)}
 \left(L^{-(\gamma+1/m-1)}\right)\\
 &=B_{N,\sigma}R_{\sigma,\Omega}L^{-\beta}
 +O_{C^2(K)}(L^{-\beta-1/m}).
\end{align*}
Indeed,
\[
 \frac{2A_{N,\sigma}\gamma}{\gamma-1}
 =\frac{|\mathbb S^{N-1}|}{\sigma}
 a_{N,\sigma}^{N/m}=B_{N,\sigma}.
\]
Together with the definition of $F_L$, this proves
\eqref{eq:real-J-expansion}.
\end{proof}

\subsection{Hessian transfer and the nondegeneracy theorem}

\begin{proposition}
\label{prop:real-Hessian-transfer-abstract}
Let $U\subset\R^N$ be open and convex and $q<0$.  Suppose that $J_L:U\to(0,\infty)$ is such that $J_L^q$ is convex for all
sufficiently large $L$, and that on every $K\Subset U$,
\begin{equation}\label{eq:real-abstract-expansion}
 J_L=A_L-F+\varepsilon_LG+o_{C^2(K)}(\varepsilon_L),
 \qquad
 A_L\to\infty,
 \qquad
 \varepsilon_L\to0,
\end{equation}
where $F,G\in C^2(U)$ and $A_L>0$, $\varepsilon_L>0$ are scalar
quantities depending only on $L$, independent of the point in $U$.  The
$C^2(K)$ expansion means that, for every $K\Subset U$ and all sufficiently
large $L$, the remainder is $C^2$ on a neighborhood of $K$ and its
$C^2(K)$ norm divided by $\varepsilon_L$ tends to zero.  In particular,
$J_L$ is $C^2$ near each fixed compact set for all sufficiently large $L$.
Then the following statements hold:
\begin{enumerate}[label=\textup{(\roman*)},leftmargin=2.8em]
\item $F$ is convex;
\item if $y_*$ is a critical point of $F$ and $D^2G(y_*)>0$, then
      $D^2F(y_*)>0$;
\item if $A_L^{-1}=o(\varepsilon_L)$ and $D^2G>0$ throughout $U$, then
      $D^2F>0$ throughout $U$.
\end{enumerate}
\end{proposition}

\begin{proof}
Set
\[
 E_L:=J_L-A_L+F-\varepsilon_LG,
 \qquad
 s_L:=F-\varepsilon_LG-E_L,
\]
so that $J_L=A_L-s_L$ and
$E_L=o_{C^2_{\rm loc}}(\varepsilon_L)$.  The scalars $A_L$ and
$\varepsilon_L$ have zero spatial derivatives.  For every fixed
$K\Subset U$, the functions $s_L$ are uniformly bounded on $K$ for large
$L$.  Since $A_L\to\infty$, we may therefore assume
$|s_L|\le A_L/2$ on $K$.  Normalize the convex function $J_L^q$ by
\begin{equation}\label{eq:real-normalized-convex}
 \mathcal K_L
 :=\frac{J_L^q-A_L^q}{-qA_L^{q-1}}
 =\Psi_L(s_L),
 \qquad
 \Psi_L(r):=
 \frac{(A_L-r)^q-A_L^q}{-qA_L^{q-1}}.
\end{equation}
The denominator is positive because $q<0$ and $A_L>0$, and is
independent of the spatial variable.  Consequently $\mathcal K_L$ is
convex.  For $r<A_L$, direct differentiation gives
\[
 \Psi_L'(r)=\left(1-\frac r{A_L}\right)^{q-1}>0,
 \qquad
 \Psi_L''(r)=\frac{1-q}{A_L}
 \left(1-\frac r{A_L}\right)^{q-2}>0.
\]
Since $\Psi_L(0)=0$, these formulas and integration of $\Psi_L'$ imply,
uniformly on every fixed bounded $r$-interval,
\begin{equation}\label{eq:real-Psi}
 \Psi_L'(r)=1+O(A_L^{-1}),
 \qquad
 \Psi_L''(r)=O(A_L^{-1}),
 \qquad
 \Psi_L(r)=r+o(1).
\end{equation}
Thus $\mathcal K_L\to F$ locally uniformly, so the locally uniform limit of
the convex functions $\mathcal K_L$ is convex.  This proves (i).

For (ii), convexity gives $D^2F(y_*)\ge0$.  If it were not positive
definite, choose $0\ne\eta$ with
$D^2F(y_*)[\eta,\eta]=0$.  Since $\nabla F(y_*)=0$ and the remainder in
\eqref{eq:real-abstract-expansion} is $o_{C^2_{\rm loc}}(\varepsilon_L)$,
\begin{align}
 D^2s_L(y_*)[\eta,\eta]
 &=-\varepsilon_LD^2G(y_*)[\eta,\eta]+o(\varepsilon_L),
 \label{eq:transfer-center-D2s}\\
 \partial_\eta s_L(y_*)
 &=-\varepsilon_L\partial_\eta G(y_*)+o(\varepsilon_L)
 =O(\varepsilon_L).
 \label{eq:transfer-center-Ds}
\end{align}
The full second-order chain rule is
\begin{equation}\label{eq:transfer-full-chain-rule}
 D^2\mathcal K_L[\eta,\eta]
 =\Psi_L'(s_L)D^2s_L[\eta,\eta]
 +\Psi_L''(s_L)(\partial_\eta s_L)^2.
\end{equation}
Using \eqref{eq:real-Psi}, \eqref{eq:transfer-center-D2s}, and
\eqref{eq:transfer-center-Ds}, we obtain at $y_*$
\begin{align*}
 D^2\mathcal K_L[\eta,\eta]
 ={}&-\varepsilon_LD^2G(y_*)[\eta,\eta]
 +o(\varepsilon_L)\\
 &+O(A_L^{-1}\varepsilon_L)
 +O(A_L^{-1}\varepsilon_L^2).
\end{align*}
Here the third term is precisely the contribution of
$(\Psi_L'(s_L)-1)D^2s_L[\eta,\eta]$.  Since $A_L^{-1}\to0$,
\[
 A_L^{-1}\varepsilon_L=o(\varepsilon_L),
 \qquad
 A_L^{-1}\varepsilon_L^2=o(\varepsilon_L),
\]
and therefore
\[
 D^2\mathcal K_L(y_*)[\eta,\eta]
 =-\varepsilon_LD^2G(y_*)[\eta,\eta]+o(\varepsilon_L)<0
\]
for large $L$, because $\varepsilon_L>0$ and
$D^2G(y_*)[\eta,\eta]>0$.  The function $\mathcal K_L$ is $C^2$ near
$y_*$ for these $L$, so its convexity requires its Hessian there to be
positive semidefinite.  This contradiction proves (ii).

For (iii), fix $y\in U$ and suppose that
$D^2F(y)[\eta,\eta]=0$ for some $\eta\ne0$.  At this fixed point,
\begin{align*}
 D^2s_L(y)[\eta,\eta]
 &=-\varepsilon_LD^2G(y)[\eta,\eta]+o(\varepsilon_L),\\
 \partial_\eta s_L(y)
 &=\partial_\eta F(y)+O(\varepsilon_L)=O(1).
\end{align*}
Applying the full chain rule \eqref{eq:transfer-full-chain-rule} gives
\begin{align*}
 D^2\mathcal K_L(y)[\eta,\eta]
 ={}&-\varepsilon_LD^2G(y)[\eta,\eta]
 +o(\varepsilon_L)\\
 &+O(A_L^{-1}\varepsilon_L)+O(A_L^{-1}).
\end{align*}
The hypothesis $A_L^{-1}=o(\varepsilon_L)$ makes both displayed
$A_L^{-1}$ terms negligible compared with $\varepsilon_L$.  Hence
\[
 D^2\mathcal K_L(y)[\eta,\eta]
 =-\varepsilon_LD^2G(y)[\eta,\eta]+o(\varepsilon_L)<0
\]
for large $L$, because $\varepsilon_L>0$ and
$D^2G(y)[\eta,\eta]>0$.  This again contradicts the positive
semidefiniteness of the Hessian of the locally $C^2$ convex function
$\mathcal K_L$.  Thus $D^2F(y)>0$ for every $y\in U$.
\end{proof}

We apply the abstract transfer principle to the finite-part expansion
obtained above.

\begin{theorem}
\label{thm:real-doubling}
Let $\sigma>0$, $N\ge2$, $N>4\sigma$, and let $\Omega$ be a bounded open
convex set.
\begin{enumerate}[label=\textup{(\roman*)},leftmargin=2.8em]
\item If $y_*$ is a critical point of $R_{2\sigma,\Omega}$ and
\begin{equation}\label{eq:real-center-input}
 D^2R_{\sigma,\Omega}(y_*)>0,
\end{equation}
then
\begin{equation}\label{eq:real-center-output}
 D^2R_{2\sigma,\Omega}(y_*)>0.
\end{equation}
\item If
\begin{equation}\label{eq:real-global-input}
 D^2R_{\sigma,\Omega}(y)>0
 \qquad\text{for every }y\in\Omega
\end{equation}
and $N>6\sigma$, then
\begin{equation}\label{eq:real-global-output}
 D^2R_{2\sigma,\Omega}(y)>0
 \qquad\text{for every }y\in\Omega.
\end{equation}
\end{enumerate}
\end{theorem}

\begin{proof}
By \cref{prop:real-J-convex},
$J_{\sigma,L}^{q_\sigma}$ is convex with
\[
 q_\sigma:=-\frac1{N-4\sigma}<0.
\]
By \cref{prop:real-finitepart-expansion}, on each $K\Subset\Omega$,
\[
 J_{\sigma,L}
 =A_L-R_{2\sigma,\Omega}
 +\varepsilon_L B_{N,\sigma}R_{\sigma,\Omega}
 +o_{C^2(K)}(\varepsilon_L),
\]
where
\[
 A_L:=C_{N,\sigma}L^\alpha,
 \qquad
 \varepsilon_L:=L^{-\beta},
 \qquad
 \alpha=\frac{N-4\sigma}{N-2\sigma},
 \qquad
 \beta=\frac{2\sigma}{N-2\sigma}.
\]
The stated remainder is $o_{C^2(K)}(\varepsilon_L)$ because it is
$O_{C^2(K)}(L^{-\beta-1/m})$.  Apply
\cref{prop:real-Hessian-transfer-abstract} with
\[
 F=R_{2\sigma,\Omega},
 \qquad
 G=B_{N,\sigma}R_{\sigma,\Omega}.
\]
Since $B_{N,\sigma}>0$, assertion (i) follows immediately from part (ii) of
the abstract transfer.

For assertion (ii), it remains only to check the scale separation.  One has
\[
 \frac{A_L^{-1}}{\varepsilon_L}
 =C_{N,\sigma}^{-1}L^{-(\alpha-\beta)},
 \qquad
 \alpha-\beta
 =\frac{N-6\sigma}{N-2\sigma}.
\]
Thus $A_L^{-1}=o(\varepsilon_L)$ exactly when $N>6\sigma$.  Part (iii) of
\cref{prop:real-Hessian-transfer-abstract} gives
\eqref{eq:real-global-output}.
\end{proof}

We complete the proof by iterating the doubling principle from a
fractional order to the prescribed spectral order.

By \cref{thm:robin-real-main}, for every $s>0$ with $N>2s$ the Robin
function has a unique critical point $y_{s,\Omega}$.  We now transfer the
fractional rigidity to this point.

\begin{proof}[Proof of \cref{thm:real-order-nondegeneracy-intro}]
For $0<s<1$, both conclusions follow from
\cref{thm:fractional-hessian-intro}.  Assume henceforth that $s\ge1$.
Choose an integer $k\ge1$ such that
\begin{equation}\label{eq:real-order-base-order}
 \sigma_0:=\frac{s}{2^k}\in(0,1),
\end{equation}
and set
\[
 \sigma_j:=2^j\sigma_0=\frac{s}{2^{k-j}},
 \qquad j=0,1,\ldots,k.
\]
Thus
\[
 \sigma_k=s,
 \qquad
 \sigma_{k-1}=\frac s2.
\]
Since $0<\sigma_0<1$, \cref{prop:main} gives the initial positivity
\begin{equation}\label{eq:real-order-base-positive}
 D^2R_{\sigma_0,\Omega}(y)>0
 \qquad\text{for every }y\in\Omega.
\end{equation}

We first propagate \emph{global} positivity only up to the order $s/2$.
For $j=0,\ldots,k-2$ one has
\[
 \sigma_j\le\frac s4.
\]
The dimensional hypothesis $N>2s$ therefore implies
\begin{equation}\label{eq:real-order-intermediate-dim}
 N>2s>\frac{3s}{2}\ge6\sigma_j.
\end{equation}
Thus the global-positivity statement in \cref{thm:real-doubling} applies successively to
\[
 \sigma_0\to\sigma_1\to\cdots\to\sigma_{k-1}=\frac s2,
\]
and gives
\begin{equation}\label{eq:half-order-global}
 D^2R_{s/2,\Omega}(y)>0
 \qquad\text{for every }y\in\Omega.
\end{equation}
If $k=1$,
no intermediate transfer is required, and
\eqref{eq:half-order-global} is simply the base assertion
\eqref{eq:real-order-base-positive}.

For the final step, apply the critical-point statement in
\cref{thm:real-doubling} with $\sigma=s/2$ and with
$y_*=y_{s,\Omega}$.  Its dimensional requirement is
\[
 N>4\frac s2=2s,
\]
which is exactly the natural hypothesis of the theorem.  Moreover,
\eqref{eq:half-order-global} gives in particular
\[
 D^2R_{s/2,\Omega}(y_{s,\Omega})>0.
\]
This uses global positivity at order $s/2$ and does not require
$y_{s,\Omega}$ to be the center of $R_{s/2,\Omega}$.
Since $y_{s,\Omega}$ is a critical point of $R_{s,\Omega}$, the critical-point
transfer yields \eqref{eq:real-order-center-positive}.  No scale separation is
needed in this last step; this is why the nondegeneracy conclusion at the critical point requires no additional
dimensional restriction beyond $N>2s$.

Finally suppose $N>3s$.  Then for every $j=0,\ldots,k-1$,
\[
 6\sigma_j\le6\sigma_{k-1}=3s<N.
\]
Hence the global-positivity statement in \cref{thm:real-doubling} applies at every step,
including the last one $s/2\to s$.  Starting from
\eqref{eq:real-order-base-positive}, induction gives
$D^2R_{s,\Omega}>0$ throughout $\Omega$, proving
\eqref{eq:real-order-global-positive}.
\end{proof}

\begin{remark}
Taking $s=p\in\mathbb N$ in \cref{thm:real-order-nondegeneracy-intro} gives the
Navier polyharmonic conclusion under the interpretation in
\cref{rem:integer-navier-interpretation}.  In particular, for every $N>2p$
the unique Navier Robin center is nondegenerate, without any symmetry
assumption on the convex domain.
\end{remark}

\begin{remark}
The distinction between the two dimensional conditions follows from
the chain rule in
\cref{prop:real-Hessian-transfer-abstract}.  At a critical point the
gradient-square term is $O(A_L^{-1}\varepsilon_L^2)$, so the doubling step
needs only $N>4\sigma$.  At a general point it is $O(A_L^{-1})$ and must
be smaller than the second finite-part scale, which requires $N>6\sigma$.
Thus the iteration yields center nondegeneracy for $N>2s$ and global
positivity for $N>3s$.  The condition $N>3s$ is sufficient for the present argument; its
optimality is not addressed.
\end{remark}

\begin{remark}
The $C^{2,\vartheta}$ boundary assumption enters only through the
fractional identity: the heat-kernel estimate uses two spatial boundary
derivatives, and the curvature calculation uses the normal and its first
tangential derivative.  The finite-part argument requires no additional boundary regularity.
Thus the same assumption suffices at every spectral order.
We do not claim it is optimal.
\end{remark}

\appendix
\section{Proof of the boundary heat-kernel derivative estimate}
\label{app:boundary-heat}

This appendix contains the proof of \cref{lem:boundary-heat}.  We separate
it from the main fractional Hessian argument because its role is purely
regularity-theoretic: it provides the two boundary derivatives required to
justify the translation--curvature identity under the boundary regularity
assumption $C^{2,\vartheta}$.

\begin{proof}[Proof of \cref{lem:boundary-heat}]
Since $\partial\Omega$ is compact and of class $C^{2,\vartheta}$, there are
$r_\partial>0$ and $\Lambda_\partial<\infty$ such that, after a rigid motion
at each $\xi\in\partial\Omega$,
\[
 \Omega\cap\bigl(B'_{4r_\partial}\times(-4r_\partial,4r_\partial)\bigr)
 =\{(x',x_N):x'\in B'_{4r_\partial},
       |x_N|<4r_\partial,\ x_N>\varphi_\xi(x')\},
\]
with
\[
 \varphi_\xi(0)=0,\qquad D\varphi_\xi(0)=0,\qquad
 \|\varphi_\xi\|_{C^{2,\vartheta}(B'_{4r_\partial})}
 \le\Lambda_\partial
\]
uniformly in $\xi$.  Fix
\[
 0<\rho_0<\min\{d/16,r_\partial/16\}
\]
small enough that the above charts also cover the closed $4\rho_0$ boundary
collar.  If $\dist(x,\partial\Omega)<4\rho_0$ and $y\in K$, then
\[
 |x-y|\ge d-4\rho_0\ge\frac{3d}{4}.
\]
Hence the killed-kernel domination by the free heat kernel gives
\begin{equation}\label{eq:heat-zero-bound}
 0\le p_\Omega(t,x,y)
 \le Ct^{-N/2}e^{-c/t},
 \qquad0<t\le1,
\end{equation}
uniformly in this collar and in $y\in K$.

We first differentiate in the pole variable.  Choose
$K\Subset K_1\Subset\Omega$ with
$\dist(K_1,\partial\Omega)>3d/4$.  For fixed $x$ in the collar,
$(t,y)\mapsto p_\Omega(t,x,y)$ is caloric in the $y$ variable.  A standard
scale-invariant interior parabolic derivative estimate on a cylinder of
spatial radius comparable to $\sqrt t$ gives, for $|\mu|\le1$ and all
sufficiently small $t$,
\begin{equation}\label{eq:y-int-est}
 |D_y^\mu p_\Omega(t,x,y)|
 \le Ct^{-|\mu|/2}
 \sup_{\substack{t/2\le s\le3t/2\\
                  y'\in B_{c_0\sqrt t}(y)}}p_\Omega(s,x,y').
\end{equation}
After reducing the small-time threshold, all such $y'$ lie in $K_1$ and
$|x-y'|\ge d/2$.  Thus \eqref{eq:heat-zero-bound}, with modified constants,
yields
\begin{equation}\label{eq:y-der-small}
 |D_y^\mu p_\Omega(t,x,y)|
 \le Ct^{-(N+|\mu|)/2}e^{-c/t}.
\end{equation}

Fix $y\in K$ and write
\[
 v_\mu(t,x):=D_y^\mu p_\Omega(t,x,y).
\]
Then
\[
 (\partial_t-\Delta_x)v_\mu=0
 \quad\text{in }(0,\infty)\times\Omega.
\]
It also has zero lateral trace.  For $|\mu|=0$ this is the Dirichlet boundary
condition.  For $|\mu|=1$, use the semigroup identity
\[
 D_y^\mu p_\Omega(t,x,y)
 =\int_\Omega p_\Omega(t/2,x,z)
 D_y^\mu p_\Omega(t/2,z,y)\,\dd z.
\]
The second factor is bounded uniformly in $z$ because the pole $y$ stays in
$K$, while the first factor tends to zero as $x\to\partial\Omega$; domination
by the fixed-time free Gaussian permits dominated convergence.  Hence
$v_\mu(t,\xi)=0$ for every $t>0$ and $\xi\in\partial\Omega$.

We now make the boundary estimate explicit.  In a chart centered at $\xi\in\partial\Omega$, we suppress the subscript
$\xi$ and set
\[
 \Psi(z',z_N):=(z',z_N+\varphi(z')),
 \qquad
 u(t,z):=v_\mu(t,\Psi(z)).
\]
Since
\[
 \partial_{x_i}=\partial_{z_i}-\varphi_i\partial_{z_N}
 \quad(i<N),\qquad
 \partial_{x_N}=\partial_{z_N},
\]
a direct computation gives
\begin{align}
 \Delta_xv_\mu
 ={}&\sum_{i<N}u_{ii}-2\sum_{i<N}\varphi_i u_{iN}
 +(1+|D\varphi|^2)u_{NN}-(\Delta'\varphi)u_N.
 \label{eq:flattened-laplacian}
\end{align}
Thus $u$ solves
\begin{equation}\label{eq:flattened-parabolic}
 u_t-a^{ij}(z')u_{ij}-b^i(z')u_i=0
\end{equation}
in a flat half-cylinder, with $u=0$ on the flat lateral face, where
\[
 a^{ij}=\delta_{ij}\ (i,j<N),\qquad
 a^{iN}=a^{Ni}=-\varphi_i,\qquad
 a^{NN}=1+|D\varphi|^2,
\]
and
\[
 b^N=-\Delta'\varphi,\qquad b^i=0\quad(i<N).
\]
For every $\zeta=(\zeta',\zeta_N)$,
\begin{equation}\label{eq:ellipticity-identity}
 a^{ij}\zeta_i\zeta_j
 =|\zeta'-D\varphi\,\zeta_N|^2+\zeta_N^2.
\end{equation}
Consequently the ellipticity constants are uniform in the boundary chart.
Moreover
\[
 \|a\|_{C^{0,\vartheta}}+\|b\|_{C^{0,\vartheta}}
 \le C(N,\Lambda_\partial).
\]
Only $D\varphi$ and $D^2\varphi$ enter these coefficients.

Fix a small universal $\varepsilon>0$ and set $r=\varepsilon\sqrt t$.
Choose $\varepsilon$ so that $4r^2\le t/2$ and the spatial cylinder of
radius $4r$ remains inside the chart.  Denote by $\mathcal Q_{4r}^+$ the
corresponding flattened half-cylinder with time interval
$(t-4r^2,t+4r^2)$ and spatial section $B_{4r}\cap\{z_N>0\}$.  With
\[
 W(T,Z):=u(t+r^2T,rZ)
\]
the rescaled equation is
\[
 W_T-A_r^{ij}(Z)W_{ij}-B_r^i(Z)W_i=0,
 \qquad
 A_r^{ij}(Z)=a^{ij}(rZ),\quad B_r^i(Z)=r b^i(rZ).
\]
The ellipticity constants are unchanged and
\[
 \|A_r\|_{C^{0,\vartheta}}+\|B_r\|_{C^{0,\vartheta}}
 \le C(N,\vartheta,\Lambda_\partial)
\]
uniformly in $r$, because
$[A_r]_{\vartheta}=r^\vartheta[a]_{\vartheta}$ and
$[B_r]_{\vartheta}=r^{1+\vartheta}[b]_{\vartheta}$.
We justify boundary regularity before applying the Schauder estimate.
At this stage $W$ is a classical solution in the open half-cylinder and
is continuous up to its flat lateral face, where it vanishes.  The
continuity assertion requires only the zeroth-order boundary continuity
of the Dirichlet heat kernel.  More explicitly, on any compact time
interval $[t_-,t_+]\Subset(0,\infty)$, fix $h\in(0,t_-)$ and use
\[
 v_\mu(\tau,x)
 =\int_\Omega p_\Omega(\tau-h,x,z)
 D_y^\mu p_\Omega(h,z,y)\,\dd z,
 \qquad t_-\le\tau\le t_+.
\]
The second factor is uniformly bounded in $z$ for the fixed interior
pole, by the interior parabolic derivative estimate in the pole variable
and Gaussian domination.  The first
factor is jointly continuous in $(\tau,x)$ up to the lateral boundary
and vanishes there.  It is bounded by a constant independent of $z$
because $\tau-h\ge t_--h>0$.  Since $|\Omega|<\infty$, dominated
convergence proves joint continuity of $v_\mu$ up to the lateral
boundary.  No spatial derivative at that boundary is used in this step.

In the rescaled coordinates write
\[
 Q_2^+:=(-2,2)\times(B_2(0)\cap\{Z_N>0\}),
 \qquad
 Q_1^+:=(-1,1)\times(B_1(0)\cap\{Z_N>0\}).
\]
Their closures, apart from the flat face, lie inside the larger
rescaled cylinder $(-4,4)\times(B_4(0)\cap\{Z_N>0\})$.
Prescribe the continuous trace of $W$ on the parabolic boundary of
$Q_2^+$: the initial face $T=-2$ and the spatial lateral boundary, with
no data prescribed at the terminal time.  On the flat lateral portion
these data are identically zero, and hence have the required parabolic
H\"older regularity there.  The coefficients $A_r^{ij}$ and $B_r^i$
are uniformly elliptic and H\"older continuous as verified above, and
the source term is zero.  The flat-boundary existence and regularity
result \cite[Theorem~5.13, pp.~93--94]{Lieberman1996}, applied locally
with these fixed cylinder proportions, therefore supplies a continuous
solution $\widetilde W$, classical in $Q_2^+$, with these boundary data
and with parabolic $C^{2+\vartheta,1+\vartheta/2}$ regularity locally
up to the open flat lateral face.  This regularity is asserted away
from the initial face and the edges where the flat and curved lateral
pieces meet; no regularity at those artificial corners is needed.  In
particular, the artificial boundary data are used only to obtain a continuous
comparison solution, whereas the Schauder regularity invoked below is local
up to the open homogeneous flat face, which is separated from those corners
inside $Q_1^+$.

Both $W$ and $\widetilde W$ solve the same equation and have the same
continuous parabolic boundary data.  The parabolic maximum principle,
applied to their difference and its negative, gives
$W=\widetilde W$ in $Q_2^+$.  Thus $W$ has the required boundary
regularity on $\overline{Q_1^+}$.  We may now apply the local
flat-boundary Schauder estimate
\cite[Theorem~4.22, p.~70]{Lieberman1996}, together with the interior
estimate, to obtain
\[
 \sup_{Q_1^+}(|D_ZW|+|D_Z^2W|)
 \le C\|W\|_{L^\infty(Q_2^+)}
 \le C\|W\|_{L^\infty((-4,4)\times B_4^+)},
 \qquad B_4^+:=B_4(0)\cap\{Z_N>0\}.
\]
The same estimate holds for the continuous derivative traces on the
flat face.  Its constant depends only on the dimension, the
ellipticity constants, the uniform $C^{0,\vartheta}$ coefficient
bounds and the fixed cylinder proportions.  It is therefore
independent of $t$, $r$, $\xi$ and $y\in K$.
Scaling back gives
\begin{equation}\label{eq:scaled-boundary-parabolic}
 |D_z^\gamma u(t,0)|
 \le Cr^{-|\gamma|}
 \sup_{\mathcal Q_{4r}^+}|u|,
 \qquad |\gamma|\le2.
\end{equation}
The inverse coordinate change satisfies
\[
 |D_xv_\mu|\le C|D_zu|,
 \qquad
 |D_x^2v_\mu|
 \le C\bigl(|D_z^2u|+\|D^2\varphi\|_\infty|D_zu|\bigr).
\]
Since $r\le1$, the lower-order $r^{-1}$ term is absorbed by $r^{-2}$.
Therefore, for $|\beta|\le2$,
\[
 |D_x^\beta v_\mu(t,\xi)|
 \le Cr^{-|\beta|}
 \sup_{\mathcal Q_{4r}^+}|v_\mu|.
\]
All points of the last cylinder have time coordinate in $[t/2,3t/2]$ and
remain in the fixed boundary collar.  Combining this estimate with
\eqref{eq:y-der-small} and $r\asymp\sqrt t$ proves
\eqref{eq:heat-small-boundary} on the boundary for all sufficiently small
$t$.  The same Schauder estimate controls every point in the smaller
half-cylinder, so it also gives the estimate when
$\dist(x,\partial\Omega)\le c\sqrt t$.  If instead
$\dist(x,\partial\Omega)>c\sqrt t$, an interior parabolic cylinder of radius
comparable to $\sqrt t$ is contained in $\Omega$ and the standard interior
derivative estimate gives the same bound.  Thus
\eqref{eq:heat-small-boundary} holds throughout the fixed collar for small
$t$.

On a compact interval $t\in[t_0,1]$, the same local boundary and interior
estimates apply with a fixed positive time scale and give uniform bounds for
all derivatives under consideration.  Increasing $C$ proves
\eqref{eq:heat-small-boundary} for all $0<t\le1$.  The local
Schauder estimate also shows that these $x$ derivatives extend continuously
to the lateral boundary.

For large time, the semigroup property gives, for $t\ge2$,
\begin{equation}\label{eq:large-semigroup}
 D_x^\beta D_y^\mu p_\Omega(t,x,y)
 =\Big\langle
 D_x^\beta p_\Omega(1,x,\cdot),
 e^{-(t-2)A}D_y^\mu p_\Omega(1,\cdot,y)
 \Big\rangle_{L^2(\Omega)}.
\end{equation}
For $s\in[1/2,3/2]$ the free-kernel bound gives
$p_\Omega(s,x,z)\le C$ uniformly in $x,z\in\Omega$.  Applying the same
boundary Schauder estimate in the $x$ variable therefore gives uniform fixed-time bounds for $D_x^\beta p_\Omega(1,x,z)$, including as the
second variable $z$ approaches the boundary.  The interior pole estimate gives the
corresponding $D_y^\mu$ bound for $y\in K$.  Since $\Omega$ is bounded,
these pointwise bounds imply
\[
 \sup_{\dist(x,\partial\Omega)\le\rho_0}
 \|D_x^\beta p_\Omega(1,x,\cdot)\|_{L^2(\Omega)}
 +\sup_{y\in K}
 \|D_y^\mu p_\Omega(1,\cdot,y)\|_{L^2(\Omega)}
 \le C.
\]
Since
\[
 \|e^{-(t-2)A}\|_{L^2\to L^2}=e^{-\lambda_1(\Omega)(t-2)},
\]
Cauchy--Schwarz proves \eqref{eq:heat-large-boundary} for $t\ge2$; the
compact interval $[1,2]$ is absorbed into the constant.  This completes the
proof.
\end{proof}

\renewcommand\refname{References}
\renewenvironment{thebibliography}[1]{%
\section*{\refname}
\list{{\arabic{enumi}}}{\def\makelabel##1{\hss{##1}}\topsep=0mm
\parsep=0mm\partopsep=0mm\itemsep=0mm
\labelsep=1ex\itemindent=0mm
\settowidth\labelwidth{\small[#1]}%
\leftmargin\labelwidth \advance\leftmargin\labelsep
\advance\leftmargin -\itemindent
\usecounter{enumi}}\small
\def\newblock{\ }
\sloppy\clubpenalty4000\widowpenalty4000
\sfcode`\.=1000\relax}{\endlist}

\end{document}